\documentclass{article}
\usepackage{amssymb}
\usepackage{hyperref}
\usepackage{amsthm}
\usepackage{mathrsfs}
\usepackage{mathtools}
\usepackage{tikz}
\usepackage{tikz-cd}
\usepackage{comment}
\usepackage[margin=1in]{geometry}
\usepackage{microtype}

\newcommand{\norm}[1]{\left\lVert#1\right\rVert}
\newcommand{\Q}{\mathbb{Q}}
\newcommand{\Z}{\mathbb{Z}}
\newcommand{\F}{\mathbb{F}}
\newcommand{\Qp}{\Q_{p}}
\newcommand{\Zp}{\Z_{p}}

\newcommand{\C}{\mathbb{C}}
\newcommand{\R}{\mathbb{R}}
\newcommand{\A}{\mathbb{A}}
\newcommand{\GL}{\mathrm{GL}}
\newcommand{\Tw}{\mathrm{tw}}
\newcommand{\Crit}{\mathrm{Crit}}
\newcommand{\vol}{\mathrm{vol}}
\newcommand{\Iw}{\mathrm{Iw}}
\newcommand{\Norm}{\mathrm{Norm}}
\newcommand{\Eis}{\mathrm{Eis}}
\newcommand{\br}{\mathrm{br}}

\newtheorem{theorem}{Theorem}[subsection]
\newtheorem{lemma}[theorem]{Lemma}

\newtheorem{proposition}[theorem]{Proposition}
\newtheorem{corollary}[theorem]{Corollary}

\theoremstyle{definition}
\newtheorem{definition}[theorem]{Definition}
\newtheorem{example}[theorem]{Example}
\newtheorem{remark}[theorem]{Remark}
\newtheorem*{theoremA}{Theorem A}
\newtheorem*{theoremB}{Theorem B}

\title{Finite-slope $p$-adic $L$-functions for $\mathrm{GL}_{3}$}
\author{Xenia Dimitrakopoulou and Rob Rockwood}
\date{}

\begin{document}

\maketitle

\begin{abstract}
Let $\Pi$ be a regular algebraic cuspidal automorphic representation of $\GL_{3}(\A_{\Q})$ of weight $(a, 0, -a)$, and let $p$ be a prime. Generalising work of Loeffler--Williams in the nearly-ordinary case, we construct $p$-adic $L$-functions for each small slope regular $P_{1}$-refinement of $\Pi$ at $p$ consistent with the conjectures of Coates--Perrin-Riou and Panchishkin. 
\end{abstract}

\section{Introduction}

Let $p$ be a prime and let $\Pi$ be a regular algebraic cuspidal automorphic representation (RACAR) of $\GL_{3}(\A_{\Q})$. A number of conjectures, most notably those of Coates--Perrin-Riou \cite{CPR89, coates89} and their refinement by Panchishkin \cite{panchishkin94}, predict that whenever $\Pi$ satisfies a suitable ordinarity condition at $p$, the critical values of the twisted $L$-functions $L(\Pi \times \eta, s)$, for $\eta$ of $p$-power conductor, are interpolated by a bounded $p$-adic measure on $\Zp^{\times}$,  the $p$-adic $L$-function of $\Pi$. Such $p$-adic $L$-functions are fundamental objects in Iwasawa theory, and for $\GL_{1}$ and $\GL_{2}$ their existence has been known for decades. For $n \geq 3$, the first (and, to date, only) construction applying to RACARs of $\GL_{n}(\A_{\Q})$ of general type is the recent work of Loeffler--Williams \cite{loefflerWilliams2021p}. For $\Pi$ of weight $\lambda = (a, 0, -a)$ which is nearly ordinary at $p$ for the maximal standard parabolic $P_{1} \subset \GL_{3}$ with Levi $\GL_{1} \times \GL_{2}$, Loeffler--Williams construct a bounded measure $L_{p}^{-}(\Pi)$ on $\Zp^{\times}$ interpolating all the critical values $L(\Pi \times \eta, -j)$, $0 \leq j \leq a$, in the left half of the critical strip. Moreover, and dually, assuming near-ordinarity for the parabolic $P_{2}$ with Levi $\GL_{2} \times \GL_{1}$, Loeffler--Williams construct a measure $L_{p}^{+}(\Pi)$ interpolating the right half of the critical strip. This proves the conjectures of Coates--Perrin-Riou and Panchishkin for RACARs of $\mathrm{GL}_{3}$.

In the present paper we remove the ordinarity hypothesis, replacing it with a \emph{small slope} hypothesis. This is not a straightforward generalisation; our construction utilises the substantive ideas developed by the second author in \cite{rockwood2022spherical} and is the first case of a $p$-adic $L$-function constructed using these methods. As in \cite{loefflerWilliams2021p} we do not assume $\Pi_{p}$ is unramified: a $P_{1}$-refinement of $\Pi_{p}$ is a choice of irreducible $\mathrm{GL}_{1} \times 
\mathrm{GL}_{2}$ subrepresentation $\sigma_{p} \times \sigma_{p}'$ of the Jacquet module $J_{P_{1}}(\Pi_{p})$; after a twist we may assume the character $\sigma_{p}$ is unramified, so the refinement is determined by $\alpha_{p} = \sigma_{p}(p)$. The refined representation $\tilde{\Pi} = (\Pi, \alpha_{p})$ contributes to Betti cohomology of the locally symmetric space for $\GL_{3}$, where the associated eigenvalue of the normalised Hecke operator $U_{p, 1}$ is $p^{a+1}\alpha_{p}$. We write $h = v_{p}(p^{a+1}\alpha_{p})$ for its slope, so that $h = 0$ is the $P_{1}$ nearly-ordinary case. We say $\tilde{\Pi}$ has \emph{small slope} if $h < a + 1$. Our main result is the following (see Theorem \ref{thm:main} and Propositions \ref{prop:unique} and \ref{prop:plus}):

\begin{theoremA}
Let $\tilde{\Pi} = (\Pi, \alpha_{p})$ be a regular $P_{1}$-refinement of a RACAR of $\GL_{3}(\A_{\Q})$ of weight $(a, 0, -a)$, of small slope $h < a+1$. Then there is a $p$-adic distribution $L^{-}_{p}(\tilde{\Pi})$ on $\Zp^{\times}$, of growth order $\mathcal{O}(\log_{p}^{h})$, such that for every Dirichlet character $\eta$ of $p$-power conductor and every $0 \leq j \leq a$ with $(-1)^{j} = \omega_{\Pi}\eta(-1)$ we have
\[
    \int_{\Zp^{\times}} \eta^{-1}(x) x^{-j} \, dL^{-}_{p}(\tilde{\Pi})(x) = e_{\infty}(\Pi_{\infty} \times \eta_{\infty}, -j) \, e_{p}(\Pi_{p} \times \eta_{p}, -j) \cdot \frac{L^{(p)}(\Pi \times \eta, -j)}{\Omega^{-}_{\Pi}},
\]
where $e_{\infty}$ and $e_{p}$ are the modified Euler factors of Coates--Perrin-Riou, $\Omega^{-}_{\Pi} \in \C^{\times}$ is a period, and if $(-1)^{j} = -\omega_{\Pi}\eta(-1)$ we have 
\[
     \int_{\Zp^{\times}} \eta^{-1}(x) x^{-j} \, dL^{-}_{p}(\tilde{\Pi})(x) = 0.
\]
The distribution $L^{-}_{p}(\tilde{\Pi})$ is uniquely determined by this interpolation property, and if $h = 0$ it coincides with the measure of \cite{loefflerWilliams2021p}.
\end{theoremA}

\begin{remark}
In the above interpolation formula, if $e_p(\Pi_p\times\eta_p,s)$ has a pole at $s=-j$ (which happens only for $\Pi_p=\mathrm{St}_3$, $\eta=1$ and $j=1$, see Remark \ref{rem:pole}), the product $e_p(\Pi_p\times\eta_p,-j)\,L^{(p)}(\Pi\times\eta,-j)$ is to be read as $\lim_{s\to-j}e_p(\Pi_p\times\eta_p,s)\,L^{(p)}(\Pi\times\eta,s)=(1-p^{-1})\,L(\Pi,-1)$.
\end{remark}

Moreover, we show:

\begin{theoremB}
If instead $\tilde{\Pi}$ is a small slope regular $P_{2}$-refinement of a RACAR of $\mathrm{GL}_{3}(\A)$, then there is a distribution $L^{+}_{p}(\tilde{\Pi})$ on $\Zp^{\times}$, of the same growth order, interpolating the critical values $L(\Pi \times \eta^{-1}, j+1)$, $0 \leq j \leq a$, in the right half of the critical strip.
\end{theoremB}

\subsection{Outline of the construction}

We briefly recall the strategy of \cite{loefflerWilliams2021p} as we work within the same framework. Building on work of Mahnkopf \cite{mahnkopf98, mahnkopf00}, Loeffler--Williams construct a tower of classes in $H^{3}$ of the locally symmetric space for $\GL_{3}$ with coefficients in group rings $\Zp[(\Z/p^{n}\Z)^{\times}]$, obtained by pushing forward Be\u{\i}linson's motivic Eisenstein classes for $\GL_{2}$ along a twisted embedding $H = \GL_{2} \times \GL_{1} \hookrightarrow \GL_{3}$. Using the theory of pushforwards between `spherical pairs' formalised in  \cite{loefflerspherical}, these classes satisfy a norm compatibility up to multiplication by the Hecke operator $U'_{p, 1}$. Pairing with a $P_{1}$-nearly-ordinary eigenclass $\phi$ then produces (up to an explicit volume factor, which we suppress) a group ring element ${}_{c}\tilde{\Xi}_{n}^{[a; j]} \in \mathcal{O}[\left(\Z/p^{n}\Z\right)^{\times}]$ (where $a$ is a weight parameter and $c$ is an integral `smoothing factor') for each critical twist $0 \leq j \leq a$ and $n \geq 1$, with $\mathcal{O}/\Zp$ some finite extension. The (nearly)-ordinary hypothesis allows the authors to integrally normalise these classes to be norm compatible on the nose and thus their inverse limits determine $p$-adic measures for each $0 \leq j \leq a$. The compatibility of these $a + 1$ measures as $j$ varies is proved using the $p$-adic interpolation of branching laws of \cite{loeffler2021spherical}.

When we pass from nearly ordinary to small-slope refinements we can no longer integrally normalise the classes ${}_{c}\tilde{\Xi}^{[a; j]}_{n}$ to be norm-compatible as they now experience unbounded growth as $n$ increases. The classical interpolation method underlying the constructions of Amice--V\'elu \cite{amicevelu} and Vi\v{s}ik \cite{visik} in the case of $\GL_{2}$ provides a way to handle this growth. In our setting, congruences between the ${}_{c}\tilde{\Xi}^{[a; j]}_{n}$ yield bounds on the corresponding local moments. Together with norm compatibility, these bounds allow us to construct a locally analytic distribution of growth $\mathcal{O}(\log_{p}^{h})$, where $h$ is the $p$-adic valuation of the $U_{p}$-eigenvalue of $\phi$.

The technical heart of this paper is therefore to prove such congruences for the Betti--Eisenstein classes of \cite{loefflerWilliams2021p}. This is where our approach diverges from \textit{op.\ cit.}, requiring genuinely new techniques. Rather than the ordinary-projector techniques of \cite{loeffler2021spherical}, we adapt the finite-slope methods introduced by the second author in \cite{rockwood2022spherical}. The key input is an explicit congruence (Proposition \ref{prop:cong}) between branching vectors in highest weight representations of $\mathrm{GL}_{3}$.  These congruences translate into congruences between the finite-level classes $\tilde{\Xi}^{[a;j]}_{n}$ (Proposition \ref{prop:classcong}), allowing us to deduce the existence of the desired distribution (Theorem \ref{thm:growth}). The computation of its values against characters $\eta(x)x^{j}$ is a global Rankin--Selberg integral, the computation of which is recalled from \cite[\S 8]{loefflerWilliams2021p} whose rather involved computations are agnostic to slope considerations. In particular, the identification of the archimedean interpolation factor is inherited wholesale from the comparison with symmetric square $p$-adic $L$-functions carried out in \textit{op.\ cit.}, \S 9.5.

\subsection{Structure of the paper}

In Section \ref{sec:groups}, we set up the relevant algebraic groups, representations, integral lattices and branching laws, and prove the congruences underlying the whole construction. Section \ref{sec:auto} collects the automorphic preliminaries: Whittaker models, cohomological representations, critical values, the Coates--Perrin-Riou interpolation factors, Hecke operators, $P$-refinements and slopes, and the local newvector theory for $P_{1}$-refinements. Section \ref{sec:lss} treats locally symmetric spaces, local systems and their cohomology, including the decomposition into connected components and the descent of coefficient twists to group rings. Section \ref{sec:eis} recalls Eisenstein series and Be\u{\i}linson's Eisenstein classes together with their $p$-adic interpolation due to Kings. Section \ref{sec:analytic} recalls the theory of $p$-adic distributions and the interpolation results we require. In Section \ref{sec:pairing}, we construct the tempered distributions $\Xi^{[a]}_{\infty}(\phi)$ attached to finite-slope eigenclasses, and in Section \ref{sec:comppair} we compute their values as Rankin--Selberg integrals, following \cite{loefflerWilliams2021p}. Section \ref{sec:padicL} assembles the $p$-adic $L$-function and proves the interpolation formula, and Section \ref{sec:duality} deduces the right-half $p$-adic $L$-function by duality.

\subsection*{Acknowledgements}
We would like to thank David Loeffler and Chris Williams for helpful conversations. 

\subsection*{Notation}
Throughout, $p$ is a prime, which for notational simplicity we assume to be odd. We write $\A = \A_{\Q}$ for the ad\`eles of $\Q$, $\A_{f}$ for the finite ad\`eles, and $\A_{f}^{(p)}$ (resp.\ $\A^{(S)}$) for the ad\`eles away from $p$ (resp.\ away from a finite set $S$). For a $p$-adic field $L,$ we write $\mathcal{O} = \mathcal{O}_{L}$ for its ring of integers.

\section{Algebraic groups and representations} \label{sec:groups}

We define the algebraic groups, characters and representations that we will use in the article.

\subsection{Algebraic groups} \label{sec:algp}

Let $G = \GL_{3} \times \GL_{1}$ and $H = \GL_{2} \times \GL_{1}$, considered as reductive group schemes over $\Z$, and write $\iota: H \hookrightarrow G$ for the embedding
\[
    \begin{psmallmatrix}
        a & b \\ c & d
    \end{psmallmatrix} \times (e) \mapsto \begin{psmallmatrix}
        a & b & \\
        c & d & \\
          &   & e
    \end{psmallmatrix} \times (e).
\]
We will use the following subgroups:
\begin{itemize}
    \item For $\GL_{3}$: the upper-triangular Borel subgroup $B_{\GL_{3}}$, with diagonal maximal torus $T = T_{\GL_{3}}$ and unipotent radical $N_{\GL_{3}}$, and the two maximal standard parabolic subgroups
    \[
        P_{1} = \begin{psmallmatrix}
            * & * & * \\
             & * & * \\
             & * & *
            \end{psmallmatrix}, \qquad P_{2} = \begin{psmallmatrix}
            * & * & * \\
            * & * & * \\
             &  & *
            \end{psmallmatrix},
    \]
    with Levi subgroups $L_{1} \cong \GL_{1} \times \GL_{2}$ and $L_{2} \cong \GL_{2} \times \GL_{1}$ and unipotent radicals $N_{1}, N_{2}$. For $i = 1, 2$ we define the opposite parabolics $\bar{P}_{i} = \bar{N}_{i}L_{i}$. These extend to parabolic subgroups $P_{i}^{G} = P_{i} \times \mathrm{GL}_{1} \subset G$ with Levi factor $L^{G}_{i} := L_{i} \times \mathrm{GL}_{1}$ and unipotent radical $N_{i} \times \{1\}$ (which we identify with $N_{i}$). 
    \item For $H$: the upper-triangular Borel subgroup $B_{H} = B_{\GL_{2}} \times \GL_{1}$, with maximal torus $T_{H} = T_{\GL_{2}} \times \GL_{1}$ and unipotent radical $N_{H} = N_{\GL_{2}} \times \{1\}$. Note that $\iota$ identifies $T_{H}$ with $T$ after projection to $\mathrm{GL}_{3}$.
\end{itemize}

We define a basis of characters $\nu_{i} \in X^{\bullet}(H) := \mathrm{Hom}(H, \mathbb{G}_{m})$, $i = 1, 2$, by
\[
    \nu_{1}((A, z)) = \frac{\det A}{z}, \qquad \nu_{2}((A, z)) = z,
\]
so that $(\nu_{1}, \nu_{2})$ defines an isomorphism $H/H^{\mathrm{der}} \cong \GL_{1} \times \GL_{1}$. These characters extend (under $\iota$) to a basis of $X^{\bullet}(G)$, which we again denote $\nu_{1}, \nu_{2}$, given by
\[
    \nu_{1}((B, t)) = \frac{\det B}{t^{2}}, \qquad \nu_{2}((B, t)) = t,
\]
defining an isomorphism $G/G^{\mathrm{der}} \cong \GL_{1} \times \GL_{1}$.

\begin{definition} \label{def:QH0}
    Define a subgroup scheme
    \[
        Q_{H}^{0} = \left\{\begin{psmallmatrix}
                x & *  \\ & 1  \\
                \end{psmallmatrix} \times (z)\right\} \subset H,
    \]
    a \emph{mirabolic subgroup} in the language of \cite{loefflerspherical}.
\end{definition}

\begin{definition} \label{def:u}
    As in  \cite[\S 1.4]{loefflerWilliams2021p}, let
    \[
        u := \begin{psmallmatrix}
        1 &  & 1\\
        & 1 & \\
        & & 1
    \end{psmallmatrix}\begin{psmallmatrix}
        1 & & \\
        & & -1\\
        & 1 &
    \end{psmallmatrix} \in \GL_{3}(\Z),
    \]
    regarded as an element of $G(\Z)$ with trivial $\GL_{1}$-component.
\end{definition}

The following fundamental computation shows that $Q_{H}^{0}$ has an open orbit on the flag variety $\bar{P}^{G}_{1} \backslash G$, with base point $\bar{P}^{G}_{1}u^{-1}$ i.e. the pair $(G, H)$ is \emph{spherical}. 

\begin{lemma} \label{lem:opencell}
    The multiplication map $\mu: \bar{P}_{1}^{G} \times Q_{H}^{0} \to G$, $(\bar{p}, q) \mapsto \bar{p}u^{-1}q$, is smooth with image an open subscheme $U_{\mathrm{Sph}} := \bar{P}_{1}^{G}u^{-1}Q_{H}^{0}$ of $G$. Moreover
    \[
        Q_{H}^{0} \cap u\bar{P}^{G}_{1}u^{-1} = \left\{\begin{psmallmatrix}
            x & & \\ & 1 & \\ & & x
        \end{psmallmatrix} \times (x)\right\} \subset \ker(\nu_{1}).
    \]
\end{lemma}
\begin{proof}
    For $q \in Q_{H}^{0}$ with parameters $(x, y, z)$ as in Definition \ref{def:QH0}, a direct computation gives
    \begin{equation} \label{eq:uconj}
        u^{-1} \begin{psmallmatrix}
            x & y & \\ & 1 & \\ & & z
        \end{psmallmatrix} \times (z) u  = \begin{psmallmatrix}
            x & x - z & -y \\ & z & \\ & & 1
        \end{psmallmatrix} \times (z).
    \end{equation}
    This lies in $\bar{P}^{G}_{1}$ if and only if $x = z$ and $y = 0$, which gives the description of $Q_{H}^{0} \cap u\bar{P}_{1}^{G}u^{-1}$, which is obviously in the kernel of $\nu_{1}$. 

    If we can show that $\mu$ is smooth (in particular flat) then it follows from \cite[01UA]{stacks-project} that it is universally open and in particular has open image.   If we write the parameters of $\bar{P}_{1}$ as $\begin{psmallmatrix}
        1 & &  \\
        n_{21} & 1 &  \\
        n_{31} & & 1
    \end{psmallmatrix}\begin{psmallmatrix}
    \ell_{11} & & \\
     & \ell_{22} & \ell_{23} \\
     & \ell_{32} & \ell_{33}
    \end{psmallmatrix} \times (t)$ then by \cite[01V9(6)]{stacks-project} it suffices to show that the Jacobian of $\mu$ computed using the $10$ standard parameters of $G$ and the $11$ parameters $\ell_{ij}, n_{ij}, x, y ,z, t$ of $\bar{P}^{G}_{1} \times Q_{H}^{0}$ has an invertible $10 \times 10$  minor. Removing the $\ell_{11}$ column of the Jacobian matrix, one computes that this minor is $-z^{3}x^{2}\ell_{11}^{5}$, which is visibly invertible. 
\end{proof}

\subsection{Representations}

For any algebraic group $A$ write $X^{\bullet}(A) = \mathrm{Hom}_{\mathrm{Grp\text{-}sch}}(A, \mathbb{G}_{m})$ for its character group. We identify $X^{\bullet}(T) \cong \Z^{3}$ by letting $\chi_{i}$ be projection to the $i$th diagonal entry and sending $a\chi_{1} + b\chi_{2} + c\chi_{3} \mapsto (a, b, c)$. Pullback via $\iota$ gives a compatible  identification $X^{\bullet}(T_{H}) \cong \Z^{3}$.

The irreducible algebraic representations of $\GL_{3}$ are parameterised by dominant weights $\lambda = (a, b, c)$, $a \geq b \geq c$ and we write $V_{\lambda} = V^{\GL_{3}}_{\lambda}$ for the irreducible $\Qp$-linear representation of highest weight $\lambda$, always regarded as a representation of $G$ with trivial $\GL_{1}$-action. We say $\lambda$ is \emph{pure} if $a + c = 2b$; these are precisely the weights for which $V_{\lambda}^{\vee}$ is a twist of $V_{\lambda}$ by an element of $X^{\bullet}(\GL_{3})$, and every pure weight is equal to some $(a, 0, -a)$ modulo $X^{\bullet}(\GL_{3})$.

Similarly, the irreducible $\Qp$-linear representations of $H$ are parameterised by weights $\mu = (r, s; t)$ with $r \geq s$, and we write $V^{H}_{\mu}$ for the corresponding representation.

\subsection{Integral lattices} \label{sec:lattices}

Let $V^{A}_{\mu}$ be an irreducible algebraic representation of a split reductive $\Z$-group $A$, with a distinguished highest weight vector $v_{\mu}$. An \emph{admissible lattice} for $(V^{A}_{\mu}, v_{\mu})$ is a $\Z$-lattice $\mathcal{L} \subset V^{A}_{\mu}$ such that:
\begin{enumerate}
    \item the map $A_{/\Q} \to \GL(V^{A}_{\mu})$ extends to a map of $\Z$-group schemes $A \to \GL(\mathcal{L})$;
    \item the intersection of the highest weight space with $\mathcal{L}$ is $\Z \cdot v_{\mu}$.
\end{enumerate}
There are finitely many admissible lattices with a unique maximal and a unique minimal one (cf.\ \cite[\S 4.2]{LZGsp}). As in \cite{loefflerWilliams2021p} we let $V_{\lambda, \Z}$ denote the \emph{maximal} admissible lattice in $V_{\lambda}$, and $V^{H}_{\mu, \Z}$ the \emph{minimal} admissible lattice in $V^{H}_{\mu}$: this ensures that the branching maps of the next section carry $V^{H}_{\mu, \Z}$ into $V_{\lambda, \Z}$. For a commutative ring $R$ we set $V^{A}_{\mu, R} := V^{A}_{\mu, \Z} \otimes R$.

\subsection{Branching maps} \label{sec:branch}

The following classical branching law describes the restriction of $V_{\lambda}$ to $H$ for pure $\lambda$, see e.g.\ \cite[\S 8]{goodmanwallach}.

\begin{theorem} \label{thm:branching}
    Let $\lambda = (a, 0, -a)$ with $a \in \Z_{\geq 0}$. Then
    \[
        V_{\lambda}\big\vert_{H} \cong \bigoplus_{0 \leq i, j \leq a} V^{H}_{(j, -i; i - j)}.
    \]
\end{theorem}

Deftly wielding $\nu_{1}$ and $Q_{H}^{0}$ like a pair of salad tongs, we can isolate the highest weight vectors of the summands with $i = 0$.

\begin{proposition} \label{prop:faj}
    For each $0 \leq j \leq a$ there is a non-zero vector 
    \[
        f^{a, j} \in \left(V_{\lambda} \otimes \nu_{1}^{-j}\right)^{Q_{H}^{0}},
    \]
    unique up to scaling. The vector $f^{a, j}$ is the image of a highest weight vector of $V^{H}_{(j, 0; -j)}$ under the branching law.
\end{proposition}
\begin{proof}
   Existence and uniqueness follow from Theorem \ref{thm:branching}.
\end{proof}

For $0 \leq j \leq a$ these vectors define
$\Z_{p}$-linear, $Q_{H}^{0}$-equivariant maps
\begin{align*}
    \Zp &\to \left(V_{\lambda} \otimes \nu_{1}^{-j}\right)^{Q_{H}^{0}}, \\
    1 &\mapsto f^{a, j}.
\end{align*}
and we define 
\[
    \mathrm{br}^{[a, j]}: V^{H}_{(j, 0; -j)} \to V_{\lambda}
\]
to be the unique $H$-equivariant map sending a fixed highest weight vector $v_{j} \in  V^{H}_{(j, 0; -j)}$ to $f^{a, j}$.
In the following section we choose integral normalisations of the vectors $f^{a, j}$ suitable for $p$-adic interpolation. 

\subsubsection{The Borel--Weil presentation and normalisation of the $f^{a, j}$} \label{sec:BW}

Since $T_{H} = T$, the weight $\lambda = (a, 0, -a)$ determines an irreducible representation of the Levi $L_{1} = \GL_{1} \times \GL_{2}$, with maximal admissible lattice
\[
    W_{\lambda, \Zp} :=\mathrm{det}^{a} \boxtimes \mathrm{Sym}^{a}\Zp^{2} \otimes \mathrm{det}^{-a}.
\]
By the theorem of Borel--Weil we can identify
\begin{equation} \label{eq:BWmodel}
    V_{\lambda, \Zp} \otimes \nu_{1}^{-j} = \left\{f: G_{/\Zp} \to \mathbb{A}^{1}_{\Zp} \otimes W_{\lambda, \Zp}: f(\bar{n}\ell g) = \ell \cdot f(g)\nu_{1}^{-j} (\ell)\ \forall \bar{n} \in \bar{N}_{1}, \ell \in L_{1}^{G}\right\},
\end{equation}
for all $j \geq 0$, with $G$-action $(\gamma \cdot f)(g) = f(g\gamma)$. 

 Under Borel--Weil the elements $f^{a, j} \in V_{\lambda}$ defined in Proposition \ref{prop:faj} correspond to functions $G \to \mathbb{A}^{1}_{\Qp} \otimes W_{\lambda}$ satisfying $f^{a, j}(gq) = f^{a, j}(g)$ for $g \in G, q \in Q_{H}^{0}$.
By Lemma \ref{lem:opencell} such a function is uniquely determined by its restriction to the open cell $U_{\mathrm{Sph}} = \bar{P}^{G}_{1}u^{-1}Q_{H}^{0}$, i.e.\ by the single non-zero value $f^{a, j}(u^{-1}) \in W_{\lambda}$.

\begin{lemma} \label{lem:normalise}
    For every $0 \leq j \leq a$, the vector $f^{a, j}(u^{-1})$ is a non-zero element of the (one-dimensional) lowest weight space of $W_{\lambda}$. In particular, the line $\Qp \cdot f^{a,j}(u^{-1})$ is independent of $j$.
\end{lemma}

\begin{proof}
    The weights of $W_{\lambda}$ are of the form $(a,  - k,  k - a)$ for $0 \leq k \leq a$ with lowest weight $(a, -a, 0)$, so it suffices to show that $f^{a, j}(u^{-1})$ is fixed by the torus $\begin{psmallmatrix}
        t & & \\ & t & \\
         & & 1
    \end{psmallmatrix}$ which is an easy computation. 
\end{proof}

We now and forever normalise the $f^{a, j}$ so that
\begin{equation} \label{eq:normalisation}
    f^{a, j}(u^{-1}) = w_{\lambda}^{-} \in W_{\lambda, \Zp} \quad \text{for all } 0 \leq j \leq a,
\end{equation}
where $w_{\lambda}^{-}$ is a fixed generator of the lowest weight space of $W_{\lambda, \Zp}$. A simple adaptation of the argument of \cite[Proposition 3.2.6]{loeffler2021spherical} shows that with this normalisation $f^{a, j} \in V_{\lambda, \Zp}$.

\subsubsection{The normalised $\tau$-action and the congruence relation}
\label{sec:normtau}
 The character $\nu_{1}^{j}$ extends to a map $\nu_{1}^{j}: \bar{P}^{G}_{1}N_{1} \to \mathbb{G}_{m}$ sending $\bar{n} \ell n \mapsto \nu_{1}^{j}(\ell)$ for $\bar{n} \in \bar{N}_{1}, \ell \in L^{G}_{1}, n \in N_{1}$ and the product $\nu_{1}^{j}f^{a, j}$ is an element of $V_{\lambda, \Zp}$. We now show the vectors $\nu_{1}^{j}f^{a, j}$ satisfy a congruence as $j$ varies. This is the crucial relation which allows us to construct tempered $p$-adic distributions.

Recall $\tau := \tau_{1} = \begin{psmallmatrix} p & & \\ & 1 & \\ & & 1\end{psmallmatrix} \in \GL_{3}(\Qp)$ (cf.\ Section \ref{sec:heckeopsandprefs} below). Let $\eta_{\circ}: \mathbb{G}_{m} \to \GL_{3}$ denote the cocharacter $x \mapsto \mathrm{diag}(x, 1, 1)$, so $\eta_{\circ}(p) = \tau$. We define a normalised action of $\tau^{-1}$ on $V_{\lambda} \otimes \nu_{1}^{-j}$ by
\begin{equation} \label{eq:staraction}
    \tau^{-1} \ast v := p^{\langle \eta_{\circ}, \lambda - j\nu_{1}\rangle}\, \tau^{-1} \cdot v = p^{a - j}\, \tau^{-1} \cdot v.
\end{equation}
Since $\lambda$ is the highest weight, this action preserves any admissible lattice $V_{\lambda, \Zp} \otimes \nu_{1}^{-j}$ \cite[Lemma 2.3.4]{loeffler2021spherical}. In the Borel--Weil presentation \eqref{eq:BWmodel}, the restriction of $v$ to the big Bruhat cell $\bar{P}^{G}_{1}L_{1}N_{1}$ then satisfies
\[
    (\tau^{-n} \ast v)(\bar{q}n) = v(\bar{q}\,\tau^{n} n \tau^{-n}) \qquad \bar{q} \in \bar{P}^{G}_{1}, \ n \in N_{1}.
\]

\begin{proposition} \label{prop:cong}
    For all $0 \leq k \leq a$ and all $n \geq 1$ we have
    \[
        \sum_{j = 0}^{k}\binom{k}{j}(-1)^{j}\, \tau^{-n} \ast \left(u^{-1} \cdot \nu_{1}^{j}f^{a, j}\right) \equiv 0 \bmod p^{nk}.
    \]
\end{proposition}

\begin{proof}
     Let $R = W(\bar{\F}_{p})$ be the ring of Witt vectors of $\bar{\F}_{p}$ and $U = \bar{P}_{1}^{G}u^{-1}Q_{H}^{0}$. As before we model $V_{\lambda, \Zp}$ as a space of regular functions on $G_{/\Zp}$ taking values $W_{\lambda, \Zp}$, so that each $\nu_{1}^{j}f^{a, j}$ is uniquely determined by its restriction to the Zariski open set $U(R)$. Moreover, for $\bar{n}\ell \in \bar{P}_{1}^{G}(R)$, $q \in Q_{H}^{0}(R)$ we have 
$$
    \nu_{1}^{j}f^{a,j}(\bar{n}\ell u^{-1} q) = \left(\ell \cdot \nu_{1}^{j}f^{a,j}(u^{-1})\right)\nu_{1}^{j}(q) =
    \nu_{1}^{j}(q)\left(\ell \cdot w^{-}_{\lambda}\right).
$$
 Let $n \in N_{1}(R)$, and note that for any $m \geq 1$ (cf. \cite[Lemma 3.4]{rockwood2022spherical}) we have $\bar{n}\ell\tau^{m}n\tau^{-m}u^{-1} \in U(R)$, so we write $\bar{n}\ell\tau^{m}n\tau^{-m}u^{-1}= \bar{n}'\ell'u^{-1}q$. We then have for any $ 0 \leq k \leq a$
\begin{align*}
\sum_{j = 0}^{k}\binom{k}{j}(-1)^{j}\tau^{-n} * u^{-1} \nu_{1}^{j}f^{a,j}(\bar{n}\ell n) &= \sum_{j = 0}^{k}\binom{k}{j}(-1)^{j}\nu_{1}^{j}f^{a,j}(\bar{n}\ell \tau^{n}n\tau^{-n}u^{-1}) \\
&= \sum_{j = 0}^{k}\binom{k}{j}(-1)^{j}\nu_{1}^{j}f^{a,j}(\bar{n}'\ell' u^{-1}q) \\
&=\left(\sum_{j = 0}^{k}\binom{k}{j}(-1)^{j}\nu_{1}^{j}(q)\right)\left(\ell' \cdot w_{\lambda}^{-}\right)\\
&= \left(1 - \nu_{1}(q)\right)^{k}\left(\ell' \cdot w_{\lambda}^{-}\right),
\end{align*}
 Now reducing modulo $p^{n}$ we have that $q \in Q_{H}^{0} \cap u\bar{P}_{1}^{G}u^{-1} = \{\begin{psmallmatrix}
x & & \\ & 1 & \\ & & x
\end{psmallmatrix} \times (x)\} \subset \mathrm{ker}(\nu_{1})$, so $\left(1 - \nu_{1}(q)\right)^{k} \equiv \ 0 \ \mathrm{mod} \ p^{nk}$. Since the big Bruhat cell $U_{\mathrm{Bru}} := \bar{N} \times L_{1}^{G} \times N$ is smooth and $\bar{\F}_{p}$-points are dense we have an injective map 
\[
   \mathcal{O}_{G_{/\Zp}}(G_{/\Zp})/p^{nk} \hookrightarrow \prod_{x \in U_{\mathrm{Bru}}(R/p^{nk})}R/p^{nk}.
\]
Our computation shows that $\sum_{j = 0}^{k}\binom{k}{j}(-1)^{j}\tau^{-n} * u^{-1} \nu_{1}^{j}f^{a,j}$ maps to zero under this map and thus must vanish modulo $p^{nk}$.  
\end{proof}

\begin{remark} \label{rem:conggeneral}
    Proposition \ref{prop:cong} does not follow directly from the general congruences for branching maps of spherical pairs established in \cite[\S 7]{rockwood2022spherical} since the representation $V_{\lambda, \Zp}$ is induced from a representation of $L_{1}$ of dimension $> 1$ when $a > 0$.
\end{remark}

\section{Automorphic representations} \label{sec:auto}

\subsection{Characters} \label{sec:characters}

Given a Dirichlet character $\chi: \left(\Z/N\Z\right)^{\times} \to \C^{\times}$ there is a unique finite-order Hecke character $\hat{\chi}: \Q^{\times} \backslash \A^{\times}/\R^{\times}_{> 0} \to \C^{\times}$ such that $\hat{\chi}(\varpi_{\ell}) = \chi(\ell)$ for all primes $\ell \nmid N$, where $\varpi_{\ell}$ is a uniformiser at $\ell$. Conversely, every finite-order Hecke character is $\hat{\chi}$ for a unique primitive Dirichlet character $\chi$. Note that the restriction of $\hat{\chi}$ to $\hat{\Z}^{\times} \subset \A_{f}^{\times}$ is the \emph{inverse} of the composite $\hat{\Z}^{\times} \twoheadrightarrow (\Z/N\Z)^{\times} \xrightarrow{\chi} \C^{\times}$. Given a representation $\Pi$ of $\GL_{3}(\A)$ we write $\Pi \times \chi$ for $\Pi \otimes [\hat{\chi} \circ \det]$.

Let $\psi$ be the unique character of the additive group $\A/\Q$ such that $\psi(x) = \exp(-2\pi i x)$ for $x \in \R$; its restriction to $\Q_{\ell}$ maps $1/\ell^{n}$ to $\exp(2\pi i/\ell^{n})$. If $\varepsilon(\hat{\chi}_{\ell}, \psi_{\ell})$ denotes the local $\varepsilon$-factor with respect to the unramified Haar measure $dx$, then
\[
    \prod_{\ell \mid N}\varepsilon(\hat{\chi}_{\ell}, \psi_{\ell}) = G(\chi) := \sum_{a \in \left(\Z/N\Z\right)^{\times}}\chi(a)\exp(2\pi i a/N)
\]
for $\chi$ primitive of conductor $N$, where $G(\chi)$ is the Gauss sum.

Finally, we fix Haar measures: on $H(\A_{f})$ we use the unramified Haar measure $dh_{f}$, normalised so that $H(\hat{\Z})$ has volume $1$. The measure $dh_{\infty}$ at infinity will be fixed in Section \ref{sec:comppair}.

\subsection{Automorphic representations of $\GL_{3}$}

Throughout, $\Pi$ denotes a regular algebraic cuspidal automorphic representation (RACAR) of $\GL_{3}(\A)$, with central character $\omega_{\Pi}$. We identify $\Pi$ with its realisation in $L^{2}_{0}(\GL_{3}(\Q)\backslash\GL_{3}(\A))$. Regular algebraicity means that the infinitesimal character of $\Pi_{\infty}$ agrees with that of an algebraic representation $V_{\lambda}$ for a (uniquely determined) dominant integral weight $\lambda$ which we refer to as the weight of $\Pi$.

\subsubsection{Whittaker models}
We write
\begin{align*}
    \mathcal{W}_{\psi}: \Pi &\xrightarrow{\ \sim\ } \mathcal{W}_{\psi}(\Pi) \subset \mathrm{Ind}_{N_{\GL_{3}}(\A)}^{\GL_{3}(\A)}\psi, \\
    \varphi &\mapsto W_{\varphi}(g) := \int_{N_{\GL_{3}}(\Q)\backslash N_{\GL_{3}}(\A)}\varphi(ng)\psi^{-1}(n)\,dn
\end{align*}
for the standard Whittaker model of $\Pi$, where $\psi\begin{psmallmatrix}
    1 & x & * \\
    & 1 & y \\
    & & 1
\end{psmallmatrix} := \psi(x + y)$. As $\psi$ is fixed throughout we will usually drop it from the notation. From the uniqueness of Whittaker models, $\mathcal{W}(\Pi_{f})$ carries a canonical $E$-structure $\mathcal{W}(\Pi_{f}, E)$ for any field of definition $E$ of $\Pi_{f}$, see \cite[\S 3.3.1]{loefflerWilliams2021p}.

\subsubsection{Cohomological representations} \label{sec:cohomological}
Any RACAR $\Pi$ of weight $\lambda$ is cohomological with coefficients in $V_{\lambda}^{\vee}$, i.e.
\[
    H^{\bullet}(\mathfrak{gl}_{3}, K^{\circ}_{3, \infty}; \Pi_{\infty} \otimes V^{\vee}_{\lambda}(\C)) \neq 0,
\]
where $\mathfrak{gl}_{3} = \mathrm{Lie}(\GL_{3}(\R))$ and $K_{3, \infty}^{\circ} = K_{\GL_{3}, \infty}^{\circ}Z^{\circ}_{\GL_{3}, \infty}$, for $K_{\GL_{3}, \infty}$ a maximal compact subgroup of $\GL_{3}(\R)$ and $Z_{\GL_{3}, \infty}$ the centre of $\GL_{3}(\R)$, with $(-)^{\circ}$ denoting identity components, see \cite[Lemma\ 4.9]{clozel2}. This cohomology is concentrated in degrees $2$ and $3$, and is one-dimensional in each. We let
\begin{equation} \label{eq:zetainfty}
    \zeta_{\infty} \in H^{2}(\mathfrak{gl}_{3}, K^{\circ}_{3, \infty}; \Pi_{\infty} \otimes V^{\vee}_{\lambda}(\C))
\end{equation}
denote a generator (later we will allow ourselves judicious rescaling of this generator).

The weight $\lambda$ is necessarily pure, and we normalise $\Pi$ (twisting by a power of $\|\cdot\|\circ\det$ if necessary) so that $\lambda = (a, 0, -a)$ with $a \in \Z_{\geq 0}$. With this normalisation we have
\begin{equation} \label{eq:piinfty}
    \Pi_{\infty} \cong \mathrm{Ind}_{P_{2}(\R)}^{\GL_{3}(\R)}\left(D_{2a + 3}, \mathrm{id}\right) \quad \text{or} \quad \mathrm{Ind}_{P_{2}(\R)}^{\GL_{3}(\R)}\left(D_{2a + 3}, \mathrm{sgn}\right),
\end{equation}
where $D_{2a+3}$ is the discrete series representation of $\GL_{2}(\R)$ of lowest weight $2a + 3$, the first case occurring when $\omega_{\Pi}$ is odd and the second when it is even. In particular $\omega_{\Pi}$ has finite order (so we identify it with a Dirichlet character, also denoted $\omega_{\Pi}$), and $\Pi$ is unitary. Note that $\Pi_{\infty}$, and hence any quantity depending only on $\Pi_{\infty}$, is determined by $a$ and the sign $\omega_{\Pi}(-1)$.

\subsubsection{Conductors and $L$-functions}
For $N \geq 1$ set
\[
    \mathcal{U}_{1}(N) = \left\{g \in \GL_{3}(\hat{\Z}): g \equiv \begin{psmallmatrix}
        * & * & * \\
        * & * & * \\
        0 & 0 & 1
    \end{psmallmatrix} \bmod N\right\}.
\]
By the local newvector theory of Jacquet--Piatetski-Shapiro--Shalika there is an $N$ for which $\Pi_{f}^{\mathcal{U}_{1}(N)} \neq 0$. If $N_{\Pi}$ is the smallest such $N$ then these invariants are one-dimensional, and we call $N_{\Pi}$ the \emph{conductor} of $\Pi$. We write $N^{(p)}_{\Pi}$ for its prime-to-$p$ part.

We let $L(\Pi, s) = \prod_{\ell < \infty}L(\Pi_{\ell}, s)$ denote the standard $L$-function of $\Pi$ (without Archimedean factors), in the analytic normalisation, so the Euler product converges for $\mathfrak{Re}(s) > 1$. Since $\Pi$ is both $C$- and $L$-algebraic, it is definable over a number field $E$ (its \emph{field of rationality}), and $L(\Pi_{\ell}, s) = P_{\ell}(\ell^{-s})^{-1}$ with $P_{\ell}(X) \in E[X]$ of degree $\leq 3$. For unramified $\ell$ we write $P_{\ell}(X) = (1 - \alpha_{\ell}X)(1 - \beta_{\ell}X)(1 - \gamma_{\ell}X)$, where the Satake parameters $\alpha_{\ell}, \beta_{\ell}, \gamma_{\ell}$ are units outside $\ell$ and have $\ell$-adic valuation $\geq -1-a$. We write $L^{(p)}(\Pi \times \eta, s)$ for the $L$-function without its Euler factor at $p$.

\subsubsection{Critical values} \label{sec:critvalues}
For a Dirichlet character $\eta$, the critical values of $L(\Pi \times \eta, s)$ are at $s = t$ for integers $t$ with either $-a \leq t \leq 0$ and $(-1)^{t} = \omega_{\Pi}\eta(-1)$, or $1 \leq t \leq 1 + a$ and $(-1)^{t} = -\omega_{\Pi}\eta(-1)$ \cite[Proposition\ 2.4]{loefflerWilliams2021p}. We thus set
\begin{align*}
    \Crit^{-}(\Pi) &= \{(-j, \eta): 0 \leq j \leq a, \ (-1)^{j} = \omega_{\Pi}\eta(-1)\}, \\
    \Crit^{+}(\Pi) &= \{(j + 1, \eta): 0 \leq j \leq a, \ (-1)^{j} = \omega_{\Pi}\eta(-1)\},
\end{align*}
the critical values in the left and right halves of the critical strip, and write $\Crit^{\pm}_{p}(\Pi) \subset \Crit^{\pm}(\Pi)$ for the subsets where $\eta$ has $p$-power conductor. 

\subsubsection{Rationality and the factor at infinity} \label{sec:einfty}
Following \cite[Definition\ 2.6]{loefflerWilliams2021p} (cf.\ \cite[\S 1]{coates89}), for $(-j, \eta) \in \Crit^{-}(\Pi)$ define the modified Euler factor at infinity
\[
    e_{\infty}(\Pi_{\infty} \times \eta_{\infty}, -j) := i^{j - a - 1}\,\Gamma_{\C}(a + 1 - j) = 2\,(2\pi i)^{j - a - 1}\,\Gamma(a + 1 - j),
\]
where $\Gamma_{\C}(s) = 2(2\pi)^{-s}\Gamma(s)$. For $(j + 1, \eta) \in \Crit^{+}(\Pi)$ define
\[
    e_{\infty}(\Pi_{\infty} \times \eta_{\infty}, j + 1) := (-i)^{-j - a - 2}\,\Gamma_{\C}(a + 2 + j) \cdot \frac{\Gamma_{\R}(j + 1 + \epsilon)}{(-i)^{\epsilon}\,\Gamma_{\R}(\epsilon - j)},
\]
where $\Gamma_{\R}(s) = \pi^{-s/2}\Gamma(s/2)$ and $\epsilon \in \{0, 1\}$ with $\epsilon \equiv j + 1 \bmod 2$.

Work of Mahnkopf \cite{mahnkopf98, mahnkopf00} and Kasten--Schmidt \cite{kastenschmidt13} shows that there is a complex period $\Omega^{-}_{\Pi} \in \C^{\times}$ and, for each $j$, an inexplicit non-zero scalar $\tilde{e}_{\infty}(\Pi_{\infty} \times \eta_{\infty}, -j) \in \C^{\times}$ such that
\[
    \tilde{e}_{\infty}(\Pi_{\infty} \times \eta_{\infty}, -j)\,\frac{L(\Pi \times \eta, -j)}{\Omega^{-}_{\Pi}} \in E[\eta, G(\eta)]
\]
for all $(-j, \eta) \in \Crit^{-}(\Pi)$. The main theorem of \cite{loefflerWilliams2021p} includes the statement that, after a suitable normalisation of $\zeta_{\infty}$, one has $\tilde{e}_{\infty} = e_{\infty}$. We will inherit this in Section \ref{sec:padicL}.

\subsection{Hecke operators and $P_{i}$-refinements} \label{sec:heckeopsandprefs}

Let $p$ be a prime, and recall the maximal parabolics $P_{1}, P_{2} \subset \GL_{3}$ from Section \ref{sec:algp}. For $i = 1, 2$, let $\mathcal{J}_{p, i}$ denote the parahoric subgroup of $\GL_{3}(\Zp)$ associated to $P_{i}$. For any open compact $K_{p} \subset \mathcal{J}_{p, i}$ admitting an Iwahori decomposition with respect to $P_{i}$ (see Definition \ref{def:iwahoridecomp}) we have normalised Hecke operators
\[
    U_{p, i} = p^{a}\left[K_{p}\tau_{i}K_{p}\right], \qquad \tau_{1} = \begin{psmallmatrix}
        p & & \\
        & 1 & \\
        & & 1
    \end{psmallmatrix}, \quad \tau_{2} = \begin{psmallmatrix}
        p & & \\
        & p & \\
        & & 1
    \end{psmallmatrix} \in \GL_{3}(\Qp),
\]
acting on $\Pi_{p}^{K_{p}}$. The normalisation factor $p^{a}$ ensures that these operators preserve integral lattices in Betti cohomology (cf.\ Section \ref{sec:heckecoh}).

As in \cite{loefflerWilliams2021p}, we denote by $J_{P_{i}}(\Pi_{p})$ the unitarily normalised Jacquet module of $\Pi_{p}$ at $P_{i}$, a smooth admissible representation of $(\GL_{i} \times \GL_{3 - i})(\Qp)$.

\begin{definition}
    For $i = 1, 2$, a \emph{$P_{i}$-refinement} of $\Pi_{p}$ is a choice of irreducible representation $\sigma_{p} \times \sigma_{p}'$ of $(\GL_{i} \times \GL_{3-i})(\Qp)$ appearing as a subrepresentation of $J_{P_{i}}(\Pi_{p})$.
\end{definition}

\begin{remark}
    This is the natural generalisation of an Iwahori refinement of $\GL_{n}$, which consists of a choice of character of the torus $T$ appearing in $J_{B}(\Pi_{p})$ (in the unramified case, a full ordering of the Satake parameters).
\end{remark}

We write $\tilde{\Pi}$ for a pair $(\Pi, \sigma_{p})$ of $\Pi$ together with a $P_{i}$-refinement; for a given $\Pi_p$ and $i$, each of $\sigma_p$, $\sigma_p'$ determines the other, so we usually specify only $\sigma_{p}$, and we say the refinement is \emph{unramified} if $\sigma_{p}$ is. Fix an embedding $E \hookrightarrow \bar{\Q}_{p}$.

\begin{definition} \label{def:slope}
    The \emph{slope} of the $P_{i}$-refinement $\sigma_{p} \times \sigma'_{p}$ is $v_{p}(\omega_{\sigma_{p}}(p))$, where $\omega_{\sigma_{p}}$ is the central character of $\sigma_{p}$.
\end{definition}

For a RACAR of weight $(a, 0, -a)$, the slope of any $P_{i}$-refinement lies in the interval $[-1 - a, 1 + a]$. The key assumption of \cite{loefflerWilliams2021p} was that $\tilde{\Pi}$ is \emph{nearly ordinary}: that the slope is exactly $-1 - a$, the minimal value. This is the assumption we remove in this article, replacing it by the following.

\begin{definition} \label{def:smallslope}
    We say the $P_{i}$-refinement $\sigma_{p} \times \sigma'_{p}$ has \emph{small slope} if its slope is $< 0$; equivalently, if
    \[
        v_{p}\left(p^{a + 1}\,\omega_{\sigma_{p}}(p)\right) < a + 1.
    \]
\end{definition}

As in the $\GL_{2}$ case, where one usually imposes a $p$-regularity assumption that the $U_{p}$-eigenvalues be distinct, we impose:

\begin{definition}
    A $P_{i}$-refinement $\sigma_{p} \times \sigma'_{p}$ is \emph{regular} if all irreducible subquotients of $J_{P_{i}}(\Pi_{p})/(\sigma_{p} \times \sigma'_{p})$ have central characters different from that of $\sigma_{p} \times \sigma'_{p}$.
\end{definition}
\begin{remark}
    Regularity is automatic for nearly ordinary refinements; at finite slope we impose it throughout, so that the relevant Hecke eigenspaces are genuine (not merely generalised) eigenspaces.
\end{remark}
\begin{example}
In the style of \cite[Example 2.13]{loefflerWilliams2021p}, we briefly recall the classification of generic representations of  $\mathrm{GL}_{3}(\Qp)$ and describe which of their regular refinements are of small slope.
\begin{itemize}
    \item If $\Pi_p$ is supercuspidal, it admits no $P_i$ refinements.
    \item If $\Pi_p= \mathrm{St}_3\otimes \lambda$ is a twist of the Steinberg representation by a character $\lambda$, it admits a unique $P_1$ and a unique $P_2$ refinement, both of which have slope -1 and are thus of small slope for all $a.$
    \item If $\Pi_p$ is parabolically induced from a representation of the form $\theta\times\pi$ of $\mathrm{GL}_1\times \mathrm{GL_2},$ where $\theta$ is a character and $\pi$ is a supercuspidal representation of $\mathrm{GL_2},$ then its unique $P_1$ refinement is $\theta$ and its unique $P_2$ refinement is $\pi,$ both of which are regular. Their slopes are $v_p(\theta(p))$ and $-v_p(\theta(p))$ respectively, so the former is of small slope if and only if $v_p(\theta(p))< 0$ and the latter if and only if $v_p(\theta(p))>0.$
    \item If $\Pi_p$ is irreducibly induced from $\theta\times(\mathrm{St}_2\otimes\lambda)$, then it has two $P_1$-refinements, $\theta$ and $\lambda|\cdot|^{1/2}$, of slopes $v_p(\theta(p))$ and $-1/2\bigl(1+v_p(\theta(p))\bigr)$, and two
    $P_2$-refinements, $\mathrm{St}_2\otimes\lambda$ and $\mathrm{Ind}(\theta\otimes\lambda|\cdot|^{1/2})$, of slopes $-v_p(\theta(p))$ and $1/2\bigl(v_p(\theta(p))-1\bigr),$ all four of which are regular. Thus, it has at least one $P_1$ and one $P_2$ refinement of small slope.
    \item If $\Pi_p$ is an irreducible principal series representation,
    induced from a character $\chi_1\times\chi_2\times\chi_3$, then its $P_1$-refinements are the $\chi_i$, of slope $v_p(\chi_i(p))$, and its $P_2$-refinements are the ${\mathrm{Ind}(\chi_j\otimes\chi_k)}$, of slope $-v_p(\chi_i(p))$, where $\{i,j,k\}=\{1,2,3\}.$ The regular ones are the ones for which $\chi_i\neq\chi_j,\chi_k.$ $\Pi_p$ has a small slope $P_1$ refinement and a small slope $P_2$ refinement unless $v_p(\chi_i(p))=0$ for all $i.$
    
\end{itemize}    
\end{example}

\begin{remark}
    Small slope refinements need not exist. For example, if $\Pi_{p}$ is unramified with Satake parameters $\alpha_{p}, \beta_{p}, \gamma_{p}$, the $P_{1}$-refinements are the unramified characters sending $p$ to one of these parameters. Since $v_{p}(\alpha_{p}) + v_{p}(\beta_{p}) + v_{p}(\gamma_{p}) = v_{p}(\omega_{\Pi}(p)) = 0$  it can happen that all three valuations vanish, in which case every $P_{1}$-refinement has slope $0$, i.e.\ normalised slope exactly $a + 1$. 
\end{remark}

\subsubsection{$P_{1}$-refined newvectors}
From now on fix a regular unramified $P_{1}$-refinement $\sigma_{p}$, determined by $\alpha_{p} := \sigma_{p}(p) \neq 0$. (This is no loss of generality: as in \cite[Proposition\ 2.19]{loefflerWilliams2021p}, since $\sigma_{p}$ is a character of $\Qp^{\times}$ we may twist $\Pi$ by a Dirichlet character to make it unramified, and the conjectural interpolation property is invariant under such twists.) For $r \geq 0$ consider the group
\[
    \mathcal{U}_{1, p}^{(P_{1})}(p^{r}) = \left\{g \in \GL_{3}(\Zp): g \equiv 1 \ \mathrm{mod} \ \begin{psmallmatrix}
        * & * & * \\
        p & * & * \\
        p^{R} & p^{r} & p^{r}
    \end{psmallmatrix}\right\} \subset \GL_{3}(\Zp), \qquad R = \max\{r, 1\}.
\]
This group admits an Iwahori decomposition with respect to $P_{1}$, and its intersection with the $\GL_{2}$-factor of $L_{1}$ is the level-$p^{r}$ mirabolic subgroup of Casselman's local newvector theory \cite{casselman73}. Combining Casselman's canonical lifting theorem (see \cite{emerton2006jacquet} for an account) with the $\GL_{2}$ newvector theorem, one obtains (cf.\ \cite[Proposition\ 2.24]{loefflerWilliams2021p}):

\begin{proposition} \label{prop:newvector}
    For all sufficiently large $r$ there exists $\varphi_{p} \in \Pi_{p}^{\mathcal{U}_{1, p}^{(P_{1})}(p^{r})}$ satisfying $U_{p, 1}\varphi_{p} = p^{1 + a}\alpha_{p}\,\varphi_{p}$. The minimal such $r$, denoted $r(\tilde{\Pi})$, equals the conductor of the $\GL_{2}(\Qp)$-representation $\sigma'_{p}$, and for $r = r(\tilde{\Pi})$ the space of such vectors is one-dimensional.
\end{proposition}

We define the \emph{$P_{1}$-refined Whittaker newvector} $W^{\alpha}$ to be the unique basis vector of this one-dimensional space in the Whittaker model of $\Pi_{p}$ normalised so that $W^{\alpha}(1) = 1$. Its values along the torus are given in terms of the $\GL_{2}$-newvector of $\sigma'_{p}$ by \cite[Proposition\ 2.25]{loefflerWilliams2021p}.

\subsection{The Coates--Perrin-Riou factor at $p$} \label{sec:ep}

 Write $\psi^{-} = \psi$ and $\psi^{+} = \psi^{-1}$ and let $i = 1$ for the sign $-$ and $i = 2$ for the sign $+$. For a $P_{i}$-refinement $\sigma_{p}$ of $\Pi_{p}$, following \cite[\S 2]{coates89} and \cite[Definition\ 2.16]{loefflerWilliams2021p} we define the modified Euler factor at $p$ the be the meromorphic function
\[
    e_{p}(\Pi_{p}, \sigma_{p}, s) := \gamma_{p}(\sigma_{p}, \psi^{\pm}, t)^{-1} = \frac{L(\sigma_{p}, s)}{L(\sigma_{p}^{\vee}, 1 - s)\,\varepsilon_{p}(\sigma_{p}, \psi^{\pm}_{p}, s)},
\]
in the variable $s \in \C$, where $\gamma_{p}$ is the local $\gamma$-factor. Since our refinements are fixed throughout, we drop $\sigma_{p}$ (and $\pm$) from the notation and write $e_{p}(\Pi_{p}, s)$, and for twists by characters $\eta_{p}$ of $\mathrm{GL}_{3}(\Qp)$ we write $e_{p}(\Pi_{p} \times \eta_{p}, s) := e_{p}(\Pi_{p} \times \eta_{p}, \sigma_{p} \times \eta_{p}, s)$.

\begin{example} \label{ex:epexplicit}
    If $\sigma_{p}$ is an unramified $P_{1}$-refinement and $(-j, \eta) \in \Crit^{-}_{p}(\Pi)$, then if $\sigma_{p}(p) = \alpha_{p} \neq 0, p^{-j}$ we have
    \[
        e_{p}(\Pi_{p} \times \eta_{p}, -j) = \begin{cases}
            G(\eta^{-1}) \cdot \left(p^{j + 1}\alpha_{p}\right)^{-n} & \mathrm{cond}(\eta) = p^{n} > 1, \\[2pt]
            \dfrac{1 - p^{-j - 1}\alpha_{p}^{-1}}{1 - p^{j}\alpha_{p}} & \eta = 1.
        \end{cases}
    \]
\end{example}
\begin{remark} \label{rem:pole}
    Under our relaxed slope assumptions it can genuinely happen that $ e_{p}(\Pi_{p}, \sigma_{p}, s)$ has a pole at some $(-j, \mathrm{id}) \in \mathrm{Crit}^{\pm}_{p}(\Pi)$. In particular, when $\Pi_{p} = \mathrm{St}_{3}$ the unique $P_{1}$-refinement satisfies $\alpha_{p} = p^{-1}$ (we call this the `exceptional Steinberg' case) and thus $e_{p}(\Pi_{p}, \sigma_{p}, s)$  has a pole at $(-1, \mathrm{id}) \in \mathrm{Crit}_{p}^{-}(\Pi)$. In fact, this is the only case of unitary, generic $\Pi_{p}$ where such poles can occur. The nearly ordinary assumption present in \cite{loefflerWilliams2021p} excludes this example.  
\end{remark}
\section{Locally symmetric spaces and cohomology} \label{sec:lss}

We define the locally symmetric spaces, open compact subgroups and local systems that we will use in this article.

\subsection{Locally symmetric spaces}
For any reductive group $A$ over $\Q$ and any open compact subgroup $K \subset A(\A_{f})$, define
\[
    Y^{A}(K) := A^{+}(\Q) \backslash \left(\left(A(\A_{f})/K\right) \times \mathcal{H}^{+}_{A}\right),
\]
where a superscript $(\cdot)^{+}$ indicates the connected component of the identity in the real topology, $\mathcal{H}^{+}_{A} = A^{+}(\R)/K^{\circ}_{A, \infty}Z^{+}_{A}(\R)$, $K_{A, \infty} \subset A(\R)$ is a maximal compact subgroup, and $Z_{A}$ is the centre of $A$.

Since $\iota(Z_H)\not\subset Z_G$, the embedding $\iota$ does not induce a map $\mathcal{H}_H^+\to\mathcal{H}_G^+$ of symmetric spaces. To account for this, as in \cite{loefflerWilliams2021p}, we consider the auxiliary space $\tilde{\mathcal{H}}^{+}_{H} := H^{+}(\R)/K^{\circ}_{H, \infty}\iota^{-1}(Z^{+}_{G}(\R))$, which, since $\iota^{-1}(Z_{G}) \subset Z_{H}$ admits natural maps
\begin{equation} \label{eq:roof}
\begin{tikzcd}
    & \tilde{\mathcal{H}}^{+}_{H} \arrow[dl, "f"'] \arrow[dr, "\iota"] & \\
    \mathcal{H}^{+}_{H} & & \mathcal{H}^{+}_{G},
\end{tikzcd}
\end{equation}
the left map being the natural quotient and the right map induced by $\iota$. For any open compact $K_{H} \subset H(\A_{f})$ we write
\[
    \tilde{Y}^{H}(K_{H}) = H^{+}(\Q) \backslash \left(\left(H(\A_{f})/K_{H}\right) \times \tilde{\mathcal{H}}^{+}_{H}\right),
\]
and for open compact $K_{G} \subset G(\A_{f})$ we obtain a diagram
\[
\begin{tikzcd}
    & \tilde{Y}^{H}(K_{G} \cap H) \arrow[dl, "f"'] \arrow[dr, "\iota"] & \\
    Y^{H}(K_{G} \cap H) & & Y^{G}(K_{G}).
\end{tikzcd}
\]
The spaces $Y^{H}$ (resp.\ $\tilde{Y}^{H}$, resp.\ $Y^{G}$) have dimension $2$ (resp.\ $3$, resp.\ $5$) as real manifolds; if $K_{G}$ is sufficiently small the map $\iota$ on the right is a closed embedding of codimension $2$ (cf.\ \cite[Theorem\ 6.10]{loefflerWilliams2021p} and \cite[\S 4.6]{loefflerspherical}).

\subsection{Components of locally symmetric spaces} \label{sec:components}

\begin{definition}\label{km1m2}
Let $A\in\{G,H\}$ and let $K\subset A(\mathbb{A}_f)$ be a compact subgroup. For integers $M_{1}, M_{2} \geq 1$ define
\[
    K(M_{1}, M_{2}) = \{k \in K: \nu_{i}(k) \in 1 + M_{i}\hat{\Z}, \ i = 1, 2\}.
\]
When $M_{2} = 1,$ we write $K(M_{1}, 1) = K(M_{1})$. 
\end{definition}

Let $K\subset A(\mathbb{A}_f)$ be open compact, $M_1,M_2\geq1$ and suppose the map
\begin{align*}
    K &\to \left(\Z/M_{1}\Z\right)^{\times} \times \left(\Z/M_{2}\Z\right)^{\times} \\
    k &\mapsto (\nu_{1}(k), \nu_{2}(k))
\end{align*}
is surjective.
We regard $\nu_i$ as a map
$$
  \nu_i\colon A^+(\mathbb{Q})\backslash A(\mathbb{A}_f)/K(M_1,M_2)\to
  \mathbb{Q}^\times_{>0}\backslash\mathbb{A}_f^\times\big/\bigl((1+M_i\hat{\mathbb{Z}})
  \textstyle\prod_{\ell\nmid M_i}\mathbb{Z}_\ell^\times\bigr)\cong(\mathbb{Z}/M_i\mathbb{Z})^\times,
$$
the last identification being normalised so that a uniformiser $\varpi_\ell$ at
$\ell\nmid M_i$ maps to $\ell\bmod M_i$. The following lemma is a special case of
a well-known general phenomenon.

\begin{lemma} \label{lem:components}
The map
$$
  \Upsilon=(\nu_1,\nu_2,\mathrm{pr})\colon Y^A(K(M_1,M_2))\xrightarrow{\ \sim\ }
  (\mathbb{Z}/M_1\mathbb{Z})^\times\times(\mathbb{Z}/M_2\mathbb{Z})^\times\times Y^A(K)
$$
is an isomorphism, equivariant for the natural action of $K$ on the source
and the action
\[
  k \cdot (x_1,x_2,y)
    = (\nu_1(k)x_1,\nu_2(k)x_2,y)
\]
on the target.
\end{lemma}

\subsection{Open compact subgroups} \label{sec:opencompacts}

We now fix the level structures used in the construction. Fix $n \geq 1$ and recall $r = r(\tilde{\Pi})$ from Proposition \ref{prop:newvector}. 

\subsubsection{Subgroups of $\GL_{2}$} Let $\mathcal{U}^{\GL_{2}, (p)} \subset \GL_{2}(\A^{(p)}_{f})$ be open compact and define
\[
\mathcal{U}^{\GL_{2}}_{n} := \left\{A \in \GL_{2}(\Zp) : A \equiv \begin{psmallmatrix} * & * \\ & 1 \end{psmallmatrix} \bmod p^{n} \right\} \times \mathcal{U}^{\GL_{2}, (p)} \subset \GL_{2}(\A_{f}).
\]

\subsubsection{Subgroups of $H$}
Let $\mathcal{U}^{H, (p)} \subset H(\A^{(p)}_{f})$ be open-compact and define for $t, n \geq 1$ the following subgroups of $H(\A_{f})$:
\begin{align*}
      Q_{H, t, n}^{0} &:= \{\begin{psmallmatrix}
        a & b \\ c & d
    \end{psmallmatrix} \times (z): b \equiv 0 \bmod p^{n}, c \equiv 0 \bmod p^{t}, d \equiv 1 \bmod p^{t}\} \times \mathcal{U}^{H, (p)} , \\
     Q^0_{H,\infty,n}&:=\bigcap_t Q^0_{H,t,n},\\
      \mathcal{U}^{H}_{1}(p^{t}) &:= \left\{A \in H(\Zp) : A \equiv \begin{psmallmatrix} * & * \\ & 1 \end{psmallmatrix} \times (*) \bmod p^{t} \right\} \times \mathcal{U}^{H, (p)}.      
\end{align*}

We use the notation of Definition \ref{km1m2} for these groups: for example
$Q^0_{H,t,n}(p^n)=\{h\in Q^0_{H,t,n}:\nu_1(h)\equiv1\bmod p^n\}$.

\subsubsection{Subgroups of $\GL_{3}$ and $G$}
Let $\mathcal{U}^{\GL_{3}, (p)} \subset \GL_{3}(\A^{(p)}_{f})$ be open compact and define the following subgroups of $\mathrm{GL}_{3}(\A_{f})$:
\begin{align*}
\mathcal{U} &:= \mathcal{U}^{(P_1)}_{1,p}(p^r)\times\mathcal{U}^{\mathrm{GL}_3,(p)}
  = \Big\{g\in\mathrm{GL}_3(\mathbb{Z}_p) : g\equiv 1 \bmod
  \left(\begin{smallmatrix} * & * & * \\ p & * & * \\ p^R & p^r & p^r\end{smallmatrix}\right)\Big\}
  \times\mathcal{U}^{\mathrm{GL}_3,(p)},\\
     \mathcal{U}^{\GL_{3}}_{n} &:= \left\{g \in \GL_{3}(\Zp): g \equiv 1 \ \mathrm{mod} \ \begin{psmallmatrix}
    * & p^{n} & p^{n} \\ * & * & * \\ p^{r} & p^{r} & p^{r}
    \end{psmallmatrix}\right\} \times \mathcal{U}^{\GL_{3}, (p)}, \\
    \mathcal{V}_{n}^{\GL_{3}} &:= \tau^{-n}\,\mathcal{U}^{\GL_{3}}_{n}\,\tau^{n} = \left\{g \equiv 1 \ \mathrm{mod} \ \begin{psmallmatrix}
    * & * & * \\ p^{n} & * & * \\ p^{n + r} & p^{r} & p^{r}
    \end{psmallmatrix}\right\} \times \mathcal{U}^{\GL_{3}, (p)},
\end{align*}
where $R=\max\{r,1\}$. Note that $\mathcal{V}_{n}^{\GL_{3}} \subset \mathcal{U}$ for all $n \geq 1$, and that $\mathcal{U}_{n + 1}^{\GL_{3}} = \mathcal{U}_{n}^{\GL_{3}} \cap \tau\,\mathcal{U}_{n}^{\GL_{3}}\,\tau^{-1}$. We define
\begin{align*}
\mathcal{U}^{G}_{n} &= \mathcal{U}^{\GL_{3}}_{n} \times \hat{\Z}^{\times} \\
\mathcal{V}^{G}_{n} &= \tau^{-n}\mathcal{U}^{G}_{n}\tau^{n} \\ \mathcal{U}^{G, (p)} &= \mathcal{U}^{\mathrm{GL}_{3}, (p)} \times \prod_{\ell \neq p}\Z_{\ell}^{\times}.
\end{align*}
The following is the analogue of \cite[Proposition\ 6.3]{loefflerWilliams2021p}. 

\begin{lemma} \label{lem:inclusion}
    Let $t \geq \max(n, r)$ and $h = \left(\begin{psmallmatrix} a & b \\ c & d\end{psmallmatrix}, z\right) \in H(\Zp)$. Then
    \begin{enumerate}
        \item $\mathcal{U}^{H}_{1}(p^{t}) \cap u\,\mathcal{U}^{G}_{n}u^{-1}$ consists of those $h \in \mathcal{U}^{H}_{1}(p^{t})$ with $b \equiv 0$ and $a \equiv z \bmod p^{n}$. Its index in $\mathcal{U}^{H}_{1}(p^{t})$ is $p^{2n - 1}(p - 1)$,
        \item We have  
     $Q^0_{H,t,n}(p^n)=\mathcal{U}_1^H(p^t)\cap u\,\mathcal{U}_n^G\,u^{-1}.$ In particular, if we have $\nu_{2}\left(\mathcal{U}^{H, (p)}\right) \in 1 + D\hat{\Z}$ for some prime-to-$p$ $D$, then
        \[
            \mathcal{U}^{H}_{1}(p^{t}) \cap u\,\mathcal{U}^{G}_{n}u^{-1} \subset u\,\mathcal{U}^{G}_{n}(p^{n}, D)\,u^{-1}.
        \]
    \end{enumerate}
\end{lemma}
\begin{proof}
    Part (1) is \cite[Proposition\ 6.3]{loefflerWilliams2021p}, the rest is easily checked. 
\end{proof}

\subsection{Local systems on locally symmetric spaces} \label{sec:localsystems}

Let $A, K$ be as above and let $V$ be a $\Q$-linear algebraic representation of $A$. There are three standard constructions of local systems on $Y^{A}(K)$ attached to $V$:
\begin{itemize}
    \item a local system $\mathscr{V}_{\Q}$ of $\Q$-vector spaces, defined as the locally constant sections of
    \[
        A^{+}(\Q) \backslash \left[\left(A(\A_{f}) \times \mathcal{H}_{A}^{+}\right) \times V(\Q)\right]/K \to Y^{A}(K),
    \]
    with action $\gamma \cdot [(g, z), v] \cdot k = [(\gamma g k, \gamma z), \gamma v]$,
    \item a local system $\mathscr{V}_{\R}$ of $\R$-vector spaces, defined as the locally constant sections of
    \[
        A^{+}(\Q) \backslash \left[\left(A(\A_{f}) \times A^{+}(\R)\right) \times V(\R)\right]/KK^{\circ}_{A, \infty}Z^{\circ}_{A, \infty} \to Y^{A}(K),
    \]
    with action $\gamma \cdot [(g, z), v] \cdot kk_{\infty}z_{\infty} = [(\gamma g k, \gamma z k_{\infty}z_{\infty}), (k_{\infty}z_{\infty})^{-1}v]$,
    \item a local system $\mathscr{V}_{\Qp}$ of $\Qp$-vector spaces, defined as the locally constant sections of
    \[
        A^{+}(\Q) \backslash \left[\left(A(\A_{f}) \times \mathcal{H}_{A}^{+}\right) \times V(\Qp)\right]/K \to Y^{A}(K),
    \]
    with action $\gamma \cdot [(g, z), v] \cdot k = [(\gamma g k, \gamma z), k_{p}^{-1}v]$, where $k_{p}$ is the $p$-component of $k$.
\end{itemize}
If $V$ carries a $K_{p}$-stable lattice $V_{\Zp}$ (e.g.\ an admissible lattice as in Section \ref{sec:lattices}), the third construction yields a local system $\mathscr{V}_{\Zp}$ of $\Zp$-modules, and similarly with coefficients in $\mathcal{O}_{L}$ for $L/\Qp$ finite. There are comparison isomorphisms
\begin{equation} \label{eq:compisom}
    \mathscr{V}_{\Q} \otimes \Qp \cong \mathscr{V}_{\Qp}, \qquad \mathscr{V}_{\Q} \otimes \R \cong \mathscr{V}_{\R},
\end{equation}
the first given on sections by $[(g, z), v] \mapsto [(g, z), g_{p}^{-1}v]$, see \cite[\S 1.2]{urbaneigen}.

\begin{lemma} \label{lem:vecsecs}
There is a natural injective map 
\[
    V^{K_{p}} \to H^{0}(Y^{A}(K), \mathscr{V}_{\Qp})
\]
sending $v \mapsto \varphi_{v}$, where 
\[
    \varphi_{v}([g, z]) = ([g, z], v).
\] 
Moreover, if $t \in A(\Qp)$ normalises $K$ then the above map is equivariant for the action of $t$ on both sides, with the natural pullback action on $ \varphi \in H^{0}(Y^{A}(K), \mathscr{V}_{\Qp})$, given explicitly by 
\[
    (t \cdot \varphi)([g, z]) =  \varphi([gt, z]) \cdot t^{-1}.
\]
\end{lemma}
\begin{proof}
    The map is well-defined since for $\gamma \in A(\Q), k \in K$ we have $\varphi_{v}([\gamma g k, \gamma z]) = ([\gamma g k, \gamma z], v) = ([g, z], k_{p}v) = ([g, z], v)$ since $v \in V^{K_{p}}$ and is clearly injective. Given $t \in A(\Qp)$ normalising $K$ we have 
    \begin{align*}
        \varphi_{t\cdot v}([g, z]) &= ([g, z], t \cdot v) \\
        &= ([gtt^{-1}, z], t \cdot v) \\
        &= ([gt, z], v) \cdot t^{-1} \\
        &= \varphi_{v}([gt, z]) \cdot t^{-1}.
    \end{align*}
\end{proof}
\begin{remark}
    More generally for $g \in A(\Qp)$ the above map is compatible with the natural maps $V^{K_{p}} \to V^{g^{-1}K_{p}g}$ and $H^{0}(Y^{A}(K), \mathscr{V}_{\Qp}) \to H^{0}(Y^{A}(g^{-1}Kg), \mathscr{V}_{\Qp})$.
\end{remark}

\begin{definition} \label{def:secs} We define two families of sections.
\begin{enumerate}
\item For $n \geq 1,$ let $K_{n} \subset G(\A_{f})$ be an open compact subgroup whose reduction mod $p^{n}$ is contained in $Q_{H}^{0}(\Z/p^{n}\Z)$. For $0 \leq j \leq a,$ we write $f^{a, j}_{n}$ for the image of $f^{a, j}$ in $\left(V_{\lambda, \Zp} \otimes \nu_{1}^{-j}\right)/p^{n}$ and 
\[
    \varphi^{a, j}_{n} := \varphi_{f^{a, j}_{n}} \in H^{0}(Y^{G}(K_{n}), \left(\mathscr{V}_{\lambda} \otimes \nu_{1}^{-j}\right)(\Zp)/p^{n}).
\]
\item Let $K_{n}\subset G(\A_f)$ that reduces mod $p^{n}$ to a subgroup of the kernel of $\nu_{1}^{j},$ for some $j \geq 0.$ We write 
\[
    \varphi_{j, n} := \varphi_{1} \in H^{0}(Y^{G}(K_{n}), \nu_{1}^{j}/p^{n}),
\]
that is, the image of the element $1 \in \Zp \otimes \nu_{1}^{j}$ under the map of Lemma \ref{lem:vecsecs}. 
\end{enumerate}
\end{definition}
\subsection{Cohomology, translation maps and Hecke operators} \label{sec:heckoconv} \label{sec:heckecoh}

The constructions above are compatible with the projections $Y^{A}(K') \to Y^{A}(K)$ for $K' \subset K$, inducing pullback maps
\begin{equation} \label{eq:proj}
    H^{i}(Y^{A}(K), \mathscr{V}_{?}) \to H^{i}(Y^{A}(K'), \mathscr{V}_{?}), \qquad ? \in \{\Q, \Qp, \R\}.
\end{equation}
For $g \in A(\A_{f})$ with $g^{-1}Kg \subset K'$, right translation gives maps
\begin{align*}
    [g]^{K}_{K'}: Y^{A}(K) &\to Y^{A}(K'), \\
    [h, z] &\mapsto [hg, z],
\end{align*}
and hence pullback and pushforward maps
\begin{align*}
    \left([g]^{K}_{K'}\right)^{*}: H^{i}(Y^{A}(K'), \mathscr{V}_{?}) &\to H^{i}(Y^{A}(K), \mathscr{V}_{?}), \\
    \left([g]^{K}_{K'}\right)_{*}: H^{i}(Y^{A}(K), \mathscr{V}_{?}) &\to H^{i}(Y^{A}(K'), \mathscr{V}_{?}),
\end{align*}
 We write $[g]$ for $[g]^{K}_{K'}$ when $g^{-1}Kg = K'$. 
 \begin{remark}
 The maps $([1]^{K}_{K'})^{*}$ and $([1]^{K}_{K'})_{*}$ are the usual restriction and norm maps.
 \end{remark}
 The pullback maps assemble into an action of $A(\A_{f})$ on $H^{i}(Y^{A}, \mathscr{V}_{?}) := \varinjlim_{K} H^{i}(Y^{A}(K), \mathscr{V}_{?})$, and the cohomology at each level carries an action of the Hecke algebra $\Q[K \backslash A(\A_{f})/K]$, compatibly with \eqref{eq:compisom}.

\begin{definition} \label{def:iwahoridecomp}
    An open compact $K_{p} \subset \GL_{3}(\Zp)$ \emph{admits an Iwahori decomposition with respect to} $P_{1}$ if multiplication induces a bijection
    \[
        \left(\bar{N}_{1}(\Zp) \cap K_{p}\right) \times \left(L_{1}(\Zp) \cap K_{p}\right) \times \left(N_{1}(\Zp) \cap K_{p}\right) \xrightarrow{\ \sim\ } K_{p}.
    \]
\end{definition}
Associated to $u\tau^{n} \in G(\Qp)$ we have a pushforward map $[u\tau^{n}]_{*}: H^{3}(Y^{G}(K), \left(\mathscr{V}_{\lambda}\otimes \nu_{1}^{-j}\right)(\Qp)) \to H^{3}(Y^{G}((u\tau^{n})^{-1}Ku\tau^{n}), \left(\mathscr{V}_{\lambda}\otimes \nu_{1}^{-j}\right)(\Qp)).$ We can normalise this map to preserve the integral cohomology using the normalised $*$-action of $\tau^{-n}$ as in Section \ref{sec:normtau}. 
\begin{definition}
 Write 
\[
    [u\tau^{n}]_{*, j} := p^{(a - j)n}[u\tau^{n}]_{*}: H^{3}(Y^{G}(K), \left(\mathscr{V}_{\lambda}\otimes \nu_{1}^{-j}\right)(\Zp)) \to H^{3}(Y^{G}((u\tau^{n})^{-1}Ku\tau^{n}), \left(\mathscr{V}_{\lambda}\otimes \nu_{1}^{-j}\right)(\Zp)). 
\]
\end{definition}
Now take $K = K^{(p)}K_{p} \subset \GL_{3}(\A_{f})$ with $K_{p}$ admitting an Iwahori decomposition with respect to $P_{1}$. We have the correspondence
\[
\begin{tikzcd}[column sep = large]
    & Y^{\GL_{3}}(K \cap \tau K \tau^{-1})\arrow[dl, "{[1]}"'] \arrow[dr, "{[\tau]}"] & \\
    Y^{\GL_{3}}(K) & & Y^{\GL_{3}}(K),
\end{tikzcd}
\]
where the left map is the natural projection and the right map is translation by $\tau$ followed by the projection from level $\tau^{-1}K\tau \cap K$. We define Hecke operators, 
\[
    U'_{p} := p^{a - j}\left({[1]}^{\tau^{-1}K\tau \cap K}_{K}\right)_{*}\circ [\tau]_{*}\circ \left({[1]}^{K \cap \tau K\tau^{-1}}_{K}\right)^{*},
\]
on $H^{i}_{?}(Y^{\GL_{3}}(K), \left(\mathscr{V}_{\lambda} \otimes \nu_{1}^{-j}\right)(\Zp))$, $? \in \{\emptyset, c\}$, and 
\[
    U_{p} := p^{a-  j}\left({[1]}^{K \cap \tau K \tau^{-1}}_{K}\right)_{*}\circ [\tau^{-1}]_{*} \circ \left({[1]}^{\tau^{-1}K\tau \cap K}_{K}\right)^{*},
\]
on $H^{i}_{?}(Y^{\GL_{3}}(K), \left(\mathscr{V}_{\lambda} \otimes \nu_{1}^{-j}\right)^{\vee}(\Zp))$, $? \in \{\emptyset, c\}$
where the action of $\tau^{\mp 1}$ on the coefficients $\mathscr{V}_{\lambda, \Zp}$ is the normalised $\ast$-action of \eqref{eq:staraction}, which preserves the integral lattice. With this normalisation $U_{p} = U_{p, 1}$ agrees with the operator of Section \ref{sec:heckeopsandprefs} on the $\Pi$-isotypic part. Poincar\'e duality gives a perfect pairing
\begin{equation} \label{eq:poincare}
    \langle -, - \rangle_{K}: \frac{H^{3}(Y^{\GL_{3}}(K), \mathscr{V}_{\lambda}(\mathcal{O}))}{(\mathrm{torsion})} \times \frac{H_{c}^{2}(Y^{\GL_{3}}(K), \mathscr{V}^{\vee}_{\lambda}(\mathcal{O}))}{(\mathrm{torsion})} \to \mathcal{O}
\end{equation}
for any finite extension $\mathcal{O}/\Zp$, given by cup product, the pairing $\mathscr{V}_{\lambda} \otimes \mathscr{V}_{\lambda}^{\vee} \to R$, and integration over the $5$-dimensional manifold $Y^{\GL_{3}}(K)$. Since pullback and pushforward are adjoint under Poincar\'e duality, the operators $U_{p}$ and $U'_{p}$ (acting on the two factors) are adjoint to one another under $\langle -, -\rangle_{K}$.

\subsection{Iwasawa cohomology}

We extend cohomology to compact subgroups which are not necessarily open.

\begin{definition}
Let $V$ be a $\mathbb{Q}_p$-linear algebraic representation of $A$, let
$V_{\mathbb{Z}_p}\subset V$ be an admissible lattice, and let $\mathscr{V}_{\mathbb{Z}_p}$ be
the associated local system. Let $B=U^{(p)}B_p\subset A(\mathbb{A}_f)$, where
$U^{(p)}\subset A(\mathbb{A}_f^{(p)})$ is open compact and $B_p\subset A(\mathbb{Z}_p)$ is
compact. The \emph{Iwasawa cohomology} of $B$ is
$$H^i_{\mathrm{Iw}}(Y^A(B),\mathscr{V}_{\mathbb{Z}_p}):=\varprojlim_{U_p\supset B_p}
H^i(Y^A(U^{(p)}U_p),\mathscr{V}_{\mathbb{Z}_p}),$$
the limit over (a cofinal system of) open compact $U_p\subset A(\mathbb{Z}_p)$ containing
$B_p$, with respect to the norm maps. We define $H^i_{\mathrm{Iw}}(\tilde{Y}^H(B),-)$ for
$B\subset H(\mathbb{A}_f)$ in the same way.
\end{definition}

Let $B$ and $\mathscr{V}_{\mathbb{Z}_p}$ be as in the definition, and for $n\geq1$ let
$U_n\subset A(\mathbb{Z}_p)$ be the preimage of the image of $B_p$ in
$A(\mathbb{Z}/p^n\mathbb{Z})$. Let $W_{\mathbb{Z}_p}$ be an admissible lattice in
another $\mathbb{Q}_p$-linear algebraic representation of $A$. For $w\in W_{\mathbb{Z}_p}^{B_p},$ let $w_n$ be its image in $(W_{\mathbb{Z}_p}/p^n)^{U_n}$, and let $\varphi_{w_n}\in H^0(Y^A(U^{(p)}U_n),\mathscr{W}_{\mathbb{Z}_p}/p^n)$ be the section
attached to $w_n$ by the mod $p^n$ analogue of Lemma \ref{lem:vecsecs}. Then there is a cup product
pairing
\begin{equation} \label{eq:iwcupprod}
H^i_{\mathrm{Iw}}(Y^A(B),\mathscr{V}_{\mathbb{Z}_p})\times W^{B_p}_{\mathbb{Z}_p}\to
H^i_{\mathrm{Iw}}(Y^A(B),\mathscr{V}_{\mathbb{Z}_p}\otimes\mathscr{W}_{\mathbb{Z}_p}),\qquad
(z,w)\mapsto\varprojlim_n\,(z_{U_n}\bmod p^n)\cup\varphi_{w_n}.
\end{equation}

\subsection{Twist operators} \label{sec:twistdescent}

Let $K \subset G(\A_{f})$ be an open compact subgroup that surjects onto $\left(\Z/p^{n}\Z\right)^{\times}\times(\Z/D\Z)^{\times}$ via $\nu_{1} \times \nu_{2}$ for some $n \geq 1$, with $D \geq 1$ prime to $p$, and consider the level $K(p^{n}, D)$. Fix a set of representatives $\{g_{1}, \ldots, g_{d}\}$ for the map 
\[
\nu_{1} \times \nu_{2}: G^{+}(\Q) \backslash G(\A_{f}) \times \mathcal{H}_{G}^{+}/K(p^{n}, D) \to \Delta_{1, n} \times \Delta_{2}.
\]
given by projecting the map of Lemma \ref{lem:components} to the finite part, where $\Delta_{1, n} \cong \left(\Z/p^{n}\Z\right)^{\times}, \Delta_{2} := (\Z/D\Z)^{\times}$ with the identifications sending uniformisers $\varpi_{\ell}$ for $\ell$ away from $p$ (resp. $D$) to $\ell \bmod p^{n}$ (resp. $\bmod \ D$).
\begin{definition} \label{def:twist}
    Set $\Delta := \Delta_{1, n} \times \Delta_{2}$, then for $j \in \Z$ define the \emph{twisting map}
    \begin{align*}
        \Tw_{j}: \mathrm{Maps}(\Delta, \Zp) &\to \mathrm{Maps}(\Delta, \Zp), \\ f &\mapsto \left(g_{i} \mapsto \left(\norm{\nu_{1}(g_{i})}\nu_{1}(g_{i,p}) \right)^{-j}f(\left(\nu_{1} \times \nu_{2}\right)(g_{i})) \right).
    \end{align*}
\end{definition}
Note that $\mathrm{tw}_{j}$ depends on the choice of $g_{i}$; replacing $g_{i}$ with $\tilde{g}_{i}$ in the same equivalence class changes $\left(\norm{\nu_{1}(g_{i})}\nu_{1}(g_{i,p}) \right)^{-j}$ by an element of $1 + p^{n}\Zp$.  By Lemma \ref{lem:components},
\[
    Y^{G}(K(p^{n}, D)) \cong \Delta_{1, n} \times \Delta_{2} \times Y^{G}(K)
\]
  For each $(x, y) \in \Delta_{1, n} \times \Delta_{2}$ fix a base point $g_{x, y} \in G(\A_{f})$ in the corresponding union of connected components.
\begin{lemma}
 There is a canonical isomorphism
\begin{equation} \label{eq:groupringiso}
    \Psi_{j}: H^{3}\big(Y^{G}(K(p^{n}, D)), (\mathscr{V}_{\lambda} \otimes \nu_{1}^{-j})(\Zp)\big) \xrightarrow{\ \sim\ } H^{3}\big(Y^{G}(K), \left(\mathscr{V}_{\lambda} \otimes \nu_{1}^{-j}\right)(\Zp)\big) \otimes \mathrm{Maps}(\Delta_{1, n} \times \Delta_{2}, \Zp).
\end{equation}
\end{lemma}
\begin{proof}
We sketch the proof. Write $Y' = Y^{G}(K(p^{n}, D))$ and $Y = Y^{G}(K)$, let $\pi: Y' \to Y$ for the natural finite covering map and write $\Zp$ for the locally constant sheaf associated to the trivial representation. By the projection formula and the Leray spectral sequence we have 
\[
    H^{i}(Y, \left(\mathscr{V}_{\lambda} \otimes \nu_{1}^{-j}\right)(\Zp) \otimes \pi_{*}\Zp) \cong H^{i}(Y', \left(\mathscr{V}_{\lambda} \otimes \nu_{1}^{-j}\right)(\Zp)).
\]
We have a canonical (independent of the choice of $g_{x, y}$) isomorphism of constant sheaves $\pi_{*}\Zp \cong \mathrm{Maps}(\Delta_{1, n} \times \Delta_{2}, \Zp)$ and thus an isomorphism 
\[
     \Psi_{j}: H^{i}(Y, \left(\mathscr{V}_{\lambda} \otimes \nu_{1}^{-j} \right)(\Zp) \otimes \pi_{*}\Zp) \cong  H^{i}(Y, \left(\mathscr{V}_{\lambda} \otimes \nu_{1}^{-j}\right)(\Zp)) \otimes_{\Zp} \mathrm{Maps}(\Delta_{1, n} \times \Delta_{2}, \Zp).
\]
\end{proof}
The choice of $\{g_{x, y}: (x, y) \in \Delta_{1, n} \times \Delta_{2}\}$ fixes a twisting map 
\[
    \mathrm{tw}_{j}:  \mathrm{Maps}(\Delta_{1, n} \times \Delta_{2}, \Zp) \to  \mathrm{Maps}(\Delta_{1, n} \times \Delta_{2}, \Zp)
\]
as in Definition \ref{def:twist}. 
\begin{lemma}\label{lem:twistcoh}
Let $j \in \Z$, $A \in \{G, H\}$ and $\mathscr{V}$ the locally constant sheaf associated to an algebraic representation of $A(\Qp)$. For any open compact $K \subset A(\A_{f})$ define a section 
\[
    \mathbb{S}^{A}_{j}(g , z) = \left([g, z], (\norm{\nu_{1}(g)}\nu_{1}(g_{p}))^{-j}\right) \in H^{0}(Y^{A}(K), \nu_{1}^{j})
\]
of the locally constant sheaf associated to $\nu_{1}^{j}$. Let $K \subset A(\A_{f})$, then cup product with $\mathbb{S}^{A}_{j}$ defines a map  
\[
    T_{j}^{A}:= \left( \cdot \cup \mathbb{S}^{A}_{j}\right): H^{i}\big(Y^{A}(K), (\mathscr{V} \otimes \nu_{1}^{-j})(\Zp)\big) \to H^{i}\big(Y^{A}(K), \mathscr{V}(\Zp)\big). 
\]
Then
\begin{enumerate}
\item We have 
\[
    \iota^{*} \circ T^{G}_{j}  = T^{H}_{j} \circ \iota^{*}. 
\]
We will therefore henceforth write $\mathbb{S}_{j}$ for $\mathbb{S}^{A}_{j}$ and $T_{j}$ for cup product with this element. 
\item Let $K \subset G(\A_{f})$, let  $\mathbb{S}_{j, n} \in H^{0}(Y^{G}(K), \nu_{1}^{j}/p^{n})$ be the reduction of $\mathbb{S}_{j}$ modulo $p^{n}$ and $T_{j, n}$ the cup product with this element, and suppose $K$ surjects onto $\left(\Z/p^{n}\Z\right)^{\times}\times(\Z/D\Z)^{\times}$ via $\nu_{1} \times \nu_{2}$. Writing $M_n=\mathrm{Maps}(\Delta_{1,n}\times\Delta_2,\Zp)$, the diagram
\[
\begin{tikzcd}[column sep = huge]
    H^{3}\big(Y^{G}(K(p^{n}, D)), (\mathscr{V}_{\lambda} \otimes \nu_{1}^{-j})(\Zp)/p^{n}\big) \arrow[r, "\cup \varphi_{j, n}"] \arrow[d, "\Psi_{j}"'] & H^{3}\big(Y^{G}(K(p^{n}, D)), \mathscr{V}_{\lambda}(\Zp)/p^{n}\big) \arrow[d, "\Psi_{0}"] \\
    H^{3}\big(Y^{G}(K), \left(\mathscr{V}_{\lambda} \otimes \nu_{1}^{-j}\right)(\Zp)/p^{n}\big) \otimes M_{n} \arrow[r, "T_{j, n}\otimes \Tw_{-j}"] & H^{3}\big(Y^{G}(K), \mathscr{V}_{\lambda}(\Zp)/p^{n}\big) \otimes  M_{n}
\end{tikzcd}
\]
commutes.
\item We have the relation 
\[
    T_{j} \circ [u\tau^{n}]_{*, j} = [u\tau^{n}]_{*, 0} \circ T_{j}.
\]
In particular, $U_{p}' \circ T_{j} = T_{j} \circ U'_{p}$. 
\end{enumerate}

\end{lemma}

\begin{proof}
Clear by construction. 
\end{proof}

Note that the diagram in part (2) is canonical since the ambiguity in the choices defining $\mathrm{tw}_{j}$ are $1 \bmod p^{n}$. Indeed, $\mathrm{tw}_{j}$ is given explicitly modulo $p^{n}$ by multiplication by the function 
\[
    [(x, y)] \mapsto x^{j},
\]
where we identify $\Delta_{1, n} \times \Delta_{2} \cong \left(\Z/p^{n}\Z\right)^{\times} \times \left(\Z/D\Z\right)^{\times}$ via $(\delta_{1}, \delta_{2}) \mapsto (x, y) = \left(\left(\norm{\delta_{1}}\delta_{1, p}\right)^{-1}, \left(\norm{\delta_{2}}\delta_{2, D}\right)^{-1} \right) \bmod \left(1 + p^{n}\hat{\Z}\right) \times \left(1 + D\hat{\Z}\right)$, with $(\delta_{1}, \delta_{2}) \in \left(\mathbb{A}_{f}^{\times}\right)^{2}$ a lift of an element of $\Delta_{1, n} \times \Delta_{2}$.

\subsection{Automorphic cohomology classes} \label{sec:autoclasses}

Let $\Pi$ be a RACAR of $\GL_{3}(\A)$ of weight $\lambda = (a, 0, -a)$, and let $K \subset \GL_{3}(\A_{f})$ be open compact with $\Pi_{f}^{K} \neq 0$. Via cuspidal cohomology and the choice of generator $\zeta_{\infty}$ of \eqref{eq:zetainfty} there is a Hecke-equivariant injection (see \cite[(3.3)]{loefflerWilliams2021p})
\[
    \phi_{\Pi}^{K}: \mathcal{W}(\Pi_{f})^{K} \hookrightarrow H^{2}_{c}(Y^{\GL_{3}}(K), \mathscr{V}_{\lambda, \Q}^{\vee}(\C)),
\]
compatible with change of $K$, assembling to a $\GL_{3}(\A_{f})$-equivariant map
\[
    \phi_{\Pi}: \mathcal{W}(\Pi_{f}) \hookrightarrow H^{2}_{c}(Y^{\GL_{3}}, \mathscr{V}_{\lambda, \Q}^{\vee}(\C));
\]
we write $H^{2}_{c}(\Pi, \C)$ for its image.

\subsubsection{Periods and rationality} \label{sec:pernrat}
There is a minimal number field $E$, the field of definition of $\Pi$, such that
\[
    H^{2}_{c}(\Pi, E) := H^{2}_{c}(\Pi, \C) \cap H^{2}_{c}(Y^{\GL_{3}}, \mathscr{V}_{\lambda, \Q}^{\vee}(E))
\]
is non-zero and gives an $E$-structure on $H^{2}_{c}(\Pi, \C)$; by strong multiplicity one, $E$ may be taken to be the field generated by the Hecke eigenvalues of $\Pi$ at unramified primes. Comparing with the Whittaker $E$-structure defines, by \cite[Proposition\ 3.1]{clozel2}, a period $\Theta_{\Pi} \in \C^{\times}$ with
\[
    \phi_{\Pi}\left(\mathcal{W}(\Pi_{f}, E)\right) = \Theta_{\Pi} \cdot H^{2}_{c}(\Pi, E).
\]
Having fixed $\zeta_{\infty}$, the period $\Theta_{\Pi}$ is well defined modulo $E^{\times}$ and rescaling $\zeta_{\infty}$ rescales $\Theta_{\Pi}$. Henceforth, replacing $E$ by a sufficiently large finite extension and extending the preceding rational structures by scalars, we assume that the chosen refinement and all fixed auxiliary characters and algebraic test data introduced below are defined over $E$. We extend the chosen $p$-adic embedding accordingly.

\subsubsection{Integral structures on cohomology}
Let $v \mid p$ be a place of $E$, with completion $E_{v}$. The $\mathcal{O}_{E_{v}}$-points of $V^{\vee}_{\lambda, \Z}$ carry an action of $\GL_{3}(\Zp)$, giving a local system $\mathscr{V}^{\vee}_{\lambda}(\mathcal{O}_{E_{v}})$ of $\mathcal{O}_{E_{v}}$-modules on $Y^{\GL_{3}}(\mathcal{U})$ as in Section \ref{sec:localsystems}. The group $H^{2}_{c}(Y^{\GL_{3}}(\mathcal{U}), \mathscr{V}_{\lambda}^{\vee}(\mathcal{O}_{E_{v}}))$ is a finitely generated $\mathcal{O}_{E_{v}}$-module, and its quotient by torsion is a lattice in $H^{2}_{c}(Y^{\GL_{3}}(\mathcal{U}), \mathscr{V}_{\lambda}^{\vee}(E_{v}))$. As in \cite[\S 3.3.3]{loefflerWilliams2021p}, given a finite set $\{W_{i}\} \subset \mathcal{W}(\Pi_{f}, E)^{\mathcal{U}}$ and a finite extension $L/\Qp$ into which $E$ embeds, the classes
\[
    \phi_{\Pi}(W_{i})/\Theta_{\Pi} \in H^{2}_{c}(Y^{\GL_{3}}(\mathcal{U}), \mathscr{V}_{\lambda}^{\vee}(L))
\]
can, after rescaling $\Theta_{\Pi}$ by a $p$-adic unit times an element of $E^{\times}$, be normalised to lie in
\[
H^{2}_{c}(Y^{\GL_{3}}(\mathcal{U}), \mathscr{V}_{\lambda}^{\vee}(\mathcal{O}_{L}))/\{\text{torsion}\}.
\]
Such a normalisation depends on $\{W_{i}\}$, which will be harmless as our test data is fixed once and for all in Section \ref{sec:padicL}.

\section{Eisenstein series and Eisenstein classes for $\GL_{2}$} \label{sec:eis}

In this section we recall, largely following \cite[\S 5]{loefflerWilliams2021p}, the theory of Eisenstein series attached to adelic Schwartz functions, the motivic Eisenstein classes of Be\u{\i}linson, and their $p$-adic interpolation due to Kings. Throughout this section, symmetric spaces are for $\GL_{2}$ unless indicated otherwise.

\subsection{Schwartz functions}
For a field $K$ of characteristic $0$, let $\mathcal{S}(\A_{f}^{2}, K)$ denote the space of locally constant, compactly supported functions $\A_{f}^{2} \to K$, and $\mathcal{S}_{0}(\A_{f}^{2}, K)$ the subspace of $\Phi$ with $\Phi(0, 0) = 0$. The group $\GL_{2}(\A_{f})$ acts on these spaces by $(g \cdot \Phi)(v) = \Phi(vg)$.

If $\chi$ is a Dirichlet character, let $R_{\chi}$ denote the projection to the $\hat{\chi}^{-1}$-isotypical component:
\[
    R_{\chi}(\Phi_{f})(x, y) := \int_{a \in \hat{\Z}^{\times}}\hat{\chi}(a)\,\Phi_{f}(ax, ay)\,d^{\times}a.
\]

\subsection{Eisenstein series} \label{sec:eisseries}
Let $\Phi = \Phi_{\infty}\Phi_{f} \in \mathcal{S}(\A^{2}, \C)$ (Schwartz at infinity) and let $\hat{\chi}$ be a finite-order Hecke character. For $g \in \GL_{2}(\A)$ and $\mathfrak{Re}(s) \gg 0$, define the Godement--Siegel section and Eisenstein series
\[
    f^{\Phi}(g; \hat{\chi}, s) := \|\det g\|^{s}\int_{\A^{\times}}\Phi\left((0, a)g\right)\hat{\chi}(a)\|a\|^{2s}\,d^{\times}a, \qquad
    E^{\Phi}(g; \hat{\chi}, s) := \sum_{\gamma \in B_{\GL_{2}}(\Q)\backslash\GL_{2}(\Q)}f^{\Phi}(\gamma g; \hat{\chi}, s),
\]
which have meromorphic continuation to all $s$. The series is analytic if $\Phi(0,0) = \hat{\Phi}(0,0) = 0$ or if $\hat{\chi}$ is ramified at some finite place. Classically, for $\Phi_{f} \in \mathcal{S}_{(0)}(\A_{f}^{2}, K)$ and $\tau$ in the upper half-plane one sets
\[
    E^{j + 2}_{\Phi_{f}}(\tau; s) := \frac{\Gamma(s + \tfrac{j+2}{2})}{(-2\pi i)^{j + 2}\pi^{s - \frac{j+2}{2}}}\sum_{(m, n) \in \Q^{2}\setminus (0,0)}\frac{\Phi_{f}(m, n)\, y^{s - \frac{j+2}{2}}}{(m\tau + n)^{j + 2}|m\tau + n|^{2s - j - 2}},
\]
extended to a function $E^{j+2}_{\Phi_{f}}(g; s)$ on $\GL_{2}(\A)$ by $\GL_{2}(\A_{f})$-equivariance in $\Phi_{f}$ and the standard automorphy recipe at infinity \cite[\S I]{weil71}, and we further set $E^{j+2}_{\Phi_{f}}(g; \chi, s) := E^{j+2}_{R_{\chi}(\Phi_{f})}(g; s)$. Taking
\begin{equation} \label{eq:phiinfty}
    \Phi_{\infty}^{j + 2}(x, y) := 2^{-1 - j}(x + iy)^{j + 2}\exp(-\pi(x^{2} + y^{2})),
\end{equation}
the adelic and classical series are related by $E^{\Phi}(g; \hat{\chi}, s) = \|\det g\|^{s}E^{j+2}_{\Phi_{f}}(g; \chi, s)$ for $\Phi = \Phi_{\infty}^{j+2}\Phi_{f}$ \cite[Proposition\ 5.2, Cor.\ 5.3]{loefflerWilliams2021p}. We will be interested exclusively in the special value $s = -j/2$, and set
\begin{equation} \label{eq:eisspecial}
    \mathcal{E}^{j + 2, \chi}_{\Phi_{f}}(g) := E^{j + 2}_{\Phi_{f}}\left(g; \chi, -\tfrac{j}{2}\right) = \|\det g\|^{j/2}\,E^{R_{\chi}(\Phi)}\left(g; \hat{\chi}, -\tfrac{j}{2}\right),
\end{equation}
which depends $\GL_{2}(\A_{f})$-equivariantly on $\Phi_{f}$ and transforms as a vector in the principal series $I(\|\cdot\|^{-1/2}, \hat{\chi}^{-1}\|\cdot\|^{j + 1/2})$.

\subsection{Betti--Eisenstein classes} \label{sec:bettieis}
For $j \geq 0$ let $\mathscr{V}_{(0, -j)}$ denote the local system on $Y^{\GL_{2}}$ attached to the $\GL_{2}$-representation $V^{\GL_{2}}_{(0, -j)}$ of highest weight $(0, -j)$, as in Section \ref{sec:localsystems}. The main geometric input to our construction is the following theorem of Be\u{\i}linson \cite{beilinson86}, in the form given in \cite[Theorem\ 5.5, Cor.\ 5.6]{loefflerWilliams2021p} (see also \cite[\S 7]{LZGsp}): for every $j \geq 0$ and every level $U$, there is a canonical $\GL_{2}(\A_{f})$-equivariant map, the \emph{Betti--Eisenstein symbol},
\[
    \mathcal{S}_{(0)}(\A_{f}^{2}, \Q)^{U} \to H^{1}\left(Y^{\GL_{2}}(U), \mathscr{V}_{(0, -j)}(\Q)\right), \qquad \Phi_{f} \mapsto \Eis^{j}_{\Phi_{f}},
\]
compatible with Hecke operators and change of level, whose image under the Betti--de Rham comparison is the differential form $-\mathcal{E}^{j + 2}_{\Phi_{f}}(g)\,(dz)^{\otimes j}\,d\tau$. Here the notation $\mathcal{S}_{(0)}$ means $\mathcal{S}_{0}$ if $j = 0$ and $\mathcal{S}$ otherwise.

These classes are not integral, but they admit integral `$c$-smoothings': for $c > 1$ coprime to $6$ and to the level, taking $\Phi_{f}$ spherical at $c$ i.e. of the form $\Phi^{(c)}_{f} \times \prod_{\ell \mid c}\mathrm{ch}\left(\Z_{\ell}^{2}\right)$, there are classes ${}_{c}\Eis^{j}_{\Phi_{f}} \in H^{1}(Y^{\GL_{2}}(U), \mathscr{V}_{(0,-j)}(\Zp))$, compatible with Hecke operators and change of level, satisfying
\begin{equation} \label{eq:csmooth}
    {}_{c}\Eis^{j}_{\Phi_{f}} = \left(c^{2} - c^{-j}\begin{psmallmatrix} c & \\ & c\end{psmallmatrix}^{-1}\right)\Eis^{j}_{\Phi_{f}}
\end{equation}
after inverting $p$, where $\begin{psmallmatrix} c & \\ & c \end{psmallmatrix} \in \GL_{2}(\A_{f}^{(c)})$, see \cite[Theorem\ 5.7]{loefflerWilliams2021p}.

\subsection{The Eisenstein--Iwasawa class} \label{sec:EIclass}
Set $\Lambda = \Zp[[\Zp^{\times}]]$, fix a prime-to-$p$ level $\mathcal{U}^{\GL_{2}, (p)}$, and integer $c$ coprime to $6p$ and to all level data and a prime-to-$p$ Schwartz function $\Phi^{(p)} \in \mathcal{S}(\A_{f}^{(p), 2}, \Z)^{\mathcal{U}^{\GL_{2}, (p)}}$ spherical at $c$. For $t \geq 1$ let $\Phi_{p, t} := \mathrm{ch}\left((0, 1) + p^{t}\Zp^{2}\right)$ and $\Phi_{f, t} := \Phi^{(p)}\Phi_{p, t}$, which is stable under $\mathcal{U}^{\GL_{2}}_{t}$ and compatible with the trace maps as $t$ varies. The resulting classes with $j = 0$ assemble into the \emph{Eisenstein--Iwasawa class}
\[
    {}_{c}\mathcal{EI}_{\Phi^{(p)}} := \left({}_{c}\Eis^{0}_{\Phi_{f, t}}\right)_{t \geq 1} \in H^{1}_{\Iw}\left(Y^{\GL_{2}}(\mathcal{U}^{\GL_{2}}_{\infty}), \Zp\right) := \varprojlim_{t}H^{1}\left(Y^{\GL_{2}}(\mathcal{U}^{\GL_{2}}_{t}), \Zp\right).
\]
This Iwasawa cohomology group is naturally a $\Lambda$-module and admits moment maps
\[
    \mathrm{mom}^{j}_{t}: H^{1}_{\Iw}\left(Y^{\GL_{2}}(\mathcal{U}^{\GL_{2}}_{\infty}), \Zp\right) \to H^{1}\left(Y^{\GL_{2}}(\mathcal{U}^{\GL_{2}}_{t}), \mathscr{V}_{(0, -j)}(\Zp)\right)
\]
for each $j \geq 0$ and $t \geq 1$ (see \cite[\S 9.1]{LZGsp}), and the fundamental theorem of Kings on the $\Lambda$-adic interpolation of Eisenstein classes (see \cite{kings2013eisenstein}, or \cite{KLZ} for a summary in the present notation) states that
\begin{equation} \label{eq:kings}
    \mathrm{mom}^{j}_{t}\left({}_{c}\mathcal{EI}_{\Phi^{(p)}}\right) = {}_{c}\Eis^{j}_{\Phi_{f, t}} \qquad \text{for all } j \geq 0, \ t \geq 1.
\end{equation}
All of the above is compatible with the action of $\GL_{2}(\A_{f}^{(pc)})$.

\section{$p$-adic distributions} \label{sec:analytic}

In this section we recall the  spaces of $p$-adic distributions and the interpolation results needed to construct our $p$-adic $L$-function.

\subsection{Spaces of analytic functions and distributions}
Let $L/\Qp$ be a complete extension. 

\begin{definition}
We define the following spaces:
\begin{itemize}
    \item For an integer $a \geq 0$ let $LP^{[0, a]}(\Zp^{\times}, L)$ denote the space of locally polynomial functions on $\Zp^{\times}$ of degree $\leq a$.
    \item Let $\mathscr{D}_{\mathrm{alg}}^{[0, a]}(\Zp^{\times}, L)$ denote the $L$-dual of $LP^{[0, a]}(\Zp^{\times}, L)$.
    \item Let $LA(\Zp^{\times}, L)$ denote the space of locally analytic functions $f: \Zp^{\times} \to L$. We equip $LA(\Zp^{\times}, L)$ with the inductive limit topology given by the Banach subspaces $LA_{h}(\Zp^{\times}, L)$ of functions which are analytic on discs of radius $p^{-h},$ for $h\geq 1,$ equipped with the following norm: For $f \in LA_h(\Zp^{\times}, L)$ given on the disc $z + p^{h}\Zp$, $z \in \Zp^{\times}$, by $f(z + p^{h}t) = \sum_{k \geq 0}a_{z, k}t^{k}$, define 
    \[
        \norm{f}_{LA_{h}} := \mathrm{sup}_{z \in \Zp^{\times}}\mathrm{sup}_{k}\vert a_{z, k}\vert.
    \]
    \item Let $\mathscr{D}_{\ell a}(\Zp^{\times}, L)$ denote the continuous $L$-dual of $LA(\Zp^{\times}, L)$, referred to as the space of \textit{locally analytic distributions} $\mu: LA(\Zp^{\times}, L) \to L$ on $\Zp^{\times}$. It has the structure of a Fréchet space defined by the family of seminorms $(\norm{\cdot}_{h})_{h \geq 0}$, defined for $\mu \in \mathscr{D}_{\ell a}(\Zp^{\times}, L)$ by
    \[
        \norm{\mu}_{h} := \mathrm{sup}_{0 \neq f \in LA_{h}(\Zp^{\times}, L)} \frac{\vert \mu(f)\vert_{p}}{\norm{f}_{LA_{h}}}.
    \]
     We equip  $\mathscr{D}_{\ell a}(\Zp^{\times}, L)$ with the action of $\delta \in \Zp^{\times}$ given by 
    \[
        \int_{\Zp^{\times}}f(x)(\delta \cdot \mu)(x) = \int_{\Zp^{\times}} f(\delta^{-1}x)\mu(x),
    \]
    where for $f \in LA(\Zp^{\times}, L)$ we use the standard notation $\int_{\Zp^{\times}}f(x)\mu(x) := \mu(f)$. 
\end{itemize}
\end{definition}

For $h \in \mathbb{R}_{\geq 0}$ there is a filtration $\mathscr{D}_{h}(\Zp^{\times}, L) \subset \mathscr{D}_{\ell a}(\Zp^{\times}, L)$ by Banach spaces, defined in \cite[II.3.1]{colmez}. For our purposes we are interested in the following properties of these spaces.

\begin{proposition} \label{prop:interp}
Let $h \in \mathbb{R}_{\geq 0}$ and let  $\mu \in \mathscr{D}_{h}(\Zp^{\times}, L)$. Then 
\begin{enumerate}
    \item Suppose that for all Dirichlet characters $\chi$ of $p$-power conductor and for all $0 \leq j \leq \lfloor h \rfloor$ we have 
    \[
        \int_{\Zp^{\times}}\chi(x)x^{j}\mu(x) = 0,
    \]
    then $\mu \equiv 0$. 
    \item There is $C \in \R$ such that for all $z \in \Zp^{\times}$, $n \geq 1 , j \geq 0$
    \[
        \left\vert \int_{z + p^{n}\Zp}\left(\frac{x - z}{p^{n}}\right)^{j}\mu(x)\right\vert \leq Cp^{hn}.
    \]
    \item Suppose for some $a \geq \lfloor h \rfloor$, all $0 \leq j \leq a$, $z \in \Zp^{\times}$, and $n \geq 1$ we have a constant $C \in \R$ and  $\mu_{\mathrm{alg}} \in \mathscr{D}_{\mathrm{alg}}^{[0, a]}(\Zp^{\times}, L)$ satisfying  
    \[
        \left\vert \int_{z + p^{n}\Zp}\left(\frac{x - z}{p^{n}}\right)^{j}\mu_{\mathrm{alg}}(x)\right\vert \leq Cp^{hn}.
    \]
    Then there is a unique $\mu \in \mathscr{D}_{h}(\Zp^{\times}, L)$ whose restriction to $LP^{[0, a]}(\Zp^{\times}, L)$ is $\mu_{\mathrm{alg}}$.
\end{enumerate}
\end{proposition}
\begin{proof}
Parts (2) and (3) for distributions on $\Zp$ are \cite[Th\'eor\`eme II.3.2]{colmez}; we can deduce the result for $\Zp^{\times}$ by considering $\mu$ as a distribution on $\Zp$ supported on $\Zp^{\times}$. For the first part finite Fourier inversion on $\left(\Z/p^{n}\Z\right)^{\times}$ implies that Dirichlet characters $\chi$ of $p$-power conductor are a basis for the space of locally constant functions on $\Zp^{\times}$, and thus that functions of the form $x \mapsto \chi(x)x^{j}$ for $0 \leq j \leq \lfloor h \rfloor$ are a basis for $LP^{[0, \lfloor h \rfloor]}(\Zp^{\times}, L)$. Then parts (2) and (3) say that $\mu$ is uniquely determined by its restriction to $LP^{[0, \lfloor h \rfloor]}(\Zp^{\times}, L)$, which is identically zero by assumption. 
\end{proof}

\section{Constructing the distribution} \label{sec:pairing}

We now assemble the ingredients of the previous sections into a $p$-adic distribution attached to a small slope $P_{1}$-refined RACAR. Throughout this section we fix:
\begin{itemize}
    \item A RACAR $\Pi$ of $\GL_{3}(\A)$ of weight $\lambda = (a, 0, -a)$, together with a regular unramified $P_{1}$-refinement $\sigma_{p}$, with $\alpha_{p} = \sigma_{p}(p)$ and $r = r(\tilde{\Pi})$ as in Proposition \ref{prop:newvector}. We write
    \[
        \alpha := p^{a + 1}\alpha_{p}, \qquad h := v_{p}(\alpha) \in [0, 2a + 2]
    \]
    for the normalised $U_{p, 1}$-eigenvalue and its slope. We further assume that $\tilde{\Pi}$ satisfies the small slope hypothesis $h < a + 1$, noting that this enters only in Theorem \ref{thm:growth}.
    \item An integer $c > 1$ coprime to $6p$.
    \item A prime-to-$p$ Schwartz function $\Phi^{(p)} \in \mathcal{S}(\A_{f}^{(p), 2}, \Z)$ spherical at $c$ and prime-to-$p$ levels $\mathcal{U}^{G, (p)} = \mathcal{U}^{\mathrm{GL}_{3}, (p)} \times \prod_{\ell \neq p}\Z_{\ell}^{\times}$, $\mathcal{U}^{H, (p)} = \mathcal{U}^{G, (p)} \cap H(\A_{f}^{(p)})$ and  $\mathcal{U}^{\GL_{2}, (p)} = \mathrm{pr}_{\mathrm{GL}_{2}}\left(\mathcal{U}^{H, (p)}\right)$, where $\mathrm{pr}_{\mathrm{GL}_{2}}$ is the projection to the $\mathrm{GL}_{2}$-factor of $H$, such that $\Phi^{(p)}$ is fixed by $\mathcal{U}^{H, (p)}$. We assume the resulting levels are spherical at $c$ and suitably small  in the sense that 
    \[
        \Gamma_{g}:= G^{+}(\Q) \cap g\mathcal{U}^{G, (p)}\mathcal{U}_{p}g^{-1}
    \]
    acts freely on $\mathcal{H}^{+}_{G}$ for  \footnote{Note that levels $\mathcal{U}^{G, (p)}\mathcal{U}_{p}$ with $\mathcal{U}_{p} = \mathcal{U}^{\mathrm{GL}_{3}} \times \hat{\Z}^{\times}$ are not neat since $(\mathrm{id}, -1) \in \mathcal{U}^{G, (p)}$, but they can still be sufficiently small (after possibly shrinking $\mathcal{U}^{\mathrm{GL}_{3}, (p)}$) since if $(\gamma, t) \in \Gamma_{g}$ then $t \in \Q_{>0} \cap \hat{\Z}^{\times} = \{1\}$.} any $g \in G(\A_{f})$ and open-compact $\mathcal{U}_{p} \subset G(\Q_{p})$,
    shrinking $\mathcal{U}^{\mathrm{GL}_{3},(p)}$ if necessary.
    \item An auxiliary even Dirichlet character $\eta_{2}$ of conductor $D$ prime to $p$ and $c$, valued in $\mathcal{O}_{L}$, with $\mathcal{U}^{G, (p)}(1, D) \subset \ker(\hat{\eta}_{2}\circ\nu_{2})$. The role of this character is solely in removing the smoothing factor (Theorem \ref{thm:growth}(2)), and its only effect on the final result is on the period $\Omega^{-}_{\Pi}$,
    
\end{itemize}

\subsection{The classes $z^{[a, j]}_{n}$} \label{sec:classesdef}

Recall the Eisenstein--Iwasawa class ${}_{c}\mathcal{EI}_{\Phi^{(p)}} \in H^{1}_{\Iw}(Y^{\GL_{2}}(\mathcal{U}^{\GL_{2}}_{\infty}), \Zp)$ of Section \ref{sec:EIclass}. Pulling back along the natural maps $\tilde{Y}^{H} \xrightarrow{f} Y^{H} \to Y^{\GL_{2}}$ of \eqref{eq:roof} yields a class. We define two maps
\begin{equation} \label{eq:iotamap}
    \iota_{a, j, n}: H^{1}_{\Iw}\left(Y^{\mathrm{GL}_{2}}(\mathcal{U}_{\infty}^{\mathrm{GL}_{2}}), \Zp\right) \longrightarrow H^{3}\left(Y^{\GL_{3}}(\mathcal{V}^{\GL_{3}}_{n}), \left(\mathscr{V}_{\lambda} \otimes \nu_{1}^{-j}\right)(\Zp)\right) \otimes \mathrm{Maps}\left(\Delta_{1, n} \times \Delta_{2}, \mathcal{O}\right)
\end{equation}
and 
\begin{equation} \label{eq:iotaLmap}
    \iota_{a, j, n, t}^{L} : H^{1}\left(Y^{\mathrm{GL}_{2}}(\mathcal{U}^{\mathrm{GL}_{2}}_{t}), \mathscr{V}_{(0, -j)}(\Zp)\right) \longrightarrow H^{3}\left(Y^{\GL_{3}}(\mathcal{V}^{\GL_{3}}_{n}), \mathscr{V}_{\lambda}(\Zp)\right) \otimes \mathrm{Maps}\left(\Delta_{1, n} \times \Delta_{2}, \mathcal{O}\right),
\end{equation}
the first given by the composition  
\begin{align*} \label{diag:map}
    H^{1}_{\Iw}(Y^{\mathrm{GL}_{2}}(\mathcal{U}^{\mathrm{GL}_{2}}_{\infty}),\Zp)&\to H^{1}_{\Iw}\left(\tilde{Y}^{H}(Q_{H,\infty,n}^{0}(p^{n}, D)), \Zp\right)  \\
    & \xrightarrow{\cup f^{a, j}} H^{1}_{\Iw}\left(\tilde{Y}^{H}(Q_{H, \infty, n}^{0}(p^{n}, D)), \iota^{*}\left(\mathscr{V}_{\lambda} \otimes \nu_{1}^{-j}\right)(\Zp)\right) \\
    & \xrightarrow{\iota_{*}} H^{3}_{\Iw}(Y^{G}(Q_{H, \infty, n}^{0}(p^{n}, D)), \left(\mathscr{V}_{\lambda} \otimes \nu_{1}^{-j}\right)(\Zp)) \\
    & \xrightarrow{[u\tau^{n}]_{*, j}} H^{3}_{\mathrm{Iw}}\left(Y^{G}((u\tau^{n})^{-1}Q_{H, \infty, n}^{0}(p^{n}, D)(u\tau^{n})), \left(\mathscr{V}_{\lambda} \otimes \nu_{1}^{-j}\right)(\Zp)\right) \\
    & \xrightarrow{\mathrm{pr}_{n}} H^{3}\left(Y^{G}(\mathcal{V}^{G}_{n}(p^{n}, D)), \left(\mathscr{V}_{\lambda} \otimes \nu_{1}^{-j}\right)(\Zp)\right) \\
    &\xrightarrow{\Psi_{j}} H^{3}_{\Iw}\left(Y^{G}(\mathcal{V}^{G}_{n}), \left(\mathscr{V}_{\lambda} \otimes \nu_{1}^{-j}\right)(\Zp)\right) \otimes \mathrm{Maps}\left(\Delta_{1, n} \times \Delta_{2}, \Zp\right),
\end{align*}
where we use the convention that when an $H$-subgroup is the level at $p$ in $G$-Iwasawa cohomology, the tame level is taken to be $\mathcal{U}^{G, (p)}$. The second map is given for any $t \geq \mathrm{max}\{n, r\}$ by 
\begin{align*}
 H^{1}(Y^{\mathrm{GL}_{2}}(\mathcal{U}^{\mathrm{GL}_{2}}_{t}), \mathscr{V}_{(0, -j)}(\Zp))&\to H^{1}\left(\tilde{Y}^{H}(Q_{H,t, n}^{0}(p^{n}, D)), \mathscr{V}_{(0, -j; 0)}(\Zp)\right) \\ 
&\xrightarrow{T_{j}} H^{1}\left(\tilde{Y}^{H}(Q_{H,t, n}^{0}(p^{n}, D)), \mathscr{V}_{(j, 0; -j)}(\Zp)\right) \\
&\xrightarrow{\mathrm{br}^{[a,j]}}
 H^{1}(\tilde{Y}^{H}(Q_{H,t, n}^{0}(p^{n}, D)), \iota^{*}\mathscr{V}_{\lambda}(\Zp))
 \\
    & \xrightarrow{\iota_{*}} H^{3}\left(Y^{G}(u\mathcal{U}^{G}_{n}(p^{n}, D)u^{-1}), \mathscr{V}_{\lambda}(\Zp)\right) \\
    & \xrightarrow{[u\tau^{n}]_{*, 0}} H^{3}\left(Y^{G}(\mathcal{V}^{G}_{n}(p^{n}, D)), \mathscr{V}_{\lambda}(\Zp)\right) \\
    &\xrightarrow{\Psi_{0}} H^{3}\left(Y^{G}(\mathcal{V}^{G}_{n}), \mathscr{V}_{\lambda}(\Zp)\right) \otimes \mathrm{Maps}\left(\Delta_{1, n} \times \Delta_{2}, \Zp\right)
\end{align*}
The map $\iota_{a, j, n}$ is useful for proving $p$-adic interpolation, whereas the map $\iota^{L}_{a, j, n, t}$ is useful for relations to $L$-values. We will routinely abuse notation and also use $\iota^{L}_{a, j, n, t}$ to refer to the composition following initial pullback from $Y^{\mathrm{GL}_{2}}$.

\begin{lemma} \label{lem:mapcomparison}
The following diagram commutes.
\[
\begin{tikzcd}
    H^{1}_{\Iw}\left(\tilde{Y}^{H}(\mathcal{U}_{\infty}^{H}(1, D)), \Zp\right) \arrow[d, "\mathrm{mom}^{H}_{j, t}"] \arrow[r, "\iota_{a, j, n}"] & H^{3}\left(Y^{G}(\mathcal{V}^{G}_{n}), \left(\mathscr{V}_{\lambda} \otimes \nu_{1}^{-j}\right)(\Zp)\right) \otimes \mathrm{Maps}\left(\Delta_{1, n} \times \Delta_{2}, \Zp\right) \arrow[d, "T_{j} \otimes 1"]  \\
     H^{1}\left(\tilde{Y}^{H}(\mathcal{U}_{t}^{H}(1, D)), \mathscr{V}_{(0, -j; 0)}(\Zp)\right)  \arrow[r, "\iota^{L}_{a, j, n, t}"] & H^{3}\left(Y^{G}(\mathcal{V}^{G}_{n}), \mathscr{V}_{\lambda}(\Zp)\right) \otimes \mathrm{Maps}\left(\Delta_{1, n} \times \Delta_{2}, \Zp\right)
\end{tikzcd}
\]
In particular, this composition is independent of $t$. 
\end{lemma}
\begin{proof}
We first note that we have a diagram 
\begin{equation} \label{fig:commrep}
\begin{tikzcd}
    \Zp \arrow[d, "b"] \arrow[r, "a"] & V_{\lambda, \Zp} \otimes \nu_{1}^{-j} \arrow[d, equal] \\
    V^{H}_{(0, -j; 0), \Zp} \arrow[r, "c"] & V_{\lambda, \Zp}\otimes \nu_{1}^{-j}
\end{tikzcd}
\end{equation}
where the $a$ and $b$ are $Q_{H}^{0}(\Zp)$-equivariant, and $c$ is $H(\Zp)$-equivariant. Let $v_{j} \in  V^{H}_{(0, -j; 0), \Zp}$ be a choice of highest weight vector, then we normalise $a$ to send $1 \mapsto f^{a, j}$, $b$ to send $1 \mapsto v_{j}$, and $c$ to send $v_{j} \mapsto f^{a, j}$. 

Let $z \in  H^{1}_{\Iw}\left(\tilde{Y}^{H}(Q_{H, \infty, n}^{0}(p^{n}, D)), \Zp\right)$ and write 
\[
\mathrm{pr}_{t, n}: H^{1}_{\Iw}\left(\tilde{Y}^{H}(Q_{H, \infty, n}^{0}(p^{n}, D)), \mathscr{V}^{H}_{(0, -j; 0)}(\Zp)\right) \to H^{1}_{\Iw}\left(\tilde{Y}^{H}(Q_{H, t, n}^{0}(p^{n}, D)), \mathscr{V}^{H}_{(0, -j; 0)}(\Zp)\right)
\]
be the projection map, so that $\mathrm{mom}_{j, t, n}^{H} = \mathrm{pr}_{t, n} \circ \left( - \cup v_{j}\right)$. Since $T_{j} \otimes 1$ is an isomorphism with inverse $T_{-j} \otimes 1$, it suffices to show that $\iota_{a, j ,n} = \left(T_{-j} \otimes 1\right) \circ \iota^{L}_{a, j, n, t} \circ \mathrm{mom}^{H}_{j, t}$. Write $\mathcal{P}_{n, j} := [u\tau^{n}]_{*, j} \circ \iota_{*}$.  Writing $w = \mathrm{mom}^{H}_{j, t, n}(z)$, we then have 
\begin{align*}
\left(T_{-j} \otimes 1\right) \circ \iota_{a, j, n, t}^{L}(w) &= \left(T_{-j} \otimes 1\right) \circ\left(\Psi_{0} \circ \mathcal{P}_{n, 0} \circ \mathrm{br}^{[a,j]} \left(w \cup \mathbb{S}_{j}\right)\right)\\
&=  \Psi_{j} \circ T_{-j} \circ \mathcal{P}_{n, 0} \circ \mathrm{br}^{[a,j]} \left(w\cup \mathbb{S}_{j}\right) \\
&= \Psi_{j} \circ \mathcal{P}_{n, j} \left( \mathrm{br}^{[a, j]} \left(w \cup \mathbb{S}_{j}\right) \cup \mathbb{S}_{-j} \right) \\
&= \Psi_{j} \circ \mathcal{P}_{n, j}\left( \left(\mathrm{br}^{[a, j]} \otimes \nu_{1}^{-j}\right) \left(w\right) \right),
\end{align*}
where $\mathrm{br}^{[a,j]} \otimes \nu_{1}^{-j}$ is the morphism of sheaves induced by the map of $H$-representations  $V_{(0, -j; 0)}^{H} \to V_{\lambda} \otimes \nu_{1}^{-j}$, the second equality uses that cupping with $\mathbb{S}_{j}$ is equivariant for the normalised action of $\tau^{n}$, and the third equality uses Lemma \ref{lem:twistcoh}, the projection formula and the fact that $\iota^{*}\mathbb{S}_{-j} = \mathbb{S}_{-j}$. In other words, $\left(T_{-j} \otimes 1\right) \circ \iota_{a, j, n, t}^{L}$ is given by the composition:
\begin{align*}
 H^{1}(Y^{\mathrm{GL}_{2}}(\mathcal{U}^{\mathrm{GL}_{2}}_{t}), \mathscr{V}_{(0, -j)}(\Zp))&\to H^{1}\left(\tilde{Y}^{H}(Q_{H, t, n}^{0}(p^{n}, D)), \mathscr{V}_{(0, -j; 0)}(\Zp)\right) \\ 
&\xrightarrow{\mathrm{br}^{[a,j]} \otimes \nu_{1}^{-j}}
 H^{1}(\tilde{Y}^{H}(Q_{H, t, n}^{0}(p^{n}, D)), \iota^{*}\left(\mathscr{V}_{\lambda}\otimes \nu_{1}^{-j}\right)(\Zp))
 \\
    & \xrightarrow{\iota_{*}} H^{3}_{\Iw}\left(Y^{G}(u\mathcal{U}^{G}_{n}(p^{n}, D)u^{-1}), \left(\mathscr{V}_{\lambda}\otimes \nu_{1}^{-j}\right)(\Zp)\right) \\
    & \xrightarrow{[u\tau^{n}]_{*, j}} H^{3}\left(Y^{G}(\mathcal{V}^{G}_{n}(p^{n}, D)), \left(\mathscr{V}_{\lambda}\otimes \nu_{1}^{-j}\right)(\Zp)\right) \\
    &\xrightarrow{\Psi_{j}} H^{3}\left(Y^{G}(\mathcal{V}^{G}_{n}), \left(\mathscr{V}_{\lambda}\otimes \nu_{1}^{-j}\right)(\Zp)\right) \otimes \mathrm{Maps}\left(\Delta_{1, n} \times \Delta_{2}, \Zp\right)
\end{align*}
Since both compositions defining $\iota_{a, j, n}$, and $ \left(T_{-j} \otimes 1\right) \circ \iota^{L}_{a, j, n, t}$ terminate in the same map, it suffices to show that 
\begin{equation} \label{eq:composition}
    [u\tau^{n}]_{*, j} \circ \iota_{*} \circ \mathrm{br}^{[a, j]} \otimes \nu_{1}^{-j} \circ \mathrm{pr}_{t, n}\left(z \cup v_{j}\right)= \mathrm{pr}_{n} \circ [u\tau^{n}]_{*, j} \circ \iota_{*}\left(z \cup \iota^{*}f^{a, j} \right).
\end{equation}
The expression $z \cup \iota^{*}f^{a, j}$ is the inverse limit over $t$ of the mod $p^{t}$ expressions 
\[
\mathrm{pr}_{t, n}(z) \cup \left(\iota^{*}\varphi_{t}^{a, j}\right) \in H^{1}(\tilde{Y}^{H}(Q_{H, t, n}^{0}(p^{n}, D)), \iota^{*}\left(\mathscr{V}_{\lambda}\otimes \nu_{1}^{-j}\right)(\Zp)/p^{t}).
\] 
On the other hand, the composition $\mathrm{br}^{[a, j]} \otimes \nu_{1}^{-j} \circ \mathrm{pr}_{t, n}\left(z \cup v_{j}\right)$ reduces modulo $p^{t}$ to the expression $\mathrm{br}^{[a, j]} \otimes \nu_{1}^{-j} \circ \left(\mathrm{pr}_{t, n}\left(z\right) \cup \left(\varphi_{v_{j}} \bmod p^{t}\right)\right)$, so we deduce from \eqref{fig:commrep} that 
\[
    \mathrm{br}^{[a, j]} \otimes \nu_{1}^{-j}\circ \left(\mathrm{pr}_{t, n}\left(z \cup v_{j}\right)\right)  \equiv \mathrm{pr}_{t, n}\left(z\cup \iota^{*}f^{a, j} \right) \bmod p^{t}.
\]
The map $\iota_{*}: H^{1}(\tilde{Y}^{H}(Q_{H, t, n}^{0}(p^{n}, D)), \iota^{*}\left(\mathscr{V}_{\lambda} \otimes \nu_{1}^{-j}\right)(\Zp))
\to H^{3}_{\Iw}\left(Y^{G}(u\mathcal{U}^{G}_{n}(p^{n}, D)u^{-1}), \left(\mathscr{V}_{\lambda} \otimes \nu_{1}^{-j}\right)(\Zp)\right)$ factors as 
\[
\begin{aligned}
    &H^{1}(\tilde{Y}^{H}(Q_{H, t, n}^{0}(p^{n}, D)), \iota^{*}\left(\mathscr{V}_{\lambda} \otimes \nu_{1}^{-j}\right)(\Zp)) \\
&\quad\to H^{3}_{\Iw}\left(Y^{G}(Q_{H, t, n}^{0}(p^{n}, D)), \left(\mathscr{V}_{\lambda} \otimes \nu_{1}^{-j}\right)(\Zp)\right) \\
&\quad\to H^{3}\left(Y^{G}(u\mathcal{U}^{G}_{n}(p^{n}, D)u^{-1}), \left(\mathscr{V}_{\lambda} \otimes \nu_{1}^{-j}\right)(\Zp)\right).
\end{aligned}
\]
and the following diagram commutes
\[
\begin{tikzcd}
    H^{3}_{\Iw}\left(Y^{G}(Q_{H, \infty, n}^{0}(p^{n}, D)), \left(\mathscr{V}_{\lambda} \otimes \nu_{1}^{-j}\right)(\Zp)\right) \arrow[r, "{[u\tau^{n}]}_{\star}"] \arrow[d, ""] & H^{3}_{\Iw}\left(Y^{G}((u\tau^{n})^{-1}Q_{H, \infty, n}^{0}(p^{n}, D)(u\tau^{n})), \left(\mathscr{V}_{\lambda} \otimes \nu_{1}^{-j}\right)(\Zp)\right) \arrow[d,  ""] \\
    H^{3}\left(Y^{G}(u\mathcal{U}^{G}_{n}(p^{n}, D)u^{-1}), \left(\mathscr{V}_{\lambda} \otimes \nu_{1}^{-j}\right)(\Zp)\right) \arrow[r, "{[u\tau^{n}]}_{\star}"] & H^{3}\left(Y^{G}(\mathcal{V}^{G}_{n}(p^{n}, D)), \left(\mathscr{V}_{\lambda} \otimes \nu_{1}^{-j}\right)(\Zp)\right)
\end{tikzcd}
\]
from which we deduce that
\[
[u\tau^{n}]_{\star} \circ \iota_{*} \circ \mathrm{br}^{[a, j]} \otimes \nu_{1}^{-j} \circ \mathrm{pr}_{t, n}\left(z \cup v_{j}\right) \equiv \mathrm{pr}_{n} \circ [u\tau^{n}]_{\star} \circ \left(\iota_{*}\left(z \cup \iota^{*}f^{a, j} \right)\right) \bmod p^{t},
\]
i.e. as elements of $H^{3}\left(Y^{G}(\mathcal{V}^{G}_{n}(p^{n}, D)), \left(\mathscr{V}_{\lambda} \otimes \nu_{1}^{-j}\right)(\Zp)/p^{t}\right)$. Taking the limit over $t$, we obtain the result.
\end{proof}
 
\begin{definition} \label{def:zclass}
  For $0 \leq j \leq a$ and $n \geq 1$, set
\[
    {_{c}}z^{[a, j]}_{n}\left(\Phi^{(p)}\right) := \iota_{a, j, n}\left({}_{c}\mathcal{EI}_{\Phi^{(p)}}\right) \in H^{3}\left(Y^{\GL_{3}}(\mathcal{V}^{\GL_{3}}_{n}), \left(\mathscr{V}_{\lambda} \otimes \nu_{1}^{-j}\right)(\Zp)\right) \otimes \mathrm{Maps}\left(\Delta_{1, n} \times \Delta_{2}, \Zp\right).
\]
 Further define `unsmoothed' versions
\[
z^{[a, j]}_{n}(\Phi^{(p)}) := \left(T_{-j} \otimes 1\right)\iota^{L}_{a, j, n, t}\left(\mathrm{Eis}_{\Phi_{f, t}}^{j}\right)
\]
using the unsmoothed Eisenstein classes $\mathrm{Eis}_{\Phi_{f, t}}^{j}$ recalled in Section \ref{sec:bettieis}.
\end{definition}

Note that by Lemma \ref{lem:mapcomparison} we have ${_{c}}z^{[a, j]}_{n}\left(\Phi^{(p)}\right) \cup \mathbb{S}_{j} = \iota^{L}_{a, j, n, t}\left({_{c}}\mathrm{Eis}^{j}_{\Phi_{f, t}}\right)$ 

\subsection{The norm relation}
We show that the classes ${}_{c}z^{[a, j]}_{n}(\Phi^{(p)}) \cup \mathbb{S}_{j}$ satisfy a norm relation as we vary $n$. Let 
\[
    \Norm^{\Delta_{1, n + 1}}_{\Delta_{1, n}}: \mathrm{Maps}\left(\Delta_{1, n + 1} \times \Delta_{2}, \Zp\right) \to \mathrm{Maps}\left(\Delta_{1, n} \times \Delta_{2}, \Zp\right)
\]
be the map satisfying $\mathrm{Norm}^{\Delta_{1, n + 1}}_{\Delta_{1, n}}(f)((\delta_{1, n} , \delta_{2})) = \sum_{(\varepsilon_{1, n+1}, \delta_{2})} f((\varepsilon_{1, n + 1}, \delta_{2}))$, where the sum is over all $(\varepsilon_{1, n+1}, \delta_{2}) \in \Delta_{1, n + 1} \times \Delta_{2}$ that map to $(\delta_{1, n}, \delta_{2})$ in $\Delta_{1,n}\times \Delta_{2}$.  
\begin{proposition} \label{prop:normrelation}
For all $n \geq 1$ we have
\[
    \left[\left([1]^{\mathcal{V}_{n+1}}_{\mathcal{V}_{n}}\right)_{*} \otimes \Norm^{\Delta_{1, n + 1}}_{\Delta_{1, n}}\right]\left({}_{c}z^{[a, j]}_{n + 1}(\Phi^{(p)}) \cup \mathbb{S}_{j}\right) = \left(U'_{p} \otimes 1\right)\cdot {}_{c}z^{[a, j]}_{n}(\Phi^{(p)})\cup \mathbb{S}_{j}
\]
in $H^{3}(Y^{\GL_{3}}(\mathcal{V}^{\GL_{3}}_{n}), \mathscr{V}_{\lambda}(\Zp)) \otimes \mathrm{Maps}\left(\Delta_{1, n} \times \Delta_{2}, \Zp\right)$, and similarly for $z^{[a, j]}_{n}$.
\end{proposition}

\begin{proof}
 This proof follows the general shape of  \cite[Proposition\ 4.5.2, \S 4.6, \S 5.2.3]{loefflerspherical}, and \cite[Theorem\ 6.10]{loefflerWilliams2021p}. We sketch the argument. Recalling that 
 \[
 {_{c}}z^{[a, j]}_{n}\left(\Phi^{(p)}\right) \cup \mathbb{S}_{j} = \iota^{L}_{a, j, n, t}\left({_{c}}\mathrm{Eis}^{j}_{\Phi_{f, t}}\right)
 \]
 and writing $\tilde{z}_{n} \in H^{3}\left(Y^{G}(\mathcal{U}^{G}_{n}(p^{n}, D)), \mathscr{V}_{\lambda}(\Zp)\right)$ for the classes before the $\tau^{n}$-translation in the composition defining $\iota_{a, j, n, t}^{L}$, one has
\begin{equation} \label{eq:normcompatproof}
    \left([1]^{\mathcal{U}^{G}_{n + 1}(p^{n + 1}, D)}_{\mathcal{U}^{G}_{n + 1}(p^{n}, D)}\right)_{*}(\tilde{z}_{n + 1}) = \left([1]^{\mathcal{U}^{G}_{n + 1}(p^{n}, D)}_{\mathcal{U}^{G}_{n}(p^{n}, D)}\right)^{*}(\tilde{z}_{n}).
\end{equation}
Indeed, for $t \geq \mathrm{max}\{n + 1, r\}$ we have a diagram 
\begin{equation}
\begin{tikzcd} \label{fig:cart}
\tilde{Y}^{H}\left(Q_{H, t, n + 1}^{0}(p^{n + 1}, D)\right) \arrow[r, "{[u]\circ\iota}"] \arrow[d] & Y^{G}(\mathcal{U}^{G}_{n + 1}(p^{n}, D)) \arrow[d] \arrow [r, "{[}\tau{]}"] & Y^{G}(\tau^{-1}\mathcal{U}^{G}_{n + 1}(p^{n}, D)\tau) \arrow[d] \\
    \tilde{Y}^{H}\left(Q_{H, t, n}^{0}(p^{n}, D)\right) \arrow[r, "{[u]\circ\iota}"] & Y^{G}(\mathcal{U}^{G}_{n}(p^{n}, D)) & Y^{G}(\mathcal{U}^{G}_{n}(p^{n}, D)).
\end{tikzcd}
\end{equation}
The bottom left square is cartesian: to show this it suffices to show that the map of cosets 
\[
    Q^{0}_{H, t, n}(p^{n}, D)/Q^{0}_{H, t, n + 1}(p^{n + 1}, D) \to \mathcal{U}_{n}^{G}(p^{n}, D)/\mathcal{U}_{n + 1}^{G}(p^{n}, D)
\]
induced by $u \circ \iota$ is a bijection. Since both vertical maps have degree $p^{2}$ by Lemma \ref{lem:inclusion}(1) it suffices to show injectivity which is equivalent to showing that $Q^{0}_{H, t, n}(p^{n}, D) \cap u\mathcal{U}_{n + 1}^{G}(p^{n}, D)u^{-1} = Q^{0}_{H, t, n + 1}(p^{n + 1}, D)$ which is easily seen. The correspondence defined by the right-hand horseshoe in \eqref{fig:cart} is the unnormalised Hecke operator $p^{-a}U_{p}'$. It follows that 
\[
    [\tau]^{\mathcal{U}^{G}_{n + 1}(p^{n + 1}, D)}_{\mathcal{U}^{G}_{n}(p^{n}, D), *}\left(\tilde{z}_{n + 1}\right) = p^{-a}U_{p}' \cdot \tilde{z}_{n},
\]
where $[\tau]^{\mathcal{U}^{G}_{n + 1}(p^{n + 1}, D)}_{\mathcal{U}^{G}_{n}(p^{n}, D)}: Y^{G}(\mathcal{U}_{n + 1}^{G}(p^{n + 1}, D)) \to Y^{G}(\mathcal{U}_{n}^{G}(p^{n}, D))$ is the map induced by pushforward along the top and right arrows in the horseshoe. Applying $[\tau^{n}]_{*}$ we obtain 
\[
    [1]^{\mathcal{V}^{G}_{n + 1}(p^{n + 1}, D)}_{\mathcal{V}^{G}_{n}(p^{n}, D), *}\left([\tau^{n + 1}]_{*, 0}\tilde{z}_{n + 1}\right) = U_{p}' \cdot [\tau^{n}]_{*, 0}\tilde{z}_{n}.
\]
These maps are all compatible with $\Psi_{0}$ and so we get the result. 
\end{proof}

\subsection{Pairing with a finite-slope eigenclass} \label{sec:eigenpairing}

Let $\varphi_{f} = \otimes\varphi_{\ell} \in \Pi_{f}$ be such that:
\begin{enumerate}
    \item[(a)] $\varphi_{p}$ is the $P_{1}$-refined newvector: $\varphi_{p}$ is fixed by $\mathcal{U}^{(P_{1})}_{1, p}(p^{r})$ and $U_{p, 1}\varphi_{p} = \alpha\,\varphi_{p}$;
    \item[(b)] $W_{\varphi_{f}} \in \mathcal{W}(\Pi_{f}, E)$ is algebraic, fixed by $\mathcal{U}^{\GL_3, (p)}$ away from $p$.
\end{enumerate}
Enlarging $L$ if necessary, and normalising $\Theta_{\Pi}$ as in Section \ref{sec:autoclasses}, we obtain a class
\[
    \phi = \phi_{\varphi_{f}} := \phi_{\Pi}(W_{\varphi_{f}})/\Theta_{\Pi} \in H^{2}_{c}\left(Y^{\GL_{3}}(\mathcal{U}), \mathscr{V}^{\vee}_{\lambda}(\mathcal{O}_{L})\right)/\{\text{torsion}\},
\]
and we write $\phi_{n}$ for its pullback to level $\mathcal{V}^{\GL_{3}}_{n} \subset \mathcal{U}$.

\begin{definition} \label{def:Xi}
For $n \geq 1$ and $0 \leq j \leq a$, let $C = \vert \vol\left(\mathcal{U}^{G, (p)}(1, D) \cap H(\A_{f}^{(p)})\right) \vert_{p}$ and define
\[
    {}_{c}\tilde{\Xi}^{[a; j]}_{n}\left(\phi, \Phi^{(p)}\right) := \vol\left(\mathcal{U}^{G, (p)}(1, D) \cap H(\A_{f}^{(p)})\right)\cdot\left(1 \otimes e_{\eta_{2}}\right)\left\langle {}_{c}z^{[a, j]}_{n}(\Phi^{(p)}) \cup \mathbb{S}_{j},\ \phi_{n}\right\rangle_{\mathcal{V}_{n}} \in C^{-1} \cdot \mathrm{Maps}(\Delta_{1, n}, \mathcal{O}),
\]
where $\langle -, -\rangle_{\mathcal{V}_{n}}$ is the Poincar\'e pairing \eqref{eq:poincare} applied at level $\mathcal{V}_{n}$, $e_{\eta_{2}}:  \mathrm{Maps}(\Delta_{2}, \mathcal{O}_{L}) \to \mathcal{O}_{L}$ is the map $f \mapsto \sum_{x \in \Delta_{2}} \eta_{2}(x) f(x)$, and the volume is with respect to $dh_{f}$. We further set
\[
    {}_{c}\Xi^{[a; j]}_{n} := \alpha^{-n}\cdot{}_{c}\tilde{\Xi}^{[a; j]}_{n} \in C^{-1}p^{-hn}\,\mathcal{O}_{L}[\Delta_{1, n}],
\]
where $h := v_{p}(\alpha)$, and define $\tilde{\Xi}^{[a;j]}_{n}, \Xi^{[a; j]}_{n}$ (with $L$-coefficients) analogously using $z^{[a, j]}_{n}$.
\end{definition}
\begin{remark}\label{rem:LWcomparison}
The map $(1 \otimes e_{\eta_2}) \circ \iota^L_{a,j,n,t}$
recovers the construction of \cite[\S 6]{loefflerWilliams2021p},
with the same tame levels and branching-map normalisations.
Here we identify
\[
    \mathrm{Maps}(\Delta_{1,n},\mathcal O_L)
    \cong \mathcal O_L[\Delta_{1,n}],
    \qquad f \longmapsto \sum_{x \in \Delta_{1,n}} f(x)[x].
\]
In particular,
\[
    (1 \otimes e_{\eta_2})
    \iota^L_{a,j,n,t}\left({}_c\mathrm{Eis}^j_{\Phi_{f,t}}\right)
    = {}_c\xi^{[a,j]}_{n,\mathrm{LW}},
\]
where the right-hand side denotes the class constructed in
\textit{op.\ cit.} Indeed, Lemma \ref{lem:inclusion} identifies the source level
with that used in their construction. Their coefficient twist,
given in the rational local-system model by multiplication by
$\norm{\nu_1(h)}^{-j}$, becomes multiplication by
\[
    \norm{\nu_1(h)}^{-j}\nu_1(h_p)^{-j}
    = \mathbb S_j(h)
\]
under the comparison isomorphism \eqref{eq:compisom}.
It therefore agrees with $T_j$. Moreover, on the union of connected components labelled by
$(x,y) \in \Delta_{1,n} \times \Delta_2$, applying $\Psi_0$
and then $1 \otimes e_{\eta_2}$ records the contribution with
weight $\eta_2(y)[x]$. This is precisely the effect of cupping
with $[\nu_{1,(n)}]$ and $\widehat{\eta}_2 \circ \nu_2$ in
\textit{op.\ cit.} The translation by $u\tau^n$ preserves
these component labels. The remaining operations are the same
branching map, pushforward and translation, with total
normalisation $p^{an}$, incorporated here in
$[u\tau^n]_{*,0}$. This proves the claimed identification.

The same comparison holds for the unsmoothed classes.
Consequently, pairing with the same automorphic cohomology
class and applying the same volume and eigenvalue
normalisations gives the finite-level quantities computed
in \cite[\S 7]{loefflerWilliams2021p}.
\end{remark}
As in \cite[Proposition\ 6.11]{loefflerWilliams2021p}, the volume normalisation makes ${}_{c}\tilde{\Xi}^{[a;j]}_{n}$ independent of the choice of $\mathcal{U}^{G, (p)}$ fixing the data.

\begin{corollary} \label{cor:normcompatXi}
For all $n \geq 1$ we have $\Norm^{\Delta_{1, n + 1}}_{\Delta_{1, n}}\left({}_{c}\Xi^{[a; j]}_{n + 1}\right) = {}_{c}\Xi^{[a; j]}_{n}$.
\end{corollary}
\begin{proof}
Write $z_{n} := {}_{c}z^{[a, j]}_{n}(\Phi^{(p)}) \cup \mathbb{S}_{j}$. Since $\phi_{n + 1} = ([1]^{\mathcal{V}_{n+1}}_{\mathcal{V}_{n}})^{*}\phi_{n}$, adjointness of pullback and pushforward under \eqref{eq:poincare} gives
\[
    \Norm\left\langle {}_{c}z_{n+1}, \phi_{n + 1}\right\rangle = \left\langle\left[([1])_{*}\otimes\Norm\right]{}_{c}z_{n + 1}, \phi_{n}\right\rangle = \left\langle (U'_{p}\otimes 1){}_{c}z_{n}, \phi_{n}\right\rangle = \left\langle {}_{c}z_{n}, U_{p}\phi_{n}\right\rangle = \alpha\left\langle {}_{c}z_{n}, \phi_{n}\right\rangle,
\]
using Proposition \ref{prop:normrelation}, the adjointness of $U_{p}$ and $U'_{p}$, and (a). Multiplying by $\alpha^{-(n + 1)}$ gives the claim.
\end{proof}

\subsection{Congruences between twists} \label{sec:classcong}

We now prove that the constructed classes ${}_{c}\Xi^{[a, j]}_{n}$ satisfy congruences as $j$ varies. Recall that we identify 
\[
\Delta_{1,n} := \Q^{\times}_{> 0}\backslash \mathbb{A}^{\times}_{f}/ \nu_{1}\left(\mathcal{V}_{n}(p^{n})\right)  \cong \left(\Z/p^{n}\Z\right)^{\times}
\]
using the conventions of \cite{loefflerWilliams2021p} that a uniformiser $\varpi_{\ell}$ at $\ell \neq p$ is sent to $\ell \bmod p^{n}$. In particular, $z \in \Zp^{\times} \subset \mathbb{A}^{\times}_{f}$ is sent to $z^{-1} \bmod p^{n}$. For all $n \geq 1, 0 \leq j \leq a$ recall that in Lemma \ref{lem:twistcoh} we defined a section $\mathbb{S}_{j, n} \in H^{0}(Y_{G}(\mathcal{V}_{n}(p^{n}, D)), \nu_{1}^{j}/p^{n})$ whose image under the isomorphism to $\mathrm{Maps}(\left(\Z/p^{n}\Z\right)^{\times} \times \Delta_{2}, \mathcal{O})$ is $(z, \delta) \mapsto z^{j}$ and is the reduction of $\mathbb{S}_{j}$ modulo $p^{n}$. 
For the proof of the following Proposition we abuse notation somewhat and write 
$\mathrm{res}^{nk}_{n}$ for the pullback \[H^{3}_{\Iw}(Y^{G}(K(p^{n}, D)), V) \to H^{3}_{\Iw}(Y^{G}(K(p^{nk}, D)), V)
\]
for any closed subgroup $K$ and any appropriate local system $V$, the specifics being dictated by context, as well as for the natural pullback map 
\[
    \mathrm{Maps}\left(\left(\Z/p^{n}\Z\right)^{\times}, \Zp\right) \to \mathrm{Maps}\left(\left(\Z/p^{nk}\Z\right)^{\times}, \Zp\right),
\]
and any linear extension thereof.
\begin{proposition} \label{prop:classcong}
For all $1 \leq k \leq a$ and $n \geq 1$ the classes $\mathrm{res}_{n}^{nk}\left({}_{c}\Xi^{[a;j]}_{n}(\Phi^{(p)})\right) \in \mathrm{Maps}\left(\left(\Z/p^{nk}\Z\right)^{\times}, L\right)$ satisfy 
 \[
    \norm{p^{-kn}\sum_{j = 0}^{k}(-1)^{j}\binom{k}{j}\Tw_{-j}\left(\mathrm{res}_{n}^{nk}\left({}_{c}\Xi^{[a, j]}_{n}(\Phi^{(p)})\right)\right)} \leq Cp^{hn},
\]
where $\norm{\cdot}$ is the sup norm, and for $k = 0$ we have
\[
     \norm{{}_{c}\Xi^{[a, j]}_{n}(\Phi^{(p)})} \leq Cp^{hn}.
\]
\end{proposition}
\begin{proof} 
We first note that for $k = 0$ the statement of the proposition follows immediately from the integrality (up to the volume factor) of the classes ${}_{c}\tilde{\Xi}_{n}^{[a; j]}$. Set $\mathcal{R}_{n, j} = \mathrm{pr}_{n} \circ [u\tau^{n}]_{*, j} \circ \iota_{*}$. Note first that Lemma \ref{lem:twistcoh} (2) gives us 
\begin{align*}
   \mathrm{tw}_{-j} \circ \mathrm{res}^{nk}_{n}\left\langle {}_{c}z^{[a, j]}_{n}(\Phi^{(p)}) \cup \mathbb{S}_{j},\ \phi_{n}\right\rangle_{\mathcal{V}_{n}} &= \left\langle \mathrm{res}^{nk}_{n}\left({}_{c}z^{[a, j]}_{n}(\Phi^{(p)})\right) \cup \mathbb{S}_{j} \otimes \mathrm{tw}_{-j},\ \phi_{n}\right\rangle_{\mathcal{V}_{n}} \\
    &\equiv \left\langle \Psi_{0}\left( \mathrm{res}^{nk}_{n}\circ \mathcal{R}_{n, j}\left( {}_{c}\mathcal{EI}_{\Phi^{(p)}}\cup f^{a, j}\right) \cup \varphi_{j, nk}\right),\ \phi_{n}\right\rangle_{\mathcal{V}_{n}} \ \bmod p^{nk}
\end{align*}
We now show that the classes 
\[
z_{n, j} = [u\tau^{n}]_{*, j} \circ  \iota_{*}\left( {}_{c}\mathcal{EI}_{\Phi^{(p)}}\cup \iota^{*}f^{a, j}\right) \in H^{3}_{\mathrm{Iw}}\left(Y^{G}((u\tau^{n})^{-1}Q_{H, \infty, n}^{0}(p^{n}, D)(u\tau^{n})), \left(\mathscr{V}_{\lambda} \otimes \nu_{1}^{-j}\right)(\Zp)\right)
\]
satisfy
\begin{equation} \label{eq:classcong}
    \sum_{j = 0}^{k}(-1)^{j}\binom{k}{j}\mathrm{res}_{n}^{nk}\left(z_{n, j}\right) \cup \varphi_{j, nk} \equiv 0 \bmod p^{nk}
\end{equation}
for all $1 \leq k \leq a$.
Indeed, recalling that we defined $[u\tau^{n}]_{*, j} = p^{(a - j) n}[u\tau^{n}]_{*}$ as a map 
\[
H^{3}_{\Iw}(Y^{G}(Q_{H,\infty,n}^{0}(p^{n}, D)), \left(\mathscr{V}_{\lambda} \otimes \nu_{1}^{-j}\right)(\Zp)) \to   H^{3}_{\Iw}\left(Y^{G}((u\tau^{n})^{-1}Q_{H, \infty, n}^{0}(p^{n}, D)(u\tau^{n})), \left(\mathscr{V}_{\lambda} \otimes \nu_{1}^{-j}\right)(\Zp)\right),
\] 
we have (all modulo $p^{nk}$), letting $z := {}_{c}\mathcal{EI}_{\Phi^{(p)}}$,
\begin{align*}
    \mathrm{res}_{n}^{nk}\left(z_{n, j}\right)\cup \varphi_{j, nk} &=  \mathrm{res}_{n}^{nk}\left( [u\tau^{n}]_{*, j} \circ  \iota_{*}\left(z \cup \iota^{*}f^{a, j}\right)\right) \cup \varphi_{j, nk} \\
    &=  \mathrm{res}_{n}^{nk} \left(  [u\tau^{n}]_{*, j} \circ  \left(\iota_{*}\left(z\right) \cup f^{a, j}\right)\right)\cup \varphi_{j, nk}\\
    &=  p^{(a - j)n}\mathrm{res}_{n}^{nk} \left(  [u\tau^{n}]_{*} \circ  \left(\iota_{*}\left(z\right) \cup f^{a, j}\right)\right)\cup \varphi_{j, nk} \\
    &= \mathrm{res}_{n}^{nk} \left(  [u\tau^{n}]_{*}  \left(\iota_{*}\left(z\right)\right) \cup p^{(a - j)n}[u\tau^{n}]_{*}f^{a, j}\right)\cup \varphi_{j, nk} \\
    &= \mathrm{res}_{n}^{nk} \left(  [u\tau^{n}]_{*}  \left(\iota_{*}\left(z\right)\right) \right) \cup p^{(a - j)n}[u\tau^{n}]_{*}f^{a, j}\cup \varphi_{j, nk} \\
    &= \mathrm{res}_{n}^{nk} \left(  [u\tau^{n}]_{*}  \left(\iota_{*}\left(z\right)\right)\right) \cup p^{an}[u\tau^{n}]_{*}\left(f^{a, j} \cup \varphi_{j, nk}\right) \\
    &= \mathrm{res}_{n}^{nk} \left(  [u\tau^{n}]_{*}  \left(\iota_{*}\left(z\right)\right)\right) \cup [u\tau^{n}]_{*, 0}\left(f^{a, j} \cup \varphi_{j, nk}\right)
\end{align*}
the section $[u\tau^{n}]_{*, 0}\left(f^{a, j} \cup\varphi_{j, nk}\right)$ is induced by the integral vector $\tau^{-n} * \left(u^{-1} \cdot \nu_{1}^{j}f^{a, j}\right) \in V_{\lambda, \Zp}$ and thus by Proposition \ref{prop:cong} the classes $z_{n, j}$ satisfy the congruence \eqref{eq:classcong}. 
We claim that the following diagram commutes
\begin{equation}  \label{fig:res}
\begin{tikzcd}
    H^{3}_{\Iw}\left(Y^{G}((u\tau^{n})^{-1}Q_{H, \infty, n}^{0}(p^{n}, D)(u\tau^{n})), \left(\mathscr{V}_{\lambda} \otimes \nu_{1}^{-j}\right)(\Zp)\right) \arrow[d, "\mathrm{res}^{nk}_{n}"] \arrow[r, "\mathrm{pr}_{n}"] &  H^{3}\left(Y^{G}(\mathcal{V}_{n}(p^{n}, D)), \left(\mathscr{V}_{\lambda} \otimes \nu_{1}^{-j}\right)(\Zp)\right) \arrow[d, "\mathrm{res}_{n}^{nk}"] \\
     H^{3}_{\Iw}\left(Y^{G}((u\tau^{n})^{-1}Q_{H, \infty, n}^{0}(p^{nk}, D)(u\tau^{n})), \left(\mathscr{V}_{\lambda} \otimes \nu_{1}^{-j}\right)(\Zp)\right) \arrow[r, "\mathrm{pr}_{nk}"]&  H^{3}\left(Y^{G}(\mathcal{V}_{n}(p^{nk}, D)), \left(\mathscr{V}_{\lambda} \otimes \nu_{1}^{-j}\right)(\Zp)\right)
\end{tikzcd}
\end{equation}
Indeed, if $K_{t} \subset G(\Qp)$ is a family of open compact subgroups satisfying $\cap_{t}K_{t} = \left((u\tau^{n})^{-1}Q_{H, \infty, n}^{0}(u\tau^{n})\right)_{p}$ and $K_{t}(p^{n}, D) \subset \mathcal{V}_{n}(p^{n}, D)$ (adjoining the tame level $\mathcal{U}^{G, (p)}$ to $K_{t}$ and leaving it implicit) then the diagram  
\[
\begin{tikzcd}[column sep=6em]
    Y^{G}(K_{t}(p^{nk}, D)) \arrow[d, ""] \arrow[r, "{[1]}^{K_{t}(p^{nk}, D)}_{\mathcal{V}_{n}(p^{nk}, D)}"] & Y^{G}(\mathcal{V}_{n}(p^{nk}, D)) \arrow[d, ""] \\
    Y^{G}(K_{t}(p^{n}, D)) \arrow[r, "{[1]}^{K_{t}(p^{n}, D)}_{\mathcal{V}_{n}(p^{n}, D)}"] & Y^{G}(\mathcal{V}_{n}(p^{n}, D))
\end{tikzcd}
\]
is cartesian as $\nu_{1}$ identifies 
\[
    K_{t}(p^{n}, D)/K_{t}(p^{nk}, D) \cong \left(1 + p^{n}\Zp\right)/\left(1 + p^{nk}\Zp\right) \cong \mathcal{V}_{n}(p^{n}, D)/\mathcal{V}_{n}(p^{nk}, D).
\]
This gives a commutative diagram 
\[
\begin{tikzcd}
    H^{3}\left(Y^{G}(K_{t}(p^{n}, D)), \mathscr{V}_{\lambda} (\Zp)\right) \arrow[d, "\mathrm{res}^{nk}_{n}"] \arrow[r, "\mathrm{pr}_{n}"] &  H^{3}\left(Y^{G}(\mathcal{V}_{n}(p^{n}, D)), \mathscr{V}_{\lambda} (\Zp)\right) \arrow[d, "\mathrm{res}_{n}^{nk}"] \\
     H^{3}\left(Y^{G}(K_{t}(p^{nk}, D)), \mathscr{V}_{\lambda} (\Zp)\right) \arrow[r, "\mathrm{pr}_{n}^{(nk)}"]&  H^{3}\left(Y^{G}(\mathcal{V}_{n}(p^{nk}, D)), \mathscr{V}_{\lambda} (\Zp)\right)
\end{tikzcd}
\]
and we deduce commutativity of \eqref{fig:res} by taking the limit over $t$. Moreover, it's easy to see that we have a commutative diagram 
\[
\begin{tikzcd}
H^{3}\left(Y^{G}(\mathcal{V}_{n}(p^{n}, D)), \mathscr{V}_{\lambda}(\Zp)\right) \arrow[d, "\mathrm{res}^{nk}_{n}"] \arrow[r, "\Psi_{0}"] & H^{3}\left(Y^{G}(\mathcal{V}^{G}_{n}),  \mathscr{V}_{\lambda}(\Zp)\right)  \otimes \mathrm{Maps}\left(\left(\Z/p^{n}\Z\right)^{\times} \times \Delta_{2}, \Zp\right) \arrow[d, "1 \otimes \mathrm{res}^{nk}_{n}"] \\
H^{3}\left(Y^{G}(\mathcal{V}_{n}(p^{nk}, D)), \mathscr{V}_{\lambda}(\Zp)\right) \arrow[r, "\Psi_{0}"] & H^{3}\left(Y^{G}(\mathcal{V}^{G}_{n}),  \mathscr{V}_{\lambda}(\Zp)\right)  \otimes \mathrm{Maps}\left(\left(\Z/p^{nk}\Z\right)^{\times} \times \Delta_{2}, \Zp\right)
\end{tikzcd}
\]
From the above discussion and Lemma \ref{lem:twistcoh} we deduce that 
\[
    \Psi_{0} \circ \mathrm{pr}_{n} \circ \left(\mathrm{res}_{n}^{nk}(z_{n, j}) \cup \varphi_{j, nk} \right) =\left(1 \otimes \Tw_{-j}\right)  \mathrm{res}_{n}^{nk}(\iota_{a, j, n}(z)) \cup \mathbb{S}_{j}, 
\]
where as before $\cup \mathbb{S}_{j}$ is acting as the identity on $\mathrm{Maps}\left(\left(\Z/p^{nk}\Z\right)^{\times}, \Z/p^{nk}\Z\right)$. Setting \[
v := \vol\left(\mathcal{U}^{G, (p)}(1, D) \cap H(\A_{f}^{(p)})\right)
\]
we conclude that 
\[
     \Tw_{-j} \circ \mathrm{res}_{n}^{nk} \left(v^{-1}{}_{c}\tilde{\Xi}^{[a, j]}_{n}\right) = \left(1 \otimes e_{\eta_{2}}\right)\langle \Psi_{0} \circ \mathrm{pr}_{n} \circ \left(\mathrm{res}_{n}^{nk}(z_{n, j}) \cup \varphi_{j, nk} \right), \phi_{n}\rangle \bmod p^{nk}
\]
from which it follows immediately from \eqref{eq:classcong} that  
\begin{align*}
   C \cdot \sum_{j = 0}^{k}(-1)^{j}\binom{k}{j}\Tw_{-j}\left(\mathrm{res}^{nk}_{n}\left({}_{c}\tilde{\Xi}^{[a; j]}_{n}\left(\phi, \Phi^{(p)}\right) \right)\right)\equiv 0 \bmod p^{nk}
\end{align*}
as elements of $\mathrm{Maps}\left(\left(\Z/p^{nk}\Z\right)^{\times}, \mathcal{O}_{L}/p^{nk}\right)$, which we can equivalently write as 
\[
    \norm{\sum_{j = 0}^{k}(-1)^{j}\binom{k}{j}\Tw_{-j}\left(\mathrm{res}^{nk}_{n}\left({}_{c}\tilde{\Xi}^{[a; j]}_{n}\left(\phi, \Phi^{(p)}\right) \right)\right)} \leq Cp^{-nk}.
\]
It is immediate that the normalised classes satisfy
\[
    \norm{\alpha^{n}\sum_{j = 0}^{k}(-1)^{j}\binom{k}{j}\Tw_{-j}\left(\mathrm{res}^{nk}_{n}\left({}_{c}\Xi^{[a; j]}_{n}\left(\phi, \Phi^{(p)}\right) \right)\right)} \leq Cp^{-nk},
\]
which rearranges to 
\[
    \norm{p^{-nk}\sum_{j = 0}^{k}(-1)^{j}\binom{k}{j}\Tw_{-j}\left(\mathrm{res}^{nk}_{n}\left({}_{c}\Xi^{[a; j]}_{n}\left(\phi, \Phi^{(p)}\right) \right)\right)} \leq Cp^{hn}.
\]
\end{proof}
\subsection{Tempered distributions} \label{sec:growththm}

Set
\[
    \Theta_{c} := c^{2} - (\eta_{2}\omega_{\Pi})(c)^{-1}\,[c]^{-1} \in \mathcal{O}_{L}[[\Zp^{\times}]],
\]
As in \cite[(6.3), Proposition\ 6.13]{loefflerWilliams2021p}, the smoothed and non-smoothed classes are related by
\begin{equation} \label{eq:crelation}
    {}_{c}\Xi^{[a; j]}_{n} = \left(c^{2} - c^{-j}(\eta_{2}\omega_{\Pi})(c)^{-1}[c]^{-1}\right)\cdot\Xi^{[a; j]}_{n}.
\end{equation}

\begin{theorem} \label{thm:growth}
Suppose $\tilde{\Pi}$ has small slope, $h < a + 1$. Then:
\begin{enumerate}
    \item There is a unique ${}_{c}\mu_{\tilde{\Pi}} = {}_{c}\mu_{\tilde{\Pi}}(\phi, \Phi^{(p)}) \in \mathscr{D}_{h}(\Zp^{\times}, L)$ such that for all $n \geq 1$, all Dirichlet characters $\chi$ of conductor dividing $p^{n}$, and all $0 \leq j \leq a$,
    \[
        \int_{\Zp^{\times}}\chi(x)x^{j}{}_{c}\mu_{\tilde{\Pi}}(x) = \sum_{\gamma\in\left(\Z/p^{n}\Z\right)^{\times}}\chi(\gamma){}_{c}\Xi^{[a; j]}_{n}\left(\gamma \right).
    \]
    \item Suppose $\eta_{2}$ is chosen so that $\eta_{2}\omega_{\Pi}$ is not congruent modulo $\mathfrak{m}_{L}$ to any character of $p$-power conductor. Then for suitable $c$ the element $\Theta_{c}$ is a unit in $\Lambda \otimes_{\Zp} \mathcal{O}_{L}$, and
    \[
        \mu_{\tilde{\Pi}}\left(\phi, \Phi^{(p)}\right) := \Theta_{c}^{-1}*{}_{c}\mu_{\tilde{\Pi}}\left(\phi, \Phi^{(p)}\right) \in \mathscr{D}_{h}(\Zp^{\times}, L),
    \]
    where $*$ denotes convolution of distributions, is independent of $c$ and satisfies 
    \[
    \int\chi(x)x^{j}\mu_{\tilde{\Pi}} = \sum_{\gamma \in \left(\Z/p^{n}\Z\right)^{\times}}\chi(\gamma)\Xi^{[a; j]}_{n}(\gamma) 
    \]
    for all $\chi, j, n$ as in (1).
\end{enumerate}
\end{theorem}

\begin{proof}
 If we let $\mathscr{C}(\Zp^{\times}, \Zp)$ denote the space of continuous functions taking values in $\Zp$ and $\mathscr{C}(\Zp, L) := \mathscr{C}(\Zp^{\times}, \Zp) \hat{\otimes} L$ the space of continuous functions taking values in $L$, then we have compatible injective restriction maps 
 \[
    \mathrm{res}: \mathrm{Maps}(\left(\Z/p^{n}\Z\right)^{\times}, \Zp) \to \mathscr{C}(\Zp^{\times}, \Zp)
 \]
 for all $n \geq 1$ and twisting maps $\Tw^{\infty}_{j}: \mathscr{C}(\Zp^{\times}, \Zp) \to \mathscr{C}(\Zp^{\times}, \Zp)$ sending $z \mapsto f(z)$ to $z \mapsto z^{j}f(z)$ which are compatible modulo $p^{n}$ with the previously defined maps $\Tw_{j}: \mathrm{Maps}(\left(\Z/p^{n}\Z\right)^{\times}, \Z/p^{n}\Z) \to \mathrm{Maps}(\left(\Z/p^{n}\Z\right)^{\times}, \Z/p^{n}\Z)$. We equip $\mathscr{C}(\Zp^{\times}, \Zp)$ with the sup norm $\norm{f} = \mathrm{sup}_{z \in \Zp^{\times}}\vert f \vert$ and $\mathscr{C}(\Zp^{\times}, L)$ with its natural linear extension. We will abuse notation and write ${}_{c}\tilde{\Xi}^{[a; j]}_{n}$ for its restriction to $\mathscr{C}(\Zp^{\times}, C^{-1}\mathcal{O}_{L})$ and similarly $ {}_{c}\Xi^{[a; j]}_{n}$ for the analogous element of $\mathscr{C}(\Zp^{\times}, L)$.
 There is a unique element ${}_{c}\mu_{\tilde{\Pi}, \mathrm{alg}} \in  \mathscr{D}_{\mathrm{alg}}^{[0, a]}(\Zp^{\times}, L)$ satisfying
\[
  \int_{z + p^{n}\Zp}x^{j}{}_{c}\mu_{\tilde{\Pi}, \mathrm{alg}}(x) = {}_{c}\Xi^{[a; j]}_{n}(z) \in L
\]
for all $0 \leq j \leq a$, which is a well-defined algebraic distribution by the norm relation of Corollary \ref{cor:normcompatXi}. We have for $z \in \Zp^{\times}$, $n \geq 1$ 
\begin{align*}
    \int_{z + p^{n}\Zp}\left(\frac{x - z}{p^{n}}\right)^{k}{}_{c}\mu_{\tilde{\Pi}, \mathrm{alg}}(x) &= p^{-nk}\sum_{j = 0}^{k}\binom{k}{j}(-1)^{k - j}z^{k - j}\int_{z + p^{n}\Zp}x^{j}{}_{c}\mu_{\tilde{\Pi}, \mathrm{alg}}(x) \\
    &= z^{k}p^{-nk}\sum_{j = 0}^{k}\binom{k}{j}(-1)^{k - j}\mathrm{tw}^{\infty}_{-j}\left({}_{c}\Xi^{[a; j]}_{n}\right)(z)
\end{align*}
Proposition \ref{prop:classcong} gives that 
\[
   A:= \vert p^{-nk}\sum_{j = 0}^{k}\binom{k}{j}(-1)^{k - j}\mathrm{tw}_{-j}\left(\mathrm{res}^{nk}_{n}\left({}_{c}\Xi^{[a; j]}_{n}\right)\right)(z \bmod p^{nk})\vert \leq Cp^{hn}
\]
for the twisting map defined in Definition \ref{def:twist}. We have $\mathrm{tw}_{-j}\left({}_{c}\Xi^{[a; j]}_{n}\right)(z \bmod p^{nk}) = \tilde{z}^{-j}{}_{c}\Xi^{[a; j]}_{n}$ where $\vert z - \tilde{z} \vert \leq p^{-nk}$, so we have
\[
 B := \vert p^{-nk}\sum_{j = 0}^{k}\binom{k}{j}(-1)^{k - j}\left(\mathrm{tw}_{-j}\left(\mathrm{res}^{nk}_{n}\left({}_{c}\Xi^{[a; j]}_{n}\right)\right)(z \bmod p^{nk}) - \mathrm{tw}^{\infty}_{-j}\left({}_{c}\Xi^{[a; j]}_{n}\right)(z)\right) \vert \leq Cp^{hn}
\]
and thus 
\[
    \vert p^{-nk}\sum_{j = 0}^{k}\binom{k}{j}(-1)^{k - j}\mathrm{tw}^{\infty}_{-j}\left({}_{c}\Xi^{[a; j]}_{n}\right)(z)\vert  \leq \mathrm{max} \left\{ A, B\right\} 
    \leq Cp^{hn}.
\]
We thus have 
\[
\left\vert\int_{z + p^{n}\Zp}\left(\frac{x - z}{p^{n}}\right)^{k}{}_{c}\mu_{\tilde{\Pi}, \mathrm{alg}}(x)\right\vert \leq Cp^{hn}
\]
for all $0 \leq k \leq a$ and $n \geq 1$, thus by Proposition \ref{prop:interp} (using the small slope assumption $h < a + 1$) there is ${}_{c}\mu_{\tilde{\Pi}} \in \mathscr{D}_{h}(\Zp^{\times}, L)$ satisfying ${}_{c}\mu_{\tilde{\Pi}}\vert_{LP^{[0, a]}(\Zp^{\times}, L)} = {}_{c}\mu_{\tilde{\Pi}, \mathrm{alg}}$.
Let $\chi$ be a Dirichlet character of conductor $p^{n}$. Then we have 
\begin{align*}
    \int_{\Zp^{\times}}\chi(x)x^{j} {}_{c}\mu_{\tilde{\Pi}} (x) &= \sum_{z \in \left(\Z/p^{n}\Z\right)^{\times}}\chi(z)\int_{z + p^{n}\Zp}x^{j}{_{c}}\mu_{\tilde{\Pi}} (x) \\
    &= \sum_{z \in \left(\Z/p^{n}\Z\right)^{\times}}\chi(z){}_{c}\Xi^{[a; j]}_{n}(z),
\end{align*}
as required. Part (2) follows immediately. 
\end{proof}

\section{Values of the distribution as Rankin--Selberg integrals} \label{sec:comppair}

In this section we compute the values $\int\eta(x)x^{j}\mu_{\tilde{\Pi}}(x)$ as global Rankin--Selberg integrals and evaluate the resulting local zeta integrals. By Theorem \ref{thm:growth} these values equal $\eta(\Xi^{[a; j]}_{n})$, and by Remark \ref{rem:LWcomparison} the classes $\Xi^{[a; j]}_{n}$ are those of \cite[\S 6]{loefflerWilliams2021p}; the computation of this section is therefore essentially identical to \cite[\S\S 7--9]{loefflerWilliams2021p}, and we give a summary adapted to our notation, referring to \textit{op.\ cit.}\ for the details. We stress that these computations are insensitive to the slope: the only point in \textit{op.\ cit.}\ where ordinarity is used is the boundedness of the measure, which we have replaced by Theorem \ref{thm:growth}, and in Lemma 8.8 of \textit{op. cit.} which generalises immediately to the small slope setting, given here as Lemma \ref{lem:Epadicunit}.

\subsection{Archimedean data and the global integral}

Recall the generator $\zeta_{\infty}$ of \eqref{eq:zetainfty}. Choosing bases $\{\delta'_{1}, \ldots, \delta'_{5}\}$ of $(\mathfrak{gl}_{3}/\mathfrak{k}^{\circ}_{3, \infty})^{\vee}$ and $\{v_{\alpha}\}$ of $V^{\vee}_{\lambda}(\C)$, there are vectors $\varphi_{\infty, r, s, \alpha} \in \Pi_{\infty}$ with
\[
    \zeta_{\infty} = \sum_{r, s = 1}^{5}\sum_{\alpha}\left(\delta'_{r}\wedge\delta'_{s}\right)\otimes\varphi_{\infty, r, s, \alpha}\otimes v_{\alpha},
\]
and we set $\varphi_{r, s, \alpha} := \varphi_{\infty, r, s, \alpha}\otimes\varphi_{f} \in \Pi$. Similarly, the Eisenstein differential is described by a basis $\{\delta_{1}, \delta_{2}\}$ of $(\mathfrak{gl}_{2}/\mathfrak{k}^{\circ}_{2, \infty})^{\vee}$ and a basis $\{w^{[j]}_{\beta}\}$ of $V^{\GL_{2}}_{(0, -j)}(\C)$; we normalise $\iota^{*}(\delta'_{i}) = \delta_{i}$ for $i = 1, 2, 3$ (extending to a basis $\{\delta_{1}, \delta_{2}, \delta_{3}\}$ on $\tilde{\mathcal{H}}_{H}$) and $\iota^{*}(\delta'_{i}) = 0$ for $i = 4, 5$, and fix the Haar measure $dh_{\infty}$ corresponding to $\delta_{1}\wedge\delta_{2}\wedge\delta_{3}$, normalised so that $\vol(K^{\circ}_{H, \infty}) = 1$. Finally, the branching maps of Section \ref{sec:branch} induce pairings
\[
    \langle -, -\rangle_{a, j}: V^{\vee}_{\lambda} \times V^{\GL_{2}}_{(0, -j)} \to \mathbb{G}_{a}, \qquad (\mu, v) \mapsto \mu\left(\br^{[a; j]}\text{-image of } v\right),
\]
as in \cite[(4.4)]{loefflerWilliams2021p}, integral for our choices of lattices.

\begin{definition} \label{def:globalzeta}
Let $\chi_{1}, \chi_{2}$ be Dirichlet characters and $s_{1}, s_{2} \in \C$. For $\varphi \in \Pi$ and $\Phi \in \mathcal{S}(\A^{2}, \C)$, set
\[
    \mathcal{Z}\left(\varphi, \Phi; \chi_{1}, \chi_{2}, s_{1}, s_{2}\right) := \int_{\GL_{2}(\Q)\backslash\GL_{2}(\A)}\varphi\left(\iota(g, 1)\right)\,E^{\Phi}(g; \hat{\chi}_{2}, s_{2})\,\hat{\chi}_{1}(\det g)\,\|\det g\|^{s_{1} - \frac{1}{2}}\,dg.
\]
For $\mathfrak{Re}(s_{1}) \gg 0$ this factors as an Euler product $\mathcal{Z} = \prod_{v}Z_{v}(W_{\varphi, v}, \Phi_{v}; \hat{\chi}_{1, v}, \hat{\chi}_{2, v}, s_{1}, s_{2})$ of local Jacquet--Piatetski-Shapiro--Shalika zeta integrals \cite{JPSS83}, cf.\ \cite[(7.8), Definition\ 8.1]{loefflerWilliams2021p}.
\end{definition}

\begin{proposition} \label{prop:RSvalues}
Let $\eta$ be a Dirichlet character of $p$-power conductor, and let $n = \max(1, v_{p}(\mathrm{cond}\,\eta))$. Then for $0 \leq j \leq a$,
\[
    \int_{\Zp^{\times}}\eta(x)x^{j}\ \mu_{\tilde{\Pi}}\left(\phi_{\varphi_{f}}, \Phi^{(p)}\right)(x) = \frac{\vol(\mathcal{U}^{H}_{n, p})^{-1}}{p^{n}\alpha_{p}^{n}\,\Theta_{\Pi}}\sum_{r, s, t}\varepsilon_{rst}\sum_{\alpha, \beta}\left\langle v_{\alpha}, w^{[j]}_{\beta}\right\rangle_{a, j}\mathcal{Z}\left(u\tau^{n}\cdot\varphi_{r, s, \alpha}, \Phi; \eta, \omega_{\Pi}\eta\eta_{2}, \tfrac{1 - j}{2}, -\tfrac{j}{2}\right),
\]
where $\varepsilon_{231} = 1 = -\varepsilon_{132}$ and $\varepsilon_{rst} = 0$ otherwise, $\Phi = \Phi_{\infty}^{j + 2}\Phi^{(p)}\Phi_{p, R}$, and $\vol(\mathcal{U}^{H}_{n, p})^{-1} = p^{2R + 2n}(1 - p^{-1})(1 - p^{-2})$ with $R = \max(n, r(\tilde{\Pi}))$.
\end{proposition}
\begin{proof}
By Theorem \ref{thm:growth} the left-hand side is $e_{\eta}\left(\Xi^{[a; j]}_{n}\right)$. Via Remark \ref{rem:LWcomparison}, this is precisely the quantity computed in \cite[\S 7.1--7.3]{loefflerWilliams2021p}: one expresses the cup products as integrals of differential forms over $\tilde{Y}^{H}$, sums over the connected components against $\eta$ and $\eta_{2}$, unfolds the Eisenstein series, and performs the change of variables collapsing the central integration into the projection $R_{\omega_{\Pi}\eta\eta_{2}}$ on the Schwartz data; the volume computation is Lemma \ref{lem:inclusion}(1). We refer to \textit{op.\ cit.}\ for the details, which apply verbatim.
\end{proof}

\subsection{Local zeta integrals} \label{sec:localzeta}

The local integrals at finite places $v \neq p$ where all data is unramified take the value $1$ after normalising by the local $L$-factors, and at the remaining places away from $p$ they may be made identically $1$ by a suitable finite linear combination of test data:

\begin{theorem}[Jacquet--Piatetski-Shapiro--Shalika] \label{JPSS}
    Let $v$ be a finite place, $\pi$ a generic irreducible representation of $\GL_{3}(\Q_{v})$, and $\chi_{1}, \chi_{2}$ smooth characters of $\Q_{v}^{\times}$. The normalised integral
    \[
        \tilde{Z}\left(W, \Phi; \chi_{1}, \chi_{2}, s_{1}, s_{2}\right) := \frac{Z(W, \Phi; \chi_{1}, \chi_{2}, s_{1}, s_{2})}{L(\pi\times\chi_{1}, s_{1} + s_{2} - \frac{1}{2})\,L(\pi\times\chi_{1}\chi_{2}^{-1}, s_{1} - s_{2} + \frac{1}{2})}
    \]
    is a polynomial in $\ell^{\pm s_{1}\pm s_{2}}$. The ideal generated by these polynomials for varying $(W, \Phi)$ is the unit ideal, and there are finite collections $\{W_{i}\}, \{\Phi_{i}\}$ defined over $\overline{\Q}$ with $\sum_{i}\tilde{Z}(W_{i}, \Phi_{i}; \chi_{1}, \chi_{2}, s_{1}, s_{2}) = 1$ identically. If all data is unramified and normalised spherical, then $\tilde{Z} = 1$.
\end{theorem}
\begin{proof}
    See \cite{JPSS83} and \cite[Theorem\ 8.2]{loefflerWilliams2021p}.
\end{proof}

At $v = p$ we use the following evaluation, which follows from the proof of \cite[Theorem\ 8.7]{loefflerWilliams2021p}. We emphasise that the proof in \textit{op.\ cit.}\ is a purely local computation with the refined newvector $W^{\alpha}$ of Proposition \ref{prop:newvector} and makes no use of ordinarity; it is valid for any regular unramified $P_{1}$-refinement of finite slope. In \textit{op. cit.} the authors compute the value of the local zeta integral at a fixed point of $\mathrm{Crit}^{-}_{p}(\Pi)$; due to the possibility of poles (cf. Remark \ref{rem:pole}) we leave this as an equality of meromorphic functions, the potential poles cancelling in the final interpolation formula.

\begin{theorem} \label{thm:localzetap}
    Let $\eta$ be of conductor $p^{n_{1}}$, $n = \max(1, n_{1})$, $R = \max(n, r)$, and take the test data $W = u\tau^{n}\cdot W^{\alpha}$ and $\Phi_{R} = \mathrm{ch}\left((0, 1) + p^{R}\Zp^{2}\right)$. Then we have an equality of meromorphic functions
    \[
        Z\left(W, \Phi_{R}; \hat{\eta}_{p}, \hat{\omega}_{\Pi, p}\hat{\eta}_{p}\hat{\eta}_{2, p}, \tfrac{1 + s}{2}, \tfrac{s}{2}\right) = \frac{\alpha_{p}^{n}}{p^{2R + n}(1 - p^{-1})(1 - p^{-2})}\cdot e_{p}\left(\Pi_{p}\times\eta_{p}, s\right)\cdot\frac{\mathcal{E}_{0}\cdot L\left(\Pi_{p}\times\omega_{\Pi, p}^{-1}\eta_{2, p}^{-1}, 1\right)}{\varepsilon\left(\Pi_{p}\times\omega_{\Pi, p}^{-1}\eta_{2, p}^{-1}, 1\right)},
    \]
    where, with $\mathcal{E}_{t} := L(\sigma'_{p}\times\sigma_{p}\eta_{2, p}, t)/L(\Pi^{\vee}_{p}\times\omega_{\Pi, p}\eta_{2, p}, t)$ for $t \in \C$, the factor $\mathcal{E}_{0} = \lim_{t \to 0}\mathcal{E}_{t}$ is as in the following lemma.
\end{theorem}
\begin{proof}
In \textit{op. cit.} the authors prove the formula by computing a zeta integral, denoted $Y(W, \Phi_{R}; \chi_{1}, \chi_{2}, s_{1}, s_{2})$, which differing from $Z\left(W, \Phi_{R}; \hat{\eta}_{p}, \hat{\omega}_{\Pi, p}\hat{\eta}_{p}\hat{\eta}_{2, p}, s_{1}, s_{2}\right)$ by a $\gamma$-factor. The computation shows that $Y(W, \Phi_{R}; \chi_{1}, \chi_{2}, s_{1}, s_{2})$ is given by,
\[
    Y(W, \Phi_{R}; \hat{\eta}_{p}, \hat{\omega}_{\Pi, p}\hat{\eta}_{p}\hat{\eta}_{2, p}, s_{1}, s_{2}) = \frac{1}{p^{2R + n}(1 - p^{-1})(1 - p^{-2})}  L(\sigma_{p}' \times \sigma_{p}\eta_{2, p}, s_{2} - s_{1} + \tfrac{1}{2}) \cdot \begin{cases}
        G(\eta^{-1}_{p})p^{n(s_{1} + s_{2} - \tfrac{3}{2})}  &\text{if $\eta_{p} \neq 1$} \\
        \alpha_{p} \frac{1 - \alpha_{p}^{-1}p^{-(\tfrac{3}{2}-s_{1}-s_{2})}}{1 - \alpha_{p}p^{-(s_{1} + s_{2} - \tfrac{1}{2})}} &\text{otherwise}
    \end{cases}
\]
and evaluating at $s_{1} = \tfrac{1 - j}{2}, s_{2} = -\tfrac{j}{2}$ gives the statement of Theorem 8.7 of \textit{op.cit.} If we instead evaluate at $s_{1} = \tfrac{1 + s}{2}, s_{2} = \tfrac{s}{2}$ for $s \in \C$ then we get the stated equality of meromorphic functions in $s$. 
\end{proof}
\begin{lemma} \label{lem:Epadicunit}
    The factor $\mathcal{E}_{0}$ equals $1$ if the induction of $\sigma_{p}\times\sigma'_{p}$ to $\GL_{3}(\Qp)$ is reducible, and if the induction is irreducible it equals $1 - \omega_{\Pi}(p)\eta_{2}(p)\alpha_{p}^{-1}$ when $\omega_{\Pi}$ is unramified at $p$ and $1$ otherwise. If $\tilde{\Pi}$ has small slope then $\mathcal{E}_{0}$ is a $p$-adic unit.
\end{lemma}
\begin{proof}
    The identification of $\mathcal{E}_{0}$ is \cite[Lemma\ 8.8, Rem.\ 8.9]{loefflerWilliams2021p}, using $L(\Pi^{\vee}_{p}\times\omega_{\Pi, p}, t) = L(\wedge^{2}\Pi_{p}, t)$. That $\mathcal{E}_{0}$ is a $p$-adic unit is immediate in the reducible case, in the irreducible case, small slope gives $v_{p}(\alpha_{p}^{-1}) > 0$ (Definition \ref{def:smallslope}), while $\omega_{\Pi}(p)\eta_{2}(p)$ is a root of unity, so $\equiv 1 \bmod \mathfrak{m}_{L}$. 
\end{proof}

At $v = \infty$, by \cite[Lemma\ 8.10]{loefflerWilliams2021p} (following \cite{JPSS83} and \cite[\S 1]{mahnkopf00}) the local integral is a polynomial multiple of the expected pair of archimedean $L$-factors, and combining the sums over $(r, s, t, \alpha, \beta)$ into a single archimedean quantity we set
\[
    \tilde{e}_{\infty}(\zeta_{\infty}\times\eta_{\infty}, -j) := \sum_{r, s, t}\varepsilon_{rst}\sum_{\alpha, \beta}\left\langle v_{\alpha}, w^{[j]}_{\beta}\right\rangle_{a, j}Z_{\infty}\left(W_{\infty, r, s, \alpha}, \Phi^{j + 2}_{\infty}; \hat{\eta}_{\infty}, \hat{\omega}_{\Pi, \infty}\hat{\eta}_{\infty}, \tfrac{1 - j}{2}, -\tfrac{j}{2}\right).
\]
For $(-j, \eta) \in \mathrm{Crit}_{p}(\Pi)$ this quantity is non-zero by the non-vanishing results of Kasten--Schmidt \cite{kastenschmidt13} and Sun \cite{sun17}, and depends only on $\Pi_{\infty}$, $\zeta_{\infty}$ and $j$.

\subsection{Interpolation, and the factor at infinity} \label{sec:interpolation}

Combining Proposition \ref{prop:RSvalues} with Theorems \ref{JPSS} and \ref{thm:localzetap} (for test data trivialising the away-from-$p$ integrals, as fixed in Definition \ref{not:testdata} below), the volume factors and powers of $\alpha_{p}$ cancel and we obtain: for all $(-j, \eta) \in \Crit^{-}_{p}(\Pi)$,
\begin{equation} \label{eq:interpolationraw}
    \int_{\Zp^{\times}}\eta(x)x^{j}\mu_{\tilde{\Pi}}(x) = \lim_{s \to -j}\tilde{e}_{\infty}\left(\zeta_{\infty}\times\eta_{\infty}, -j\right)\,e_{p}\left(\Pi_{p}\times\eta_{p}, s\right)\cdot\frac{L^{(p)}\left(\Pi\times\eta, s\right)}{\Omega^{-}_{\Pi}},
\end{equation}
where
\begin{equation} \label{eq:period}
    \Omega^{-}_{\Pi} := \frac{\Theta_{\Pi}\cdot\varepsilon\left(\Pi_{p}\times(\omega_{\Pi, p}\eta_{2, p})^{-1}, 1\right)}{\mathcal{E}_{0}\cdot L\left(\Pi\times(\omega_{\Pi}\eta_{2})^{-1}, 1\right)} \in \C^{\times}.
\end{equation}
Indeed, this quantity is well defined and non-zero since the $L$-value is non-zero by \cite[Theorem 1.3]{jacquetShalika}, $\mathcal{E}_{0} \neq 0$ by Lemma \ref{lem:Epadicunit}. 
\begin{remark}
The limit in the interpolation formula accounts for the pole occurring in the `exceptional Steinberg' case (as discussed in Remark \ref{rem:pole}) when $\Pi_{p} = \mathrm{St}_{3}$, $\eta=1$ and $j=1$. In this case $L(\Pi_{p}, s) =( 1 - p^{-s}\alpha_{p})^{-1}$, so we have  
\[
    \lim_{s \to -j}e_{p}(\Pi_{p}, s)L^{(p)}(\Pi, s) = \lim_{s \to -j}\frac{1 - p^{s-  1}\alpha_{p}^{-1}}{1 - p^{-s}\alpha_{p}}(1 - p^{-s}\alpha_{p})L(\Pi, s) = (1 - p^{-j - 1}\alpha_{p}^{-1})L(\Pi, -j)
\]
which is well-defined by criticality. 
\end{remark}
It remains to identify $\tilde{e}_{\infty}$ with the Coates--Perrin-Riou factor $e_{\infty}$ of Section \ref{sec:einfty}. The quantity $\tilde{e}_{\infty}(\zeta_{\infty}\times\eta_{\infty}, -j)$ is precisely the same archimedean factor as in \cite{loefflerWilliams2021p}, depending only on $(\Pi_{\infty}, \zeta_{\infty}, j)$, and in particular is insensitive to $p$ and to the slope. In \cite[\S 9.3--9.5]{loefflerWilliams2021p} it is shown, by comparison with symmetric square $p$-adic $L$-functions at the prime $3$ for auxiliary self-dual representations sharing the given $\Pi_{\infty}$, that the $\tilde{e}_{\infty}$ satisfy the same $j$-compatibility as the $e_{\infty}$. We may therefore normalise $\zeta_{\infty}$ (rescaling $\Theta_{\Pi}$ and $\Omega^{-}_{\Pi}$ accordingly) so that
\begin{equation} \label{eq:einftyid}
    \tilde{e}_{\infty}\left(\zeta_{\infty}\times\eta_{\infty}, -j\right) = e_{\infty}\left(\Pi_{\infty}\times\eta_{\infty}, -j\right) \qquad \text{for all } 0 \leq j \leq a.
\end{equation}

\section{The $p$-adic $L$-function} \label{sec:padicL}

\begin{definition}\label{not:testdata}
Fix the following test data:
\begin{itemize}
    \item at each $v \nmid p\infty$: a finite set $I_{v}$ and collections $W_{v, i} \in \mathcal{W}(\Pi_{v}, E)$, $\Phi_{v, i} \in \mathcal{S}(\Q_{v}^{2}, \Z)$, $i \in I_{v}$, with $\sum_{i}\tilde{Z}(W_{v, i}, \Phi_{v, i}; 1, \hat{\omega}_{\Pi, v}\hat{\eta}_{2, v}, s_{1}, s_{2}) = 1$ identically (Theorem \ref{JPSS}); when $\Pi_{v}$, $\omega_{\Pi, v}$ and $\eta_{2, v}$ are unramified we take $|I_{v}| = 1$ with normalised spherical data (so $I := \prod_{v}I_{v}$ is finite). Note that the identity $\sum_{i}\tilde{Z} = 1$ persists under twisting $\chi_{1}, \chi_{2}$ by unramified characters, so applies with $(\chi_{1}, \chi_{2}) = (\hat{\eta}_{v}, \hat{\omega}_{\Pi, v}\hat{\eta}_{v}\hat{\eta}_{2, v})$ for $\eta$ of $p$-power conductor;
    \item at $p$: $\varphi_{p}$ with $\mathcal{W}(\varphi_{p}) = W^{\alpha}$ (Proposition \ref{prop:newvector});
    \item this yields vectors $\varphi_{f, i} \in \Pi_{f}$ and $\Phi_{f, i}^{(p)} \in \mathcal{S}(\A_{f}^{(p), 2}, \Z)$ for $i \in I$; fix a prime-to-$p$ level $\mathcal{U}^{G, (p)}$ fixing all of them, and normalise $\Theta_{\Pi}$ so that all $\phi_{\varphi_{f, i}}$ are integral as in Section \ref{sec:autoclasses}.
\end{itemize}
\end{definition}

\begin{definition} \label{def:padicL}
Let $\tilde{\Pi}$ be a regular, unramified $P_{1}$-refined RACAR of weight $(a, 0, -a)$, and slope $h < a + 1$ and let $\iota$ denote the involution of $\mathscr{D}_{h}(\Zp^{\times}, L)$ induced by $x \mapsto x^{-1}$ on $\Zp^{\times}$. The \emph{left-half $p$-adic $L$-function} of $\tilde{\Pi}$ is
\[
    L^{-}_{p}(\tilde{\Pi}) := \iota\left(\sum_{i \in I}\mu_{\tilde{\Pi}}\left(\phi_{\varphi_{f, i}}, \Phi^{(p)}_{f, i}\right)\right) \in \mathscr{D}_{h}(\Zp^{\times}, L).
\]
\end{definition}

\begin{remark}
Following \cite{loefflerWilliams2021p}, the involution $\iota$ is included so that the interpolation property holds at the integers $\{0, -1, \ldots, -a\}$ rather than $\{0, 1, \ldots, a\}$, matching the critical range of the complex $L$-function. Note that \[
\int\eta^{-1}(x)x^{-j}\left(\iota\mu\right)(x) = \int\eta(x)x^{j}\mu(x).
\]
\end{remark}

\begin{theorem} \label{thm:main}
The distribution $L^{-}_{p}(\tilde{\Pi}) \in \mathscr{D}_{h}(\Zp^{\times}, L)$ satisfies: for all $(-j, \eta) \in \Crit^{-}_{p}(\Pi)$,
\[
    \int_{\Zp^{\times}}\eta^{-1}(x)x^{-j}\ L^{-}_{p}(\tilde{\Pi})(x) = e_{\infty}\left(\Pi_{\infty}\times\eta_{\infty}, -j\right)e_{p}\left(\Pi_{p}\times\eta_{p}, -j\right)\cdot\frac{L^{(p)}\left(\Pi\times\eta, -j\right)}{\Omega^{-}_{\Pi}}.
\]
\end{theorem}
\begin{proof}
Combine \eqref{eq:interpolationraw} (summed over the test data $i \in I$, which trivialises the local integrals away from $p\infty$) with \eqref{eq:einftyid} and the preceding remark.
\end{proof}
We take this opportunity to remind the reader that while the above definition and theorem are stated for an unramified $P_{1}$-refinement, every $P_{1}$-refinement is unramified up to twist by a Dirichlet character of $p$-power conductor and thus we obtain $p$-adic $L$-functions for all regular $P_{1}$-refinements. 
\begin{remark}
The interpolation formula determines the values of $L^{-}_{p}(\tilde{\Pi})$ only at $(j, \eta)$ with $(-1)^{j} = \omega_{\Pi}\eta(-1)$. As in \cite[Conj.\ 1.2]{loefflerWilliams2021p}, one can further show the vanishing $\int\eta^{-1}(x)x^{-j}L^{-}_{p}(\tilde{\Pi}) = 0$ for characters of the opposite sign: the computation of Proposition \ref{prop:RSvalues} applies to every $(j, \eta)$; since $\Phi_{\infty}^{j + 2}(-x, -y) = (-1)^{j}\Phi_{\infty}^{j + 2}(x, y)$, the local Godement--Siegel section at $\infty$  
\[
    f^{\Phi_{\infty}}(g; \hat{\chi}_{\infty}, s) := \vert \det g\vert^{s}\int_{\R^{\times}}\Phi_{\infty}((0, t)g)\hat{\chi}_{\infty}(t)\vert t \vert^{2s} d^{\times}t
\]
with $\chi = \omega_{\Pi}\eta\eta_{2}$ satisfies 
\[f^{\Phi_{\infty}}(g; \hat{\chi}_{\infty}, s) = (-1)^{j}\omega_{\Pi}\eta\eta_{2}(-1)f^{\Phi_{\infty}}(g; \hat{\chi}_{\infty}, s) = (-1)^{j}\omega_{\Pi}\eta(-1)f^{\Phi_{\infty}}(g; \hat{\chi}_{\infty}, s),
\]
the last equality using that $\eta_{2}$ is even, so must vanish when $(-1)^{j} \neq \omega_{\Pi}\eta(-1)$ whence the Eisenstein series and global integral occurring in Proposition \ref{prop:RSvalues} vanish also.
\end{remark}

\begin{proposition} \label{prop:unique}
$L^{-}_{p}(\tilde{\Pi})$ is the unique element of $\mathscr{D}_{h}(\Zp^{\times}, L)$ satisfying the interpolation property of Theorem \ref{thm:main} together with the vanishing at characters of the opposite sign.
\end{proposition}
\begin{proof}
The difference of two such elements lies in some $\mathscr{D}_{h'}(\Zp^{\times}, L)$ with $h' < a + 1$ and their image under $\iota$ integrates to zero against $\eta(x)x^{j}$ for \emph{all} finite-order $\eta$ of $p$-power conductor and $0 \leq j \leq a$. The claim then follows from Proposition \ref{prop:interp}(1) (and $\iota$-invariance of $\mathscr{D}_{h}(\Zp^{\times}, L)$).
\end{proof}

\begin{remark} \label{rem:recoverLW}
If $\tilde{\Pi}$ is ordinary ($h = 0$), then $L^{-}_{p}(\tilde{\Pi})$ is a bounded measure and, by Proposition \ref{prop:unique} applied with $h' = 0$, coincides with the $p$-adic $L$-function of \cite[Definition\ 9.3]{loefflerWilliams2021p} -- including the choice of periods, provided the same normalisations of $\zeta_{\infty}$ and test data are made.
\end{remark}

\section{Duality and the right half of the critical strip} \label{sec:duality}

Finally we treat $P_{2}$-refinements and the critical values $\Crit^{+}_{p}(\Pi)$, by reduction to the $P_{1}$-theory for the dual representation, following the analogous treatment of \cite[\S 10]{loefflerWilliams2021p}. Note that if $\Pi$ is a RACAR of weight $(a, 0, -a)$ then so is $\Pi^{\vee}$, with $\omega_{\Pi^{\vee}} = \omega_{\Pi}^{-1}$ and $L(\Pi^{\vee}\times\eta^{-1}, s) = L((\Pi\times\eta)^{\vee}, s)$.

\begin{lemma} \label{lem:slopeduality}
    $\sigma_{p}\times\sigma'_{p}$ is a regular $P_{2}$-refinement of $\Pi_{p}$ if and only if $(\sigma'_{p})^{\vee}\times\sigma_{p}^{\vee}$ is a regular $P_{1}$-refinement of $\Pi^{\vee}_{p}$, and the two refinements have the \emph{same} slope. In particular, small slope $P_{2}$-refinements of $\Pi$ correspond to small slope $P_{1}$-refinements of $\Pi^{\vee}$. 
\end{lemma}
\begin{proof}
    
By combining the equality $w_0 \bar{P_2} w_0^{-1}=P_1$ where $w_0$ is the longest Weyl element $w_0=\left(\begin{smallmatrix}&&1\\&1&\\1&&\end{smallmatrix}\right)=w_0^{-1}$ 
    with  Casselman duality $J_{P_i}(\Pi_p^\vee)\cong J_{\bar{P_i}}(\Pi_p)^\vee,$
    we get an isomorphism $J_{P_1}(\Pi_p^\vee)\cong J_{P_2}(\Pi_p)^\vee \circ \operatorname{Ad}(w_0).$ Under this isomorphism, a refinement $\sigma_{p} \times \sigma_{p}'$ of $J_{P_2}(\Pi_p)$ corresponds to a refinement $\sigma_{p}'^\vee\times \sigma_{p}^\vee$. Since $J_{P_2}(\Pi_p)$ has finite length and $\sigma_{p} \times \sigma'_{p}$ is regular (and is thus a direct summand), we can dualise to get that $\sigma_{p}'^\vee \times \sigma_{p}^\vee$ is a direct summand of $J_{P_1}(\Pi_p^\vee)$ and thus a regular refinement. Then 
    \[
        \mathrm{slope}\left(\sigma_{p}^{'\vee} \times \sigma_{p}^{\vee}\right) = v_{p}(\sigma'^{\vee}(p)) = -v_{p}(\sigma_{p}'(p)) = v_{p}(\omega_{\sigma_{p}}(p)) = \mathrm{slope}\left(\sigma_{p} \times \sigma_{p}'\right),
    \]
    where we use that the central character of $\Pi_{p}$ is finite order. 
    The converse direction follows similarly by $(\Pi_p^\vee,P_1)$ in place of $(\Pi_p,P_2)$ and using $w_0 \bar{P_1}w_0^{-1}=P_2$.

\end{proof}

\begin{definition} \label{def:plus}
Let $\tilde{\Pi} = (\Pi, \sigma_{p}\times\sigma'_{p})$ be a regular small slope $P_{2}$-refined RACAR of weight $(a, 0, -a)$, with $\sigma'_{p}$ unramified, and let $\tilde{\Pi}^{\vee}$ denote the corresponding $P_{1}$-refinement of $\Pi^{\vee}$ from Lemma \ref{lem:slopeduality}. For $C \in \Zp^{\times}$ let $[C] \in \mathscr{D}_{0}\left(\Zp^{\times}, \mathcal{O}\right)$ denote the Dirac measure at $C$, and recall the operators $\iota$ and $\mathrm{tw}_{1}$ of Section \ref{sec:analytic} ($\mathrm{tw}_{1}[C] = C[C]$), which preserve $\mathscr{D}_{h}(\Zp^{\times}, L)$. The \emph{right-half $p$-adic $L$-function} of $\tilde{\Pi}$ is
\[
    L^{+}_{p}(\tilde{\Pi}) := \frac{1}{\varepsilon^{(p)}\left(\Pi^{\vee, (p)}, 0\right)}\cdot\left(\iota\circ\mathrm{tw}_{1}\right)\left(\left[N^{(p)}_{\Pi}\right]* L^{-}_{p}\left(\tilde{\Pi}^{\vee}\right)\right) \in \mathscr{D}_{h}(\Zp^{\times}, L),
\]
where $\varepsilon^{(p)}(\Pi^{\vee, (p)}, 0) = \prod_{\ell \neq p}\varepsilon_{\ell}(\Pi^{\vee}_{\ell}, 0)$ and $N^{(p)}_{\Pi}$ is the prime-to-$p$ conductor, and $*$ is convolution of distributions. We also set $\Omega^{+}_{\Pi} := \Omega^{-}_{\Pi^{\vee}}$.
\end{definition}

\begin{proposition} \label{prop:plus}
For all $(j + 1, \eta) \in \Crit^{+}_{p}(\Pi)$ we have
\[
    \int_{\Zp^{\times}}\eta(x)x^{j + 1}L^{+}_{p}(\tilde{\Pi})(x) = e_{\infty}\left(\Pi_{\infty}\times\eta^{-1}_{\infty}, j + 1\right)e_{p}\left(\Pi_{p}\times\eta^{-1}_{p}, j + 1\right)\cdot\frac{L^{(p)}\left(\Pi\times\eta^{-1}, j + 1\right)}{\Omega^{+}_{\Pi}}.
\]
\end{proposition}
\begin{proof}
Identical to \cite[Proposition\ 10.2]{loefflerWilliams2021p}, which we recall. Unwinding Definition \ref{def:plus} and applying Theorem \ref{thm:main} for $\tilde{\Pi}^{\vee}$,
\[
\begin{aligned}
    \int\eta(x)x^{j + 1}L^{+}_{p}(\tilde{\Pi}) &= \left(N^{(p)}_{\Pi}\right)^{-j}\eta\left(N^{(p)}_{\Pi}\right)^{-1}\varepsilon^{(p)}\left(\Pi^{\vee, (p)}, 0\right)^{-1} \\ & \qquad \times e_{\infty}\left(\Pi^{\vee}_{\infty}\times\eta_{\infty}, -j\right)e_{p}\left(\Pi^{\vee}_{p}\times\eta_{p}, -j\right)\frac{L^{(p)}\left(\Pi^{\vee}\times\eta, -j\right)}{\Omega^{-}_{\Pi^{\vee}}}.
\end{aligned}
\]
By the identity of Coates \cite[(20)]{coates89} cf.\ Section \ref{sec:ep} and Lemma 2.18 of \cite{loefflerWilliams2021p}, the right-hand side equals
\[
\begin{aligned}
    &\left(N^{(p)}_{\Pi}\right)^{-j}\eta\left(N^{(p)}_{\Pi}\right)^{-1}\varepsilon^{(p)}\left(\Pi^{\vee, (p)}, 0\right)^{-1}\left(\prod_{\ell \neq p}\varepsilon_{\ell}\left(\Pi^{\vee}_{\ell}\times\eta_{\ell}, -j\right)\right)\\ 
    &\qquad{} \times e_{\infty}\left(\Pi_{\infty}\times\eta^{-1}_{\infty}, j + 1\right)e_{p}\left(\Pi_{p}\times\eta^{-1}_{p}, j + 1\right)\frac{L^{(p)}\left(\Pi\times\eta^{-1}, j + 1\right)}{\Omega^{+}_{\Pi}},
\end{aligned}
\]
and since $\eta$ is unramified at $\ell \neq p$, Tate's twisting formula for local $\varepsilon$-factors gives $\varepsilon_{\ell}(\Pi^{\vee}_{\ell}\times\hat{\eta}_{\ell}, -j) = \varepsilon_{\ell}(\Pi^{\vee}_{\ell}, 0)\cdot\ell^{jc_{\ell}}\hat{\eta}_{\ell}(\ell^{c_{\ell}})$, with $c_{\ell}$ the exponent of $\ell$ in $N_{\Pi}$. Taking the product over $\ell \neq p$ and using the conventions of Section \ref{sec:characters} to convert $\hat{\eta}^{(p)}(N^{(p)}_{\Pi}) = \eta(N^{(p)}_{\Pi})$, all factors to the left of $e_{\infty}$ cancel and the proposition follows.
\end{proof}

This completes the proofs of Theorems A and B of the introduction: Theorem A is Theorem \ref{thm:main} together with Proposition \ref{prop:unique} and Remark \ref{rem:recoverLW}, and Theorem B is Proposition \ref{prop:plus}.

\bibliographystyle{alpha}
\bibliography{GL3}

\end{document}